\documentclass{conm-p-l}

\usepackage{amssymb}

\usepackage{graphicx}

\usepackage{amsmath, amsfonts, mathtools, tikz,lipsum}

\usepackage{etoolbox}
\newcommand{\zerodisplayskips}{%
  \setlength{\abovedisplayskip}{1pt}%
  \setlength{\belowdisplayskip}{1pt}%
  \setlength{\abovedisplayshortskip}{1pt}%
  \setlength{\belowdisplayshortskip}{1pt}}
\appto{\normalsize}{\zerodisplayskips}
\appto{\small}{\zerodisplayskips}
\appto{\footnotesize}{\zerodisplayskips}

\newtheorem{theorem}{Theorem}[section]

\theoremstyle{definition}

\theoremstyle{remark}

\theoremstyle{corollary}

\numberwithin{equation}{section}

\newcommand{\Bfamsix}{\{B^{6,l}\}_{l \geq 1}}
\newcommand{\veps}{\varepsilon}
\newcommand{\vphi}{\varphi}
\newcommand{\mZ}{\mathbb{Z}}

\allowdisplaybreaks
\def\ge{\mathfrak{g}}

\def\al{\alpha}

\def\cV{\mathcal{V}}
\def\cB{\mathcal{B}}
\def\C{\mathbb{C}}
\def\Z{\mathbb{Z}}

\def\L{\Lambda}
\def\e{\tilde{e}}
\def\f{\tilde{f}}

\begin{document}

\title[Ultra-discretization of $D_6^{(1)}$- Geometric Crystal]{Ultra-Discretization of $D_6^{(1)}$- Geometric Crystal at the spin node}


\author{Kailash C. Misra}
\address{Department of Mathematics, North Carolina State University,  Raleigh,  
NC 27695-8205, USA}
\email{misra@ncsu.edu}
\thanks{KCM is partially supported by the Simons Foundation Grant \#636482.}
\author{Suchada Pongprasert}
\address{Department of Mathematics,
Srinakharinwirot University, Bangkok, Thailand 10310}
\email{suchadapo@g.swu.ac.th }



\dedicatory{To Mirko Primc on his 70th Birthday}


\begin{abstract}
Let $\ge$ be an affine Lie algebra with index set $I = \{0, 1, 2, \cdots , n\}$. It is conjectured in \cite{KNO} that  for each Dynkin node $k \in I \setminus \{0\}$ the affine Lie algebra $\ge$ has a positive geometric crystal whose ultra-discretization is isomorphic to the limit of a coherent family of perfect crystals for the Langland dual $\ge ^L$.  In this paper we show that at the spin node $k=6$, the family of perfect crystals given in \cite{KMN2} form a coherent family and show that its limit $B^{6,\infty}$ is isomorphic to the ultra-discretization of the positive geometric crystal we constructed in \cite{MP} for the affine Lie algebra $D_6^{(1)}$ which proves the conjecture in this case.
\end{abstract}

\maketitle

\section{Introduction}
\setcounter{equation}{0}
Let $\ge$ denote a simply-laced affine Lie algebra with Cartan datum $(A, \{\al_i\}_{i \in I},$ \\ $ \{\al^\vee_i\}_{\i\in I})$, where $A= (a_{ij})_{i,j \in I}, I = \{0, 1, \cdots , n\}$ is a symmetric affine Cartan matrix and $U_q(\ge)$ denote the corresponding quantum affine algebra. Let $P= \Z \Lambda_0 \oplus \Z \Lambda_1\oplus 
\cdots \oplus \Z \Lambda_n \oplus \Z\delta$ and  
$P^\vee = \Z \al^\vee_0 \oplus \Z \al^\vee_1 \oplus \cdots \oplus 
\Z \al^\vee_n \oplus \Z d $ denote the affine weight lattice 
and the dual affine weight lattice  where $\delta$ and $d$ denote the simple imaginary root and the degree derivation respectively.
For a dominant weight  $\lambda \in P^+ = \{\mu \in P \mid \mu (h_i) 
\geq 0 \quad  {\rm for \ \  all} \quad i \in I \}$ of level 
$l = \lambda ({\bf c})$ (${\bf c}$ is the canonical central element), 
 $(L(\lambda), B(\lambda))$ denote the crystal base \cite{Kas1, Kas2, Lu}
for the integrable highest weight $U_q(\ge)$-module $V(\lambda)$. To give explicit realization of the crystal $B(\lambda)$, 
the notion of affine crystal and perfect crystal has been introduced 
in \cite{KMN1}. In particular, it is shown in \cite{KMN1, KMN2} that 
the affine crystal $B(\lambda)$ for the level $l \in \Z_{\geq 1}$ 
integrable highest weight $U_q(\ge)$-module $V(\lambda)$ can be 
realized as the semi-infinite tensor product $\cdots \otimes B^l \otimes
B^l \otimes B^l$, 
where $B^l$ is a perfect crystal of level $l$. This is known as the path realization of the crystal $B(\lambda)$.
Subsequently it is noticed in \cite{KKM} that one needs 
a coherent family of perfect crystals $\{B^l\}_{l \geq 1}$ 
in order to give a path realization of the crystal $B^{\infty}$ of $U_q^-(\ge)$. In particular, 
the crystal $B(\infty)$ can be realized as the
semi-infinite tensor product $\cdots \otimes B^{\infty} \otimes 
B^{\infty} \otimes B^{\infty}$ where $B^{\infty}$ is the limit of 
the coherent family of perfect crystals $\{B^l\}_{l \geq 1}$. 

On the other hand the geometric crystal \cite{BK, N} for the simply-laced affine Lie algebra $\ge$ is a quadruple $\cV(\ge)=(X, \{e_i\}_{i \in I}, \{\gamma_i\}_{i \in I},$ 
$\{\veps_i\}_{i\in I})$, 
where $X$ is an ind-variety,  $e_i:\C^\times\times
X\longrightarrow X$ $((c,x)\mapsto e^c_i(x))$
are rational $\C^\times$-actions and  
$\gamma_i,\veps_i:X\longrightarrow 
\C$ $(i\in I)$ are rational functions satisfying the following:
\begin{enumerate}
\item $\{1\}\times X\subset {\rm dom}(e_i) \;
{\rm for} \; {\rm any} \; i\in I,$
\item $\gamma_j(e^c_i(x))=c^{a_{ij}}\gamma_j(x),$
\item $\begin{cases} 
\begin{array}{lll}
&\hspace{-20pt} \quad e^{c_1}_{i}e^{c_2}_{j}
=e^{c_2}_{j}e^{c_1}_{i}&
{\rm if }\,\,a_{ij}=a_{ji}=0,\\
&\hspace{-20pt} \quad e^{c_1}_{i}e^{c_1c_2}_{j}e^{c_2}_{i}
=e^{c_2}_{j}e^{c_1c_2}_{i}e^{c_1}_{j}&
{\rm if }\,\,a_{ij}=a_{ji}=-1,\\
\end{array}
\end{cases}$
\item $\veps_i(e_i^c(x))=c^{-1}\veps_i(x)$ and $\veps_i(e_j^c(x))=\veps_i(x) \qquad {\rm if }\,
a_{i,j}=a_{j,i}=0.$
\end{enumerate}
The geometric crystal $\cV(\ge)$ is said to be positive if it has a 
positive structure \cite{BK, KNO, N}. 
Roughly speaking this means that each of the rational maps 
$e^c_i$, $\veps_i$  and $\gamma_i$ are
ratios of polynomial functions with positive coefficients. 

A remarkable relation between positive geometric crystals 
and algebraic crystals is the ultra-discretization functor $\mathcal
{UD}$ 
between them \cite{BK}. Applying this functor, positive rational 
functions are transfered to piecewise linear 
functions by the simple correspondence:
$$
x \times y \longmapsto x+y, \qquad \frac{x}{y} \longmapsto x - y, 
\qquad x + y \longmapsto {\rm max}\{x, y\}.
$$

It was conjectured in \cite{KNO} that for each affine Lie algebra $\ge$ and 
each Dynkin index $i \in I \setminus \{0\}$, there exists a positive geometric crystal
$\cV(\ge)=(X, \{e_i\}_{i \in I}, \{\gamma_i\}_{i \in I}, $ $
\{\veps_i\}_{i\in I})$ whose ultra-discretization $\mathcal{UD}(\cV)$ is isomorphic 
to the limit $B^{\infty}$ of a coherent family of perfect crystals for the Langlands dual $\ge^L$.
So far this conjecture has been proved for  $( k= 1; \ge = A_n^{(1)}, 
B_n^{(1)}, C_n^{(1)}, D_n^{(1)}, A_{2n-1}^{(2)}, A_{2n}^{(2)},
D_{n+1}^{(2)}$)  \cite{KNO}; ($k \geq 2; A_n^{(1)}$) \cite{MN1, MN2}; ($k = 1; G_2^{(1)}$) \cite{N2, N3}; $(k = 1; D_4^{(3)}$) \cite{IN, IMN}; ($k=5, D_5^{(1)}$) \cite{IMP}. In \cite{MP} we construct a positive geometric crystal for the affine Lie algebra $D_6^{(1)}$ at the Dynkin spin node $k= 6$. In this paper for $l \in \mathbb{Z}_{\geq 1}$, we show that the family of perfect crystals $\{B^{6, l}\}_{l\geq 1}$ for $D_6^{(1)}$ given in \cite{KMN2} is a coherent family of perfect crystals with limit $B^{6, \infty}$. Furthermore, we prove that the ultra-discretization of the positive geometric crystal 
$\cV (D_6^{(1)})$ constructed in \cite{MP} is isomorphic as crystal to $B^{6, \infty}$ proving the conjecture  \cite{KNO} in this case.

\section{Perfect Crystals of type \bf{$D_6^{(1)}$}}
From now on we assume $\ge$ to be the affine Lie algebra $D_6^{(1)}$ with index set $I = \{0,1,2,3,4,5,6\}$, Cartan matrix $A = (a_{ij})_{i,j \in I}$ where $a_{ii} = 2, a_{j,j + 1} = -1 = a_{j+1,j}, \; j = 1,2,3,4, a_{02} = a_{20} = a_{46} = a_{64} = -1, a_{ij} = 0$ otherwise, and Dynkin diagram:
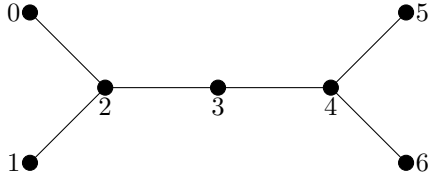
\begin{center}
\begin{tikzpicture}
\draw (-2,1)--(-1,0); \draw (-2,-1)--(-1,0); \draw (-1,0)--(.5,0); \draw (.5,0)--(2,0); \draw (2,0)--(3,1); \draw (2,0)--(3,-1);
\draw [fill] (-2,1) circle [radius=0.1] node[left=.1pt] (a) {0};
\draw [fill] (-2,-1) circle [radius=0.1] node[left=.1pt] (b) {1};
\draw [fill] (-1,0) circle [radius=0.1] node[below=.3pt] (c) {2};
\draw [fill] (.5,0) circle [radius=0.1] node[below=.3pt] (d) {3};
\draw [fill] (2,0) circle [radius=0.1] node[below=.3pt] (e) {4};
\draw [fill] (3,1) circle [radius=0.1] node[right=.1pt] (f) {5};
\draw [fill] (3,-1) circle [radius=0.1] node[right=.1pt] (g) {6};        
\end{tikzpicture}
\end{center}
Let $\{\alpha_0, \alpha_1, \alpha_2, \alpha_3, \alpha_4, \alpha_5, \alpha_6\}, \ \{\check{\alpha_0}, \check{\alpha_1}, \check{\alpha_2}, \check{\alpha_3}, \check{\alpha_4}, \check{\alpha_5}, \check{\alpha_6}\}$ and $\{\Lambda_0, \Lambda_1, \Lambda_2, \Lambda_3, \Lambda_4, \\\Lambda_5, \Lambda_6\}$ denote the set of simple roots, simple coroots and fundamental weights, respectively.
Then ${\bf c} =\check{\alpha_0}+\check{\alpha_1}+2\check{\alpha_2}+2\check{\alpha_3}+2\check{\alpha_4}+\check{\alpha_5}+\check{\alpha_6}$ and $\delta = \al_0 +\al_1+2\al_2+2\al_3+2\al_4+\al_5+\al_6$ are the canonical central element and null root respectively. The sets $P_{cl} = \oplus_{j=0}^6 \Z\L_j$ and $P = P_{cl}\oplus\Z\delta$ are called classical weight lattice and weight lattice respectively.

In this section we reformulate the $D_6^{(1)}$-perfect crystals $\{B^{6,l}\}_{l\in\mathbb{Z}_{\geq 1}}$ corresponding to the spin node $k=6$ given in \cite{KMN2} in coordinatized form and show that 
it a coherent family of perfect crystal with limit $B^{6,\infty}$.

For a positive integer $l$, we consider the sets $B^{6,l}$ and $B^{6,\infty}$ as follows. 
\begin{align*}
B^{6,l} &= 
     \left \{ b= (b_{ij})_{\scriptsize{\begin{array}{l}i \leq j \leq i+5,\\ 1 \leq i \leq 6\end{array}}} \middle|
\begin{aligned}
&b_{ij} \in \mZ_{\geq 0},\ \sum_{j=i}^{i+5} b_{ij} = l,\ 1 \leq i \leq 6,\\ 
&\sum_{j=i}^{6-t} b_{ij} = \sum_{j=i+t}^{5+t} b_{i+t,j},\ 1 \leq i, t \leq 5,\\
& \sum_{j=i}^{t} b_{ij} \geq \sum_{j=i+1}^{t+1} b_{i+1,j},\ 1 \leq i \leq t \leq 5
\end{aligned} \right\}, \\
B^{6,\infty} &= 
    \left \{ b= (b_{ij})_{\scriptsize{\begin{array}{l}i \leq j \leq i+5,\\ 1 \leq i \leq 6\end{array}}} \middle|
\begin{aligned}
&b_{ij} \in \mZ,\ \sum_{j=i}^{i+5} b_{ij} = 0,\ 1 \leq i \leq 6, \\
&\sum_{j=i}^{6-t} b_{ij} = \sum_{j=i+t}^{5+t} b_{i+t,j},\ 1 \leq i, t \leq 5
 \end{aligned} \right\}.
\end{align*}
For $\cB = B^{6,l} \, \text{or} \, B^{6,\infty}$ we define the maps $\e_k, \f_k : \cB \longrightarrow \cB \cup \{0\}$, $\veps_k , \vphi_k : \cB \longrightarrow \mZ$, $0 \leq k \leq 6$ and $\text{wt} : \cB \longrightarrow P_{cl}$, as follows. First we define conditions $(E_j), \ 1 \leq j \leq 14$:
\begin{align*}
(E_1) 	& \hspace{5pt}-b_{13}-b_{14}-b_{15}+b_{22} > b_{55}, \\
		& \hspace{5pt}-b_{13}-b_{14}-b_{15}+b_{22} > -b_{13}-b_{23}+b_{44}+b_{45}, \\
		& \hspace{5pt}-b_{13}-b_{14}-b_{15}+b_{22} > -b_{13}-b_{34}+b_{44}+b_{45}, \\
		& \hspace{5pt}-b_{13}-b_{14}-b_{15}+b_{22} > -b_{13}+b_{44}, \\
		& \hspace{5pt}-b_{13}-b_{14}-b_{15}+b_{22} > -b_{24}-b_{34}+b_{44}+b_{45}, \\
		& \hspace{5pt}-b_{13}-b_{14}-b_{15}+b_{22} > -b_{24}+b_{44}, \\
		& \hspace{5pt}-b_{13}-b_{14}-b_{15}+b_{22} > -b_{35}+b_{44}, \\
		& \hspace{5pt}-b_{13}-b_{14}-b_{15}+b_{22} > -b_{13}-b_{14}-b_{23}-b_{24}+b_{33}+b_{34}+b_{35}, \\
		& \hspace{5pt}-b_{13}-b_{14}-b_{15}+b_{22} > -b_{13}-b_{14}-b_{23}+b_{33}+b_{34}, \\
		& \hspace{5pt}-b_{13}-b_{14}-b_{15}+b_{22} > -b_{13}-b_{23}-b_{25}+b_{33}+b_{34}, \\
		& \hspace{5pt}-b_{13}-b_{14}-b_{15}+b_{22} > -b_{13}-b_{14}+b_{33}, \\
		& \hspace{5pt}-b_{13}-b_{14}-b_{15}+b_{22} > -b_{13}-b_{25}+b_{33}, \\
		& \hspace{5pt}-b_{13}-b_{14}-b_{15}+b_{22} > -b_{24}-b_{25}+b_{33}, \\
(E_2) 	& \hspace{5pt}-b_{13}-b_{14}-b_{23}-b_{24}+b_{33}+b_{34}+b_{35} > b_{55}, \\
		& \hspace{5pt}-b_{13}-b_{14}-b_{23}-b_{24}+b_{33}+b_{34}+b_{35} > -b_{13}-b_{23}+b_{44}+b_{45}, \\
		& \hspace{5pt}-b_{13}-b_{14}-b_{23}-b_{24}+b_{33}+b_{34}+b_{35} > -b_{13}-b_{34}+b_{44}+b_{45}, \\
		& \hspace{5pt}-b_{13}-b_{14}-b_{23}-b_{24}+b_{33}+b_{34}+b_{35} > -b_{13}+b_{44}, \\
		& \hspace{5pt}-b_{13}-b_{14}-b_{23}-b_{24}+b_{33}+b_{34}+b_{35} > -b_{24}-b_{34}+b_{44}+b_{45}, \\
		& \hspace{5pt}-b_{13}-b_{14}-b_{23}-b_{24}+b_{33}+b_{34}+b_{35} > -b_{24}+b_{44}, \\
		& \hspace{5pt}-b_{13}-b_{14}-b_{23}-b_{24}+b_{33}+b_{34}+b_{35} > -b_{35}+b_{44}, \\
		& \hspace{5pt}-b_{13}-b_{14}-b_{23}-b_{24}+b_{33}+b_{34}+b_{35} \geq -b_{13}-b_{14}-b_{15}+b_{22}, \\
		& \hspace{5pt}-b_{13}-b_{14}-b_{23}-b_{24}+b_{33}+b_{34}+b_{35} > -b_{13}-b_{14}-b_{23}+b_{33}+b_{34}, \\
		& \hspace{5pt}-b_{13}-b_{14}-b_{23}-b_{24}+b_{33}+b_{34}+b_{35} > -b_{13}-b_{23}-b_{25}+b_{33}+b_{34}, \\
		& \hspace{5pt}-b_{13}-b_{14}-b_{23}-b_{24}+b_{33}+b_{34}+b_{35} > -b_{13}-b_{14}+b_{33}, \\
		& \hspace{5pt}-b_{13}-b_{14}-b_{23}-b_{24}+b_{33}+b_{34}+b_{35} > -b_{13}-b_{25}+b_{33}, \\
		& \hspace{5pt}-b_{13}-b_{14}-b_{23}-b_{24}+b_{33}+b_{34}+b_{35} > -b_{24}-b_{25}+b_{33}, \\
(E_3) 	& \hspace{5pt}-b_{13}-b_{14}-b_{23}+b_{33}+b_{34} > b_{55}, \\
		& \hspace{5pt}-b_{13}-b_{14}-b_{23}+b_{33}+b_{34} > -b_{13}-b_{23}+b_{44}+b_{45}, \\
		& \hspace{5pt}-b_{13}-b_{14}-b_{23}+b_{33}+b_{34} > -b_{13}-b_{34}+b_{44}+b_{45}, \\
		& \hspace{5pt}-b_{13}-b_{14}-b_{23}+b_{33}+b_{34} > -b_{13}+b_{44}, \\
		& \hspace{5pt}-b_{13}-b_{14}-b_{23}+b_{33}+b_{34} > -b_{24}-b_{34}+b_{44}+b_{45}, \\
		& \hspace{5pt}-b_{13}-b_{14}-b_{23}+b_{33}+b_{34} > -b_{24}+b_{44}, \\
		& \hspace{5pt}-b_{13}-b_{14}-b_{23}+b_{33}+b_{34} > -b_{35}+b_{44}, \\
		& \hspace{5pt}-b_{13}-b_{14}-b_{23}+b_{33}+b_{34} \geq -b_{13}-b_{14}-b_{15}+b_{22}, \\
		& \hspace{5pt}-b_{13}-b_{14}-b_{23}+b_{33}+b_{34} \geq -b_{13}-b_{14}-b_{23}-b_{24}+b_{33}+b_{34}+b_{35}, \\
		& \hspace{5pt}-b_{13}-b_{14}-b_{23}+b_{33}+b_{34} > -b_{13}-b_{23}-b_{25}+b_{33}+b_{34}, \\
		& \hspace{5pt}-b_{13}-b_{14}-b_{23}+b_{33}+b_{34} > -b_{13}-b_{14}+b_{33}, \\
		& \hspace{5pt}-b_{13}-b_{14}-b_{23}+b_{33}+b_{34} > -b_{13}-b_{25}+b_{33}, \\
		& \hspace{5pt}-b_{13}-b_{14}-b_{23}+b_{33}+b_{34} > -b_{24}-b_{25}+b_{33}, \\
(E_4) 	& \hspace{5pt}-b_{13}-b_{23}-b_{25}+b_{33}+b_{34} > b_{55}, \\
		& \hspace{5pt}-b_{13}-b_{23}-b_{25}+b_{33}+b_{34} > -b_{13}-b_{23}+b_{44}+b_{45}, \\
		& \hspace{5pt}-b_{13}-b_{23}-b_{25}+b_{33}+b_{34} > -b_{13}-b_{34}+b_{44}+b_{45}, \\
		& \hspace{5pt}-b_{13}-b_{23}-b_{25}+b_{33}+b_{34} > -b_{13}+b_{44}, \\
		& \hspace{5pt}-b_{13}-b_{23}-b_{25}+b_{33}+b_{34} > -b_{24}-b_{34}+b_{44}+b_{45}, \\
		& \hspace{5pt}-b_{13}-b_{23}-b_{25}+b_{33}+b_{34} > -b_{24}+b_{44}, \\
		& \hspace{5pt}-b_{13}-b_{23}-b_{25}+b_{33}+b_{34} > -b_{35}+b_{44}, \\
		& \hspace{5pt}-b_{13}-b_{23}-b_{25}+b_{33}+b_{34} \geq -b_{13}-b_{14}-b_{15}+b_{22}, \\
		& \hspace{5pt}-b_{13}-b_{23}-b_{25}+b_{33}+b_{34} \geq -b_{13}-b_{14}-b_{23}-b_{24}+b_{33}+b_{34}+b_{35}, \\
		& \hspace{5pt}-b_{13}-b_{23}-b_{25}+b_{33}+b_{34} \geq -b_{13}-b_{14}-b_{23}+b_{33}+b_{34}, \\
		& \hspace{5pt}-b_{13}-b_{23}-b_{25}+b_{33}+b_{34} > -b_{13}-b_{14}+b_{33}, \\
		& \hspace{5pt}-b_{13}-b_{23}-b_{25}+b_{33}+b_{34} > -b_{13}-b_{25}+b_{33}, \\
		& \hspace{5pt}-b_{13}-b_{23}-b_{25}+b_{33}+b_{34} > -b_{24}-b_{25}+b_{33}, \\
(E_5) 	& \hspace{5pt}-b_{13}-b_{14}+b_{33} > b_{55}, \\
		& \hspace{5pt}-b_{13}-b_{14}+b_{33} > -b_{13}-b_{23}+b_{44}+b_{45}, \\
		& \hspace{5pt}-b_{13}-b_{14}+b_{33} > -b_{13}-b_{34}+b_{44}+b_{45}, \\
		& \hspace{5pt}-b_{13}-b_{14}+b_{33} > -b_{13}+b_{44}, \\
		& \hspace{5pt}-b_{13}-b_{14}+b_{33} > -b_{24}-b_{34}+b_{44}+b_{45}, \\
		& \hspace{5pt}-b_{13}-b_{14}+b_{33} > -b_{24}+b_{44}, \\
		& \hspace{5pt}-b_{13}-b_{14}+b_{33} > -b_{35}+b_{44}, \\
		& \hspace{5pt}-b_{13}-b_{14}+b_{33} \geq -b_{13}-b_{14}-b_{15}+b_{22}, \\
		& \hspace{5pt}-b_{13}-b_{14}+b_{33} \geq -b_{13}-b_{14}-b_{23}-b_{24}+b_{33}+b_{34}+b_{35}, \\
		& \hspace{5pt}-b_{13}-b_{14}+b_{33} \geq -b_{13}-b_{14}-b_{23}+b_{33}+b_{34}, \\
		& \hspace{5pt}-b_{13}-b_{14}+b_{33} > -b_{13}-b_{23}-b_{25}+b_{33}+b_{34}, \\
		& \hspace{5pt}-b_{13}-b_{14}+b_{33} > -b_{13}-b_{25}+b_{33}, \\
		& \hspace{5pt}-b_{13}-b_{14}+b_{33} > -b_{24}-b_{25}+b_{33}, \\
(E_6) 	& \hspace{5pt}-b_{13}-b_{23}+b_{44}+b_{45} > b_{55}, \\
		& \hspace{5pt}-b_{13}-b_{23}+b_{44}+b_{45} > -b_{13}-b_{34}+b_{44}+b_{45}, \\
		& \hspace{5pt}-b_{13}-b_{23}+b_{44}+b_{45} > -b_{13}+b_{44}, \\
		& \hspace{5pt}-b_{13}-b_{23}+b_{44}+b_{45} > -b_{24}-b_{34}+b_{44}+b_{45}, \\
		& \hspace{5pt}-b_{13}-b_{23}+b_{44}+b_{45} > -b_{24}+b_{44}, \\
		& \hspace{5pt}-b_{13}-b_{23}+b_{44}+b_{45} > -b_{35}+b_{44}, \\
		& \hspace{5pt}-b_{13}-b_{23}+b_{44}+b_{45} \geq -b_{13}-b_{14}-b_{15}+b_{22}, \\
		& \hspace{5pt}-b_{13}-b_{23}+b_{44}+b_{45} \geq -b_{13}-b_{14}-b_{23}-b_{24}+b_{33}+b_{34}+b_{35}, \\
		& \hspace{5pt}-b_{13}-b_{23}+b_{44}+b_{45} \geq -b_{13}-b_{14}-b_{23}+b_{33}+b_{34}, \\
		& \hspace{5pt}-b_{13}-b_{23}+b_{44}+b_{45} \geq -b_{13}-b_{23}-b_{25}+b_{33}+b_{34}, \\
		& \hspace{5pt}-b_{13}-b_{23}+b_{44}+b_{45} > -b_{13}-b_{14}+b_{33}, \\
		& \hspace{5pt}-b_{13}-b_{23}+b_{44}+b_{45} > -b_{13}-b_{25}+b_{33}, \\
		& \hspace{5pt}-b_{13}-b_{23}+b_{44}+b_{45} > -b_{24}-b_{25}+b_{33}, \\	
(E_7) 	& \hspace{5pt}-b_{13}-b_{25}+b_{33} > b_{55}, \\
		& \hspace{5pt}-b_{13}-b_{25}+b_{33} > -b_{13}-b_{23}+b_{44}+b_{45}, \\
		& \hspace{5pt}-b_{13}-b_{25}+b_{33} > -b_{13}-b_{34}+b_{44}+b_{45}, \\
		& \hspace{5pt}-b_{13}-b_{25}+b_{33} > -b_{13}+b_{44}, \\
		& \hspace{5pt}-b_{13}-b_{25}+b_{33} > -b_{24}-b_{34}+b_{44}+b_{45}, \\
		& \hspace{5pt}-b_{13}-b_{25}+b_{33} > -b_{24}+b_{44}, \\
		& \hspace{5pt}-b_{13}-b_{25}+b_{33} > -b_{35}+b_{44}, \\
		& \hspace{5pt}-b_{13}-b_{25}+b_{33} \geq -b_{13}-b_{14}-b_{15}+b_{22}, \\
		& \hspace{5pt}-b_{13}-b_{25}+b_{33} \geq -b_{13}-b_{14}-b_{23}-b_{24}+b_{33}+b_{34}+b_{35}, \\
		& \hspace{5pt}-b_{13}-b_{25}+b_{33} \geq -b_{13}-b_{14}-b_{23}+b_{33}+b_{34}, \\
		& \hspace{5pt}-b_{13}-b_{25}+b_{33} \geq -b_{13}-b_{23}-b_{25}+b_{33}+b_{34}, \\
		& \hspace{5pt}-b_{13}-b_{25}+b_{33} \geq -b_{13}-b_{14}+b_{33}, \\
		& \hspace{5pt}-b_{13}-b_{25}+b_{33} > -b_{24}-b_{25}+b_{33}, \\		
(E_8) 	& \hspace{5pt}-b_{24}-b_{25}+b_{33} > b_{55}, \\
		& \hspace{5pt}-b_{24}-b_{25}+b_{33} > -b_{13}-b_{23}+b_{44}+b_{45}, \\
		& \hspace{5pt}-b_{24}-b_{25}+b_{33} > -b_{13}-b_{34}+b_{44}+b_{45}, \\
		& \hspace{5pt}-b_{24}-b_{25}+b_{33} > -b_{13}+b_{44}, \\
		& \hspace{5pt}-b_{24}-b_{25}+b_{33} > -b_{24}-b_{34}+b_{44}+b_{45}, \\
		& \hspace{5pt}-b_{24}-b_{25}+b_{33} > -b_{24}+b_{44}, \\
		& \hspace{5pt}-b_{24}-b_{25}+b_{33} > -b_{35}+b_{44}, \\
		& \hspace{5pt}-b_{24}-b_{25}+b_{33} \geq -b_{13}-b_{14}-b_{15}+b_{22}, \\
		& \hspace{5pt}-b_{24}-b_{25}+b_{33} \geq -b_{13}-b_{14}-b_{23}-b_{24}+b_{33}+b_{34}+b_{35}, \\
		& \hspace{5pt}-b_{24}-b_{25}+b_{33} \geq -b_{13}-b_{14}-b_{23}+b_{33}+b_{34}, \\
		& \hspace{5pt}-b_{24}-b_{25}+b_{33} \geq -b_{13}-b_{23}-b_{25}+b_{33}+b_{34}, \\
		& \hspace{5pt}-b_{24}-b_{25}+b_{33} \geq -b_{13}-b_{14}+b_{33}, \\
		& \hspace{5pt}-b_{24}-b_{25}+b_{33} \geq -b_{13}-b_{25}+b_{33}, \\
(E_9) 	& \hspace{5pt}-b_{13}-b_{34}+b_{44}+b_{45} > b_{55}, \\
		& \hspace{5pt}-b_{13}-b_{34}+b_{44}+b_{45} \geq -b_{13}-b_{23}+b_{44}+b_{45}, \\
		& \hspace{5pt}-b_{13}-b_{34}+b_{44}+b_{45} > -b_{13}+b_{44}, \\
		& \hspace{5pt}-b_{13}-b_{34}+b_{44}+b_{45} > -b_{24}-b_{34}+b_{44}+b_{45}, \\
		& \hspace{5pt}-b_{13}-b_{34}+b_{44}+b_{45} > -b_{24}+b_{44}, \\
		& \hspace{5pt}-b_{13}-b_{34}+b_{44}+b_{45} > -b_{35}+b_{44}, \\
		& \hspace{5pt}-b_{13}-b_{34}+b_{44}+b_{45} \geq -b_{13}-b_{14}-b_{15}+b_{22}, \\
		& \hspace{5pt}-b_{13}-b_{34}+b_{44}+b_{45} \geq -b_{13}-b_{14}-b_{23}-b_{24}+b_{33}+b_{34}+b_{35}, \\
		& \hspace{5pt}-b_{13}-b_{34}+b_{44}+b_{45} \geq -b_{13}-b_{14}-b_{23}+b_{33}+b_{34}, \\
		& \hspace{5pt}-b_{13}-b_{34}+b_{44}+b_{45} \geq -b_{13}-b_{23}-b_{25}+b_{33}+b_{34}, \\
		& \hspace{5pt}-b_{13}-b_{34}+b_{44}+b_{45} \geq -b_{13}-b_{14}+b_{33}, \\
		& \hspace{5pt}-b_{13}-b_{34}+b_{44}+b_{45} \geq -b_{13}-b_{25}+b_{33}, \\
		& \hspace{5pt}-b_{13}-b_{34}+b_{44}+b_{45} > -b_{24}-b_{25}+b_{33}, \\	
(E_{10}) 	& \hspace{5pt}-b_{13}+b_{44} > b_{55}, \\
		& \hspace{5pt}-b_{13}+b_{44} \geq -b_{13}-b_{23}+b_{44}+b_{45}, \\
		& \hspace{5pt}-b_{13}+b_{44} \geq -b_{13}-b_{34}+b_{44}+b_{45}, \\
		& \hspace{5pt}-b_{13}+b_{44} > -b_{24}-b_{34}+b_{44}+b_{45}, \\
		& \hspace{5pt}-b_{13}+b_{44} > -b_{24}+b_{44}, \\
		& \hspace{5pt}-b_{13}+b_{44} > -b_{35}+b_{44}, \\
		& \hspace{5pt}-b_{13}+b_{44} \geq -b_{13}-b_{14}-b_{15}+b_{22}, \\
		& \hspace{5pt}-b_{13}+b_{44} \geq -b_{13}-b_{14}-b_{23}-b_{24}+b_{33}+b_{34}+b_{35}, \\
		& \hspace{5pt}-b_{13}+b_{44} \geq -b_{13}-b_{14}-b_{23}+b_{33}+b_{34}, \\
		& \hspace{5pt}-b_{13}+b_{44} \geq -b_{13}-b_{23}-b_{25}+b_{33}+b_{34}, \\
		& \hspace{5pt}-b_{13}+b_{44} \geq -b_{13}-b_{14}+b_{33}, \\
		& \hspace{5pt}-b_{13}+b_{44} \geq -b_{13}-b_{25}+b_{33}, \\
		& \hspace{5pt}-b_{13}+b_{44} > -b_{24}-b_{25}+b_{33}, \\
(E_{11}) 	& \hspace{5pt}-b_{24}-b_{34}+b_{44}+b_{45} > b_{55}, \\
		& \hspace{5pt}-b_{24}-b_{34}+b_{44}+b_{45} \geq -b_{13}-b_{23}+b_{44}+b_{45}, \\
		& \hspace{5pt}-b_{24}-b_{34}+b_{44}+b_{45} \geq -b_{13}-b_{34}+b_{44}+b_{45}, \\
		& \hspace{5pt}-b_{24}-b_{34}+b_{44}+b_{45} > -b_{13}+b_{44}, \\
		& \hspace{5pt}-b_{24}-b_{34}+b_{44}+b_{45} > -b_{24}+b_{44}, \\
		& \hspace{5pt}-b_{24}-b_{34}+b_{44}+b_{45} > -b_{35}+b_{44}, \\
		& \hspace{5pt}-b_{24}-b_{34}+b_{44}+b_{45} \geq -b_{13}-b_{14}-b_{15}+b_{22}, \\
		& \hspace{5pt}-b_{24}-b_{34}+b_{44}+b_{45} \geq -b_{13}-b_{14}-b_{23}-b_{24}+b_{33}+b_{34}+b_{35}, \\
		& \hspace{5pt}-b_{24}-b_{34}+b_{44}+b_{45} \geq -b_{13}-b_{14}-b_{23}+b_{33}+b_{34}, \\
		& \hspace{5pt}-b_{24}-b_{34}+b_{44}+b_{45} \geq -b_{13}-b_{23}-b_{25}+b_{33}+b_{34}, \\
		& \hspace{5pt}-b_{24}-b_{34}+b_{44}+b_{45} \geq -b_{13}-b_{14}+b_{33}, \\
		& \hspace{5pt}-b_{24}-b_{34}+b_{44}+b_{45} \geq -b_{13}-b_{25}+b_{33}, \\
		& \hspace{5pt}-b_{24}-b_{34}+b_{44}+b_{45} \geq -b_{24}-b_{25}+b_{33}, \\		
(E_{12}) 	& \hspace{5pt}-b_{24}+b_{44} > b_{55}, \\
		& \hspace{5pt}-b_{24}+b_{44} \geq -b_{13}-b_{23}+b_{44}+b_{45}, \\
		& \hspace{5pt}-b_{24}+b_{44} \geq -b_{13}-b_{34}+b_{44}+b_{45}, \\
		& \hspace{5pt}-b_{24}+b_{44} \geq -b_{13}+b_{44}, \\
		& \hspace{5pt}-b_{24}+b_{44} \geq -b_{24}-b_{34}+b_{44}+b_{45}, \\
		& \hspace{5pt}-b_{24}+b_{44} > -b_{35}+b_{44}, \\
		& \hspace{5pt}-b_{24}+b_{44} \geq -b_{13}-b_{14}-b_{15}+b_{22}, \\
		& \hspace{5pt}-b_{24}+b_{44} \geq -b_{13}-b_{14}-b_{23}-b_{24}+b_{33}+b_{34}+b_{35}, \\
		& \hspace{5pt}-b_{24}+b_{44} \geq -b_{13}-b_{14}-b_{23}+b_{33}+b_{34}, \\
		& \hspace{5pt}-b_{24}+b_{44} \geq -b_{13}-b_{23}-b_{25}+b_{33}+b_{34}, \\
		& \hspace{5pt}-b_{24}+b_{44} \geq -b_{13}-b_{14}+b_{33}, \\
		& \hspace{5pt}-b_{24}+b_{44} \geq -b_{13}-b_{25}+b_{33}, \\
		& \hspace{5pt}-b_{24}+b_{44} \geq -b_{24}-b_{25}+b_{33}, \\		
(E_{13}) 	& \hspace{5pt}-b_{35}+b_{44} > b_{55}, \\
		& \hspace{5pt}-b_{35}+b_{44} \geq -b_{13}-b_{23}+b_{44}+b_{45}, \\
		& \hspace{5pt}-b_{35}+b_{44} \geq -b_{13}-b_{34}+b_{44}+b_{45}, \\
		& \hspace{5pt}-b_{35}+b_{44} \geq -b_{13}+b_{44}, \\
		& \hspace{5pt}-b_{35}+b_{44} \geq -b_{24}-b_{34}+b_{44}+b_{45}, \\
		& \hspace{5pt}-b_{35}+b_{44} \geq -b_{24}+b_{44}, \\
		& \hspace{5pt}-b_{35}+b_{44} \geq -b_{13}-b_{14}-b_{15}+b_{22}, \\
		& \hspace{5pt}-b_{35}+b_{44} \geq -b_{13}-b_{14}-b_{23}-b_{24}+b_{33}+b_{34}+b_{35}, \\
		& \hspace{5pt}-b_{35}+b_{44} \geq -b_{13}-b_{14}-b_{23}+b_{33}+b_{34}, \\
		& \hspace{5pt}-b_{35}+b_{44} \geq -b_{13}-b_{23}-b_{25}+b_{33}+b_{34}, \\
		& \hspace{5pt}-b_{35}+b_{44} \geq -b_{13}-b_{14}+b_{33}, \\
		& \hspace{5pt}-b_{35}+b_{44} \geq -b_{13}-b_{25}+b_{33}, \\
		& \hspace{5pt}-b_{35}+b_{44} \geq -b_{24}-b_{25}+b_{33}, \\
(E_{14}) 	& \hspace{5pt}b_{55} \geq -b_{13}-b_{23}+b_{44}+b_{45}, \\
		& \hspace{5pt}b_{55} \geq -b_{13}-b_{34}+b_{44}+b_{45}, \\
		& \hspace{5pt}b_{55} \geq -b_{13}+b_{44}, \\
		& \hspace{5pt}b_{55} \geq -b_{24}-b_{34}+b_{44}+b_{45}, \\
		& \hspace{5pt}b_{55} \geq -b_{24}+b_{44}, \\
		& \hspace{5pt}b_{55} \geq -b_{35}+b_{44}, \\
		& \hspace{5pt}b_{55} \geq -b_{13}-b_{14}-b_{15}+b_{22}, \\
		& \hspace{5pt}b_{55} \geq -b_{13}-b_{14}-b_{23}-b_{24}+b_{33}+b_{34}+b_{35}, \\
		& \hspace{5pt}b_{55} \geq -b_{13}-b_{14}-b_{23}+b_{33}+b_{34}, \\
		& \hspace{5pt}b_{55} \geq -b_{13}-b_{23}-b_{25}+b_{33}+b_{34}, \\
		& \hspace{5pt}b_{55} \geq -b_{13}-b_{14}+b_{33}, \\
		& \hspace{5pt}b_{55} \geq -b_{13}-b_{25}+b_{33}, \\
		& \hspace{5pt}b_{55} \geq -b_{24}-b_{25}+b_{33}.  
\end{align*} Then we define conditions $(F_j) \ (1 \leq j \leq 14)$ by replacing $>$ (resp. $\geq$) with $\geq$ (resp. $>$) in $(E_j)$. Let $b=(b_{ij}) \in \cB$. Then for $\tilde{e_k}(b) = (b'_{ij})$ where 
\begin{align*}
k=0 &: 
	\small \begin{cases}
	b'_{11} = b_{11} - 1,b'_{16} = b_{16} + 1, b'_{22} = b_{22} - 1, b'_{27} = b_{27} + 1, 
	b'_{36} = b_{36} - 1,  \\ b'_{38} = b_{38} + 1,  b'_{47} = b_{47} - 1, b'_{49} = b_{49} + 1, 
	b'_{59} = b_{59} - 1, b'_{5,10} = b_{5,10} + 1, \\ b'_{69} = b_{69} - 1, b'_{6,11} = b_{6,11} + 1 \ \text{if} \ (E_1)  \vspace{1pt}\\ 
	b'_{11} = b_{11} - 1,b'_{15} = b_{15} + 1, b'_{22} = b_{22} - 1, b'_{26} = b_{26} + 1, 
	b'_{35} = b_{35} - 1,  \\ b'_{38} = b_{38} + 1, b'_{46} = b_{46} - 1, b'_{49} = b_{49} + 1, 
	b'_{58} = b_{58} - 1, b'_{5,10} = b_{5,10} + 1, \\ b'_{69} = b_{69} - 1, b'_{6,11} = b_{6,11} + 1 \ \text{if} \ (E_2)  \vspace{1pt}\\ 
	b'_{11} = b_{11} - 1,b'_{15} = b_{15} + 1, b'_{22} = b_{22} - 1, b'_{24} = b_{24} + 1, 
	b'_{25} = b_{25} - 1, \\ b'_{26} = b_{26} + 1, b'_{34} = b_{34} - 1,  b'_{38} = b_{38} + 1, 
	b'_{46} = b_{46} - 1, b'_{47} = b_{47} + 1, \\ b'_{48} = b_{48} - 1, b'_{49} = b_{49} + 1  
	b'_{57} = b_{57} - 1, b'_{5,10} = b_{5,10} + 1, b'_{69} = b_{69} - 1, \\ b'_{6,11} = b_{6,11} + 1 \ \text{if} \ (E_3)  \vspace{1pt}\\ 
	b'_{11} = b_{11} - 1,b'_{14} = b_{14} + 1, b'_{22} = b_{22} - 1, b'_{26} = b_{26} + 1, 
	b'_{34} = b_{34} - 1,  \\ b'_{37} = b_{37} + 1, b'_{46} = b_{46} - 1, b'_{49} = b_{49} + 1,  
	b'_{57} = b_{57} - 1, b'_{5,10} = b_{5,10} + 1,\\ b'_{69} = b_{69} - 1, b'_{6,11} = b_{6,11} + 1 \ \text{if} \ (E_4)  \vspace{1pt}\\ 
	b'_{11} = b_{11} - 1,b'_{15} = b_{15} + 1, b'_{22} = b_{22} - 1, b'_{23} = b_{23} + 1, 
	b'_{25} = b_{25} - 1, \\ b'_{26} = b_{26} + 1, b'_{33} = b_{33} - 1,  b'_{38} = b_{38} + 1, 
	b'_{46} = b_{46} - 1, b'_{47} = b_{47} + 1, \\ b'_{48} = b_{48} - 1, b'_{49} = b_{49} + 1  
	b'_{57} = b_{57} - 1, b'_{58} = b_{58} + 1, b'_{59} = b_{59} - 1, \\ b'_{5,10} = b_{5,10} + 1, 
	b'_{68} = b_{68} - 1, b'_{6,11} = b_{6,11} + 1 \ \text{if} \ (E_5)  \vspace{1pt}\\ 
	b'_{11} = b_{11} - 1,b'_{14} = b_{14} + 1, b'_{22} = b_{22} - 1, b'_{25} = b_{25} + 1, 
	b'_{34} = b_{34} - 1, \\ b'_{36} = b_{36} + 1, b'_{45} = b_{45} - 1, b'_{49} = b_{49} + 1,  
	b'_{56} = b_{56} - 1, b'_{5,10} = b_{5,10} + 1,\\ b'_{69} = b_{69} - 1, b'_{6,11} = b_{6,11} + 1 \ \text{if} \ (E_6)  \vspace{1pt}\\ 
	b'_{11} = b_{11} - 1,b'_{14} = b_{14} + 1, b'_{22} = b_{22} - 1, b'_{23} = b_{23} + 1,  
	b'_{24} = b_{24} - 1, \\ b'_{26} = b_{26} + 1, b'_{33} = b_{33} - 1,  b'_{37} = b_{37} + 1, 
	b'_{46} = b_{46} - 1, b'_{48} = b_{48} + 1,\\ b'_{57} = b_{57} - 1, b'_{58} = b_{58} + 1,
	b'_{59} = b_{59} - 1, b'_{5,10} = b_{5,10} + 1, b'_{68} = b_{68} - 1,\\ b'_{6,11} = b_{6,11} + 1 \ \text{if} \ (E_7)  \vspace{1pt}\\ 
	b'_{11} = b_{11} - 1,b'_{13} = b_{13} + 1, b'_{22} = b_{22} - 1, b'_{26} = b_{26} + 1,  
	b'_{33} = b_{33} - 1, \\ b'_{37} = b_{37} + 1, b'_{46} = b_{46} - 1, b'_{48} = b_{48} + 1,  
	b'_{57} = b_{57} - 1, b'_{5,10} = b_{5,10} + 1,\\ b'_{68} = b_{68} - 1, b'_{6,11} = b_{6,11} + 1 \ \text{if} \ (E_8)  \vspace{1pt}\\ 
	b'_{11} = b_{11} - 1,b'_{14} = b_{14} + 1, b'_{22} = b_{22} - 1, b'_{23} = b_{23} + 1,  
	b'_{24} = b_{24} - 1, \\ b'_{25} = b_{25} + 1, b'_{33} = b_{33} - 1,  b'_{36} = b_{36} + 1, 
	b'_{45} = b_{45} - 1, b'_{49} = b_{49} + 1, \\ b'_{56} = b_{56} - 1, b'_{58} = b_{58} + 1  
	b'_{59} = b_{59} - 1, b'_{5,10} = b_{5,10} + 1, b'_{68} = b_{68} - 1, \\
	b'_{6,11} = b_{6,11} + 1 \ \text{if} \ (E_9)  \vspace{1pt}\\
	\end{cases}\\
	k=0&:
	\begin{cases}
	b'_{11} = b_{11} - 1, b'_{14} = b_{14} + 1, b'_{22} = b_{22} - 1, b'_{23} = b_{23} + 1,  
	b'_{24} = b_{24} - 1,\\ b'_{25} = b_{25} + 1, b'_{33} = b_{33} - 1,  b'_{34} = b_{34} + 1, 
	b'_{35} = b_{35} - 1, b'_{36} = b_{36} + 1,\\ b'_{44} = b_{44} - 1, b'_{49} = b_{49} + 1  
	b'_{56} = b_{56} - 1, b'_{57} = b_{57} + 1, b'_{59} = b_{59} - 1,\\ b'_{5,10} = b_{5,10} + 1, 
	b'_{67} = b_{67} - 1, b'_{6,11} = b_{6,11} + 1 \ \text{if} \ (E_{10})  \vspace{1pt}  \\
	b'_{11} = b_{11} - 1, b'_{13} = b_{13} + 1, b'_{22} = b_{22} - 1, b'_{25} = b_{25} + 1,  
	b'_{33} = b_{33} - 1, \\ b'_{36} = b_{36} + 1, b'_{45} = b_{45} - 1, b'_{48} = b_{48} + 1,  
	b'_{56} = b_{56} - 1, b'_{5,10} = b_{5,10} + 1, \\ b'_{68} = b_{68} - 1, b'_{6,11} = b_{6,11} + 1 \ \text{if} \ (E_{11})  \vspace{1pt}\\ 
	b'_{11} = b_{11} - 1,  b'_{13} = b_{13} + 1, b'_{22} = b_{22} - 1, b'_{25} = b_{25} + 1,  
	b'_{33} = b_{33} - 1,\\  b'_{34} = b_{34} + 1, b'_{35} = b_{35} - 1,  b'_{36} = b_{36} + 1, 
	b'_{44} = b_{44} - 1, b'_{48} = b_{48} + 1,\\ b'_{56} = b_{56} - 1, b'_{57} = b_{57} + 1  
	b'_{58} = b_{58} - 1, b'_{5,10} = b_{5,10} + 1, b'_{67} = b_{67} - 1,\\ b'_{6,11} = b_{6,11} + 1 \ \text{if} \ (E_{12})  \vspace{1pt}\\ 
	b'_{11} = b_{11} - 1, b'_{13} = b_{13} + 1, b'_{22} = b_{22} - 1, b'_{24} = b_{24} + 1,  
	b'_{33} = b_{33} - 1,\\ b'_{36} = b_{36} + 1, b'_{44} = b_{44} - 1,  b'_{47} = b_{47} + 1, 
	b'_{56} = b_{56} - 1, b'_{5,10} = b_{5,10} + 1, \\ b'_{67} = b_{67} - 1, b'_{6,11} = b_{6,11} + 1 \ \text{if} \ (E_{13})  \vspace{1pt}\\ 
	b'_{11} = b_{11} - 1,b'_{13} = b_{13} + 1, b'_{22} = b_{22} - 1, b'_{24} = b_{24} + 1,
	b'_{33} = b_{33} - 1, \\ b'_{35} = b_{35} + 1, b'_{44} = b_{44} - 1, b'_{46} = b_{46} + 1,  
	b'_{56} = b_{56} - 1, b'_{5,10} = b_{5,10} + 1,\\ b'_{67} = b_{67} - 1, b'_{6,11} = b_{6,11} + 1 \ \text{if} \ (E_{14})  
	\end{cases}
	\\
k=1 &: b'_{11} =b_{11} +1, b'_{12} =b_{12} -1, b'_{6,10} =b_{6,10} +1, b'_{6,11} =b_{6,11} -1 \\
k=2 &: 
	\begin{cases} 
	b'_{12} =b_{12} +1, b'_{13} =b_{13} -1, b'_{59} =b_{59} +1, b'_{5,10} =b_{5,10} -1 \ \text{if} \ b_{12} \geq b_{23}\\
	b'_{22} =b_{22} +1, b'_{23} =b_{23} -1, b'_{69} =b_{69} +1, b'_{6,10} =b_{6,10} -1 \ \text{if} \ b_{12} < b_{23} 
	\end{cases}
	\\
k=3 &: 
	\begin{cases} 
	b'_{13} =b_{13} +1, b'_{14} =b_{14} -1, b'_{48} =b_{48} +1, b'_{49} =b_{49} -1  \\
	 \hspace{1cm} \text{if} \ b_{13} \geq b_{24}, b_{13}+b_{23} \geq b_{24}+b_{34} \\ 
	 b'_{23} =b_{23} +1, b'_{24} =b_{24} -1, b'_{58} =b_{58} +1, b'_{59} =b_{59} -1  \\ 
	 \hspace{1cm} \text{if} \ b_{13} <  b_{24}, b_{23} \geq b_{34}\\ 
	 b'_{33} =b_{33} +1, b'_{34} =b_{34} -1, b'_{68} =b_{68} +1, b'_{69} =b_{69} -1  \\
	 \hspace{1cm} \text{if} \ b_{13}+b_{23} < b_{24}+b_{34}, b_{23} <  b_{34}
	 \end{cases}
	 \\
k=4 &: 
	\begin{cases} 
	b'_{14} =b_{14} +1, b'_{15} =b_{15} -1, b'_{37} =b_{37} +1, b'_{38} =b_{38} -1 \\ 
	\hspace{1cm} \text{if} \ b_{14} \geq b_{25}, b_{14}+b_{24} \geq b_{25}+b_{35}, b_{14}+b_{24}+b_{34} \geq b_{25}+b_{35}+b_{45}\\ 
	b'_{24} =b_{24} +1, b'_{25} =b_{25} -1, b'_{47} =b_{47} +1, b'_{48} =b_{48} -1 \\ 
	\hspace{1cm} \text{if} \ b_{14} < b_{25}, b_{24} \geq b_{35}, b_{24}+b_{34} \geq b_{35}+b_{45}
	\end{cases}
	\\	 
k=4 &: 
	\begin{cases} 
	b'_{34} =b_{34} +1, b'_{35} =b_{35} -1, b'_{57} =b_{57} +1, b'_{58} =b_{58} -1 \\ 
	\hspace{1cm} \text{if} \ b_{14}+b_{24} < b_{25}+b_{35}, b_{24} < b_{35}, b_{34} \geq b_{45}\\ 
	b'_{44} =b_{44} +1, b'_{45} =b_{45} -1, b'_{67} =b_{67} +1, b'_{68} =b_{68} -1 \\ 
	\hspace{1cm} \text{if} \ b_{14}+b_{24}+b_{34} < b_{25}+b_{35}+b_{45}, b_{24}+b_{34} < b_{35}+b_{45}, b_{34} < b_{45}
	\end{cases}
	\\
k=5 &: 
	\begin{cases} 
	b'_{25} =b_{25} +1, b'_{26} =b_{26} -1, b'_{36} =b_{36} +1, b'_{37} =b_{37} -1 \\ 
	\hspace{1cm}  \text{if} \ b_{25}+b_{44}+b_{45} \geq b_{33}+b_{34}\\ 
	b'_{45} =b_{45} +1, b'_{46} =b_{46} -1, b'_{56} =b_{56} +1, b'_{57} =b_{57} -1 \\ 
	\hspace{1cm}  \text{if} \ b_{25}+b_{44}+b_{45} < b_{33}+b_{34}
	\end{cases} 
	\\
k=6 &: 
	\begin{cases} 
	b'_{15} =b_{15} +1, b'_{16} =b_{16} -1, b'_{26} =b_{26} +1, b'_{27} =b_{27} -1 \\ 
	\hspace{1cm}  \text{if} \ b_{15}+b_{33}+b_{34}+b_{35} \geq b_{22}+b_{23}+b_{24}, \\
	\hspace{1.5cm} b_{15}+b_{33}+b_{34}+2b_{35}+b_{55} \geq b_{22}+b_{23}+b_{24}+b_{44}\\ 
	b'_{35} =b_{35} +1, b'_{36} =b_{36} -1, b'_{46} =b_{46} +1, b'_{47} =b_{47} -1 \\ 
	\hspace{1cm}  \text{if} \ b_{15}+b_{33}+b_{34}+b_{35} < b_{22}+b_{23}+b_{24}, b_{35}+b_{55} \geq b_{44}\\ 
	b'_{55} =b_{55} +1, b'_{56} =b_{56} -1, b'_{66} =b_{66} +1, b'_{67} =b_{67} -1 \\ 
	\hspace{1cm}  \text{if} \ b_{15}+b_{33}+b_{34}+2b_{35}+b_{55} < b_{22}+b_{23}+b_{24}+b_{44}, \\
	\hspace{1.5cm} b_{35}+b_{55} < b_{44} 
	\end{cases} \end{align*} and $b'_{ij} = b_{ij}$ otherwise. 
	
	Also $\tilde{f_k}(b) = (b'_{ij})$ where
\begin{align*}
k=0 &: 
	\small \begin{cases} 
	b'_{11} = b_{11} + 1,b'_{16} = b_{16} - 1, b'_{22} = b_{22} + 1, b'_{27} = b_{27} - 1,  
	b'_{36} = b_{36} + 1, \\ b'_{38} = b_{38} - 1, b'_{47} = b_{47} + 1, b'_{49} = b_{49} - 1,  
	b'_{59} = b_{59} + 1, b'_{5,10} = b_{5,10} - 1, \\ b'_{69} = b_{69} + 1, b'_{6,11} = b_{6,11} - 1 \ \text{if} \ (F_1)  \vspace{1pt}\\ 
	b'_{11} = b_{11} + 1,b'_{15} = b_{15} - 1, b'_{22} = b_{22} + 1, b'_{26} = b_{26} - 1,  
	b'_{35} = b_{35} + 1, \\ b'_{38} = b_{38} - 1, b'_{46} = b_{46} + 1, b'_{49} = b_{49} - 1,  
	b'_{58} = b_{58} + 1, b'_{5,10} = b_{5,10} - 1, \\ b'_{69} = b_{69} + 1, b'_{6,11} = b_{6,11} - 1 \ \text{if} \ (F_2)  \vspace{1pt}\\ 
	b'_{11} = b_{11} + 1,b'_{15} = b_{15} - 1, b'_{22} = b_{22} + 1, b'_{24} = b_{24} - 1,  
	b'_{25} = b_{25} + 1, \\ b'_{26} = b_{26} - 1, b'_{34} = b_{34} + 1,  b'_{38} = b_{38} - 1, 
	b'_{46} = b_{46} + 1, b'_{47} = b_{47} - 1,\\ b'_{48} = b_{48} + 1, b'_{49} = b_{49} - 1,
	b'_{57} = b_{57} + 1, b'_{5,10} = b_{5,10} - 1, b'_{69} = b_{69} + 1, \\ b'_{6,11} = b_{6,11} - 1 \ \text{if} \ (F_3)  \vspace{1pt}\\ 	
	b'_{11} = b_{11} + 1,b'_{14} = b_{14} - 1, b'_{22} = b_{22} + 1, b'_{26} = b_{26} - 1,  
	b'_{34} = b_{34} + 1, \\ b'_{37} = b_{37} - 1, b'_{46} = b_{46} + 1, b'_{49} = b_{49} - 1,  
	b'_{57} = b_{57} + 1, b'_{5,10} = b_{5,10} - 1, \\ b'_{69} = b_{69} + 1, b'_{6,11} = b_{6,11} - 1 \ \text{if} \ (F_4)  
\vspace{1pt}\\
	b'_{11} = b_{11} + 1,b'_{15} = b_{15} - 1, b'_{22} = b_{22} + 1, b'_{23} = b_{23} - 1, 
	b'_{25} = b_{25} + 1,\\ b'_{26} = b_{26} - 1, b'_{33} = b_{33} + 1,  b'_{38} = b_{38} - 1, 
	b'_{46} = b_{46} + 1, b'_{47} = b_{47} - 1,\\ b'_{48} = b_{48} + 1, b'_{49} = b_{49} - 1  
	b'_{57} = b_{57} + 1, b'_{58} = b_{58} - 1, b'_{59} = b_{59} + 1, \\ b'_{5,10} = b_{5,10} - 1, 
	b'_{68} = b_{68} + 1, b'_{6,11} = b_{6,11} - 1 \ \text{if} \ (F_5)  \vspace{1pt}\\ 
	b'_{11} = b_{11} + 1,b'_{14} = b_{14} - 1, b'_{22} = b_{22} + 1, b'_{25} = b_{25} - 1, 
	b'_{34} = b_{34} + 1,\\  b'_{36} = b_{36} - 1, b'_{45} = b_{45} + 1, b'_{49} = b_{49} - 1,  
	b'_{56} = b_{56} + 1, b'_{5,10} = b_{5,10} - 1,\\ b'_{69} = b_{69} + 1, b'_{6,11} = b_{6,11} - 1 \ \text{if} \ (F_6)  \vspace{1pt}\\ 
	b'_{11} = b_{11} + 1,b'_{14} = b_{14} - 1, b'_{22} = b_{22} + 1, b'_{23} = b_{23} - 1, 
	b'_{24} = b_{24} + 1,\\ b'_{26} = b_{26} - 1, b'_{33} = b_{33} + 1,  b'_{37} = b_{37} - 1, 
	b'_{46} = b_{46} + 1, b'_{48} = b_{48} - 1,\\ b'_{57} = b_{57} + 1, b'_{58} = b_{58} - 1, 
	b'_{59} = b_{59} + 1, b'_{5,10} = b_{5,10} - 1, b'_{68} = b_{68} + 1,\\ b'_{6,11} = b_{6,11} - 1 \ \text{if} \ (F_7)  \vspace{1pt} \\
	b'_{11} = b_{11} + 1,b'_{13} = b_{13} - 1, b'_{22} = b_{22} + 1, b'_{26} = b_{26} - 1, 
	b'_{33} = b_{33} + 1,\\  b'_{37} = b_{37} - 1, b'_{46} = b_{46} + 1, b'_{48} = b_{48} - 1,  
	b'_{57} = b_{57} + 1, b'_{5,10} = b_{5,10} - 1, \\ b'_{68} = b_{68} + 1, b'_{6,11} = b_{6,11} - 1 \ \text{if} \ (F_8)  \vspace{1pt}\\ 
	b'_{11} = b_{11} + 1,b'_{14} = b_{14} - 1, b'_{22} = b_{22} + 1, b'_{23} = b_{23} - 1, 
	b'_{24} = b_{24} + 1,\\ b'_{25} = b_{25} - 1, b'_{33} = b_{33} + 1,  b'_{36} = b_{36} - 1, 
	b'_{45} = b_{45} + 1, b'_{49} = b_{49} - 1,\\ b'_{56} = b_{56} + 1, b'_{58} = b_{58} - 1  
	b'_{59} = b_{59} + 1, b'_{5,10} = b_{5,10} - 1, b'_{68} = b_{68} + 1,\\ b'_{6,11} = b_{6,11} - 1 \ \text{if} \ (F_9)  \vspace{1pt}\\ 
	\end{cases}\\ 
	k=0 &: 
	\small \begin{cases}
	b'_{11} = b_{11} + 1, b'_{14} = b_{14} - 1, b'_{22} = b_{22} + 1, b'_{23} = b_{23} - 1,  
	b'_{24} = b_{24} + 1,\\ b'_{25} = b_{25} - 1, b'_{33} = b_{33} + 1,  b'_{34} = b_{34} - 1, 
	b'_{35} = b_{35} + 1, b'_{36} = b_{36} - 1,\\ b'_{44} = b_{44} + 1, b'_{49} = b_{49} - 1  
	b'_{56} = b_{56} + 1, b'_{57} = b_{57} - 1, b'_{59} = b_{59} + 1, \\ b'_{5,10} = b_{5,10} - 1, 
	b'_{67} = b_{67} + 1, b'_{6,11} = b_{6,11} - 1 \ \text{if} \ (F_{10})  \vspace{1pt}\\ 
	b'_{11} = b_{11} + 1, b'_{13} = b_{13} - 1, b'_{22} = b_{22} + 1, b'_{25} = b_{25} - 1, 
	b'_{33} = b_{33} + 1,\\  b'_{36} = b_{36} - 1, b'_{45} = b_{45} + 1, b'_{48} = b_{48} - 1,  
	b'_{56} = b_{56} + 1, b'_{5,10} = b_{5,10} - 1,\\ b'_{68} = b_{68} + 1, b'_{6,11} = b_{6,11} - 1 \ \text{if} \ (F_{11})  \vspace{1pt}\\ 
	b'_{11} = b_{11} + 1,  b'_{13} = b_{13} - 1, b'_{22} = b_{22} + 1, b'_{25} = b_{25} - 1,  
	b'_{33} = b_{33} + 1,\\ b'_{34} = b_{34} - 1, b'_{35} = b_{35} + 1,  b'_{36} = b_{36} - 1, 
	b'_{44} = b_{44} + 1, b'_{48} = b_{48} - 1,\\ b'_{56} = b_{56} + 1, b'_{57} = b_{57} - 1  
	b'_{58} = b_{58} + 1, b'_{5,10} = b_{5,10} - 1, b'_{67} = b_{67} + 1,\\ b'_{6,11} = b_{6,11} - 1 \ \text{if} \ (F_{12})  \vspace{1pt}\\ 
	b'_{11} = b_{11} + 1, b'_{13} = b_{13} - 1, b'_{22} = b_{22} + 1, b'_{24} = b_{24} - 1,  
	b'_{33} = b_{33} + 1,\\ b'_{36} = b_{36} - 1, b'_{44} = b_{44} + 1,  b'_{47} = b_{47} - 1, 
	b'_{56} = b_{56} + 1, b'_{5,10} = b_{5,10} - 1,\\ b'_{67} = b_{67} + 1, b'_{6,11} = b_{6,11} - 1 \ \text{if} \ (F_{13})  \vspace{1pt}\\ 
	b'_{11} = b_{11} + 1,b'_{13} = b_{13} - 1, b'_{22} = b_{22} + 1, b'_{24} = b_{24} - 1, 
	b'_{33} = b_{33} + 1, \\ b'_{35} = b_{35} - 1, b'_{44} = b_{44} + 1, b'_{46} = b_{46} - 1,  
	b'_{56} = b_{56} + 1, b'_{5,10} = b_{5,10} - 1,\\ b'_{67} = b_{67} + 1, b'_{6,11} = b_{6,11} - 1 \ \text{if} \ (F_{14})  
	\end{cases}
	\\
k=1 &: b'_{11} =b_{11} -1, b'_{12} =b_{12} +1, b'_{6,10} =b_{6,10} -1, b'_{6,11} =b_{6,11} +1 \\
k=2 &: 
	\begin{cases} 
	b'_{12} =b_{12} -1, b'_{13} =b_{13} +1, b'_{59} =b_{59} -1, b'_{5,10} =b_{5,10} +1 \ \text{if} \ b_{12} > b_{23}\\
	b'_{22} =b_{22} -1, b'_{23} =b_{23} +1, b'_{69} =b_{69} -1, b'_{6,10} =b_{6,10} +1 \ \text{if} \ b_{12} \leq b_{23} 
	\end{cases}
	\\
k=3 &: 
	\begin{cases} 
	b'_{13} =b_{13} -1, b'_{14} =b_{14} +1, b'_{48} =b_{48} -1, b'_{49} =b_{49} +1  \\
	 \hspace{1cm} \text{if} \ b_{13} > b_{24}, b_{13}+b_{23} > b_{24}+b_{34} \\ 
	 b'_{23} =b_{23} -1, b'_{24} =b_{24} +1, b'_{58} =b_{58} -1, b'_{59} =b_{59} +1  \\ 
	 \hspace{1cm} \text{if} \ b_{13} \leq  b_{24}, b_{23} > b_{34}\\ 
	 b'_{33} =b_{33} -1, b'_{34} =b_{34} +1, b'_{68} =b_{68} -1, b'_{69} =b_{69} +1  \\
	 \hspace{1cm} \text{if} \ b_{13}+b_{23} \leq b_{24}+b_{34}, b_{23} \leq  b_{34}
	 \end{cases}
	 \\
k=4 &: 
	\begin{cases} 
	b'_{14} =b_{14} -1, b'_{15} =b_{15} +1, b'_{37} =b_{37} -1, b'_{38} =b_{38} +1 \\ 
	\hspace{1cm} \text{if} \ b_{14} > b_{25}, b_{14}+b_{24} > b_{25}+b_{35}, b_{14}+b_{24}+b_{34} > b_{25}+b_{35}+b_{45}\\ 
	b'_{24} =b_{24} -1, b'_{25} =b_{25} +1, b'_{47} =b_{47} -1, b'_{48} =b_{48} +1 \\ 
	\hspace{1cm} \text{if} \ b_{14} \leq b_{25}, b_{24} > b_{35}, b_{24}+b_{34} > b_{35}+b_{45}\\ 
	b'_{34} =b_{34} -1, b'_{35} =b_{35} +1, b'_{57} =b_{57} -1, b'_{58} =b_{58} +1 \\ 
	\hspace{1cm} \text{if} \ b_{14}+b_{24} \leq b_{25}+b_{35}, b_{24} \leq b_{35}, b_{34} > b_{45}\\ 
	b'_{44} =b_{44} -1, b'_{45} =b_{45} +1, b'_{67} =b_{67} -1, b'_{68} =b_{68} +1 \\ 
	\hspace{1cm} \text{if} \ b_{14}+b_{24}+b_{34} \leq b_{25}+b_{35}+b_{45}, b_{24}+b_{34} \leq b_{35}+b_{45}, b_{34} \leq b_{45}
	\end{cases}
	\\
k=5 &: 
	\begin{cases} 
	b'_{25} =b_{25} -1, b'_{26} =b_{26} +1, b'_{36} =b_{36} -1, b'_{37} =b_{37} +1 \\ 
	\hspace{1cm}  \text{if} \ b_{25}+b_{44}+b_{45} > b_{33}+b_{34}\\ 
	b'_{45} =b_{45} -1, b'_{46} =b_{46} +1, b'_{56} =b_{56} -1, b'_{57} =b_{57} +1 \\ 
	\hspace{1cm}  \text{if} \ b_{25}+b_{44}+b_{45} \leq b_{33}+b_{34}
	\end{cases} 
	\\
k=6 &: 
	\begin{cases} 
	b'_{15} =b_{15} -1, b'_{16} =b_{16} +1, b'_{26} =b_{26} -1, b'_{27} =b_{27} +1 \\ 
	\hspace{1cm}  \text{if} \ b_{15}+b_{33}+b_{34}+b_{35} > b_{22}+b_{23}+b_{24}, \\
	\hspace{1.5cm} b_{15}+b_{33}+b_{34}+2b_{35}+b_{55} > b_{22}+b_{23}+b_{24}+b_{44}\\ 
	b'_{35} =b_{35} -1, b'_{36} =b_{36} +1, b'_{46} =b_{46} -1, b'_{47} =b_{47} +1 \\ 
	\hspace{1cm}  \text{if} \ b_{15}+b_{33}+b_{34}+b_{35} \leq b_{22}+b_{23}+b_{24}, b_{35}+b_{55} > b_{44}\\ 
	b'_{55} =b_{55} -1, b'_{56} =b_{56} +1, b'_{66} =b_{66} -1, b'_{67} =b_{67} +1 \\ 
	\hspace{1cm}  \text{if} \ b_{15}+b_{33}+b_{34}+2b_{35}+b_{55} \leq b_{22}+b_{23}+b_{24}+b_{44}, b_{35}+b_{55} \leq b_{44}
	\end{cases} 
\end{align*} and $b'_{ij} = b_{ij}$ otherwise. 

For $b \in B^{6,l}$ if $\tilde{e}_k(b)$ or $\tilde{f}_k(b)$ does not belong to $B^{6,l}$, then we assume it to be $0$. The maps 
$\veps_k(b), \ \vphi_k(b)$ and $\text{wt}_k(b)$ for $k=0,1,2,3,4,5,6$ are given as follows. We observe that 
$\text{wt}_k(b) = \vphi_k(b) - \veps_k(b)$,
$\vphi(b) = \sum_{k=0}^6\vphi_k(b)\L_k$, $\veps(b) = \sum_{k=0}^6\veps_k(b)\L_k$ and $\text{wt}(b) = \vphi(b) - \veps(b)$.
\begin{align*}
\veps_0(b) &= \begin{cases} l + \mathcal{A}_1  \, \ \text{if} \,  \  b \in B^{6,l}, \\
\mathcal{A}_1  \, \  \text{if} \,  \  b \in B^{6, \infty}, \end{cases}\\
\text{where} \ \mathcal{A}_1 &= \text{max} \{-b_{56}-b_{57}-b_{58}-b_{59}-b_{5,10}, -b_{13}-b_{23}-b_{46}-b_{47}-b_{48}-b_{49}, \\ 
			&	\hspace{15pt}  -b_{13}-b_{34}-b_{46}-b_{47}-b_{48}-b_{49}, -b_{13}-b_{45}-b_{46}-b_{47}-b_{48}-b_{49}, \\
			&	\hspace{15pt} -b_{24}-b_{34}-b_{46}-b_{47}-b_{48}-b_{49}, -b_{24}-b_{45}-b_{46}-b_{47}-b_{48} -b_{49},\\
			&	\hspace{15pt} -b_{35}-b_{45}-b_{46}-b_{47}-b_{48}-b_{49}, -b_{13}-b_{14}-b_{15}-b_{23}-b_{24} -b_{25}\\
			&	\hspace{15pt} -b_{26}-b_{27}, -b_{13}-b_{14}-b_{23}-b_{24}-b_{36}-b_{37}-b_{38}, -b_{13}-b_{14} -b_{23} \\
			&	\hspace{15pt}  -b_{35}-b_{36}-b_{37}-b_{38}, -b_{13}-b_{23}-b_{25}-b_{35}-b_{36}-b_{37}-b_{38}, -b_{13} \\
			&	\hspace{15pt}  -b_{14}-b_{34}-b_{35}-b_{36}-b_{37}-b_{38}, -b_{13}-b_{25}-b_{34}-b_{35}-b_{36} -b_{37}\\
			&	\hspace{15pt} -b_{38}, -b_{24}-b_{25}-b_{34}-b_{35}-b_{36}-b_{37}-b_{38} \}.\\
\veps_1(b) 	&= b_{12},\\
\veps_2(b) 	&= \text{max} \{b_{13}, -b_{12}+b_{13}+b_{23}\},\\
\veps_3(b) 	&= \text{max} \{b_{14}, -b_{13}+b_{14}+b_{24}, -b_{13}+b_{14}-b_{23}+b_{24}+b_{34}\},\\
\veps_4(b)	&= \text{max} \{b_{15},  -b_{14}+b_{15}+b_{25}, -b_{14}+b_{15}-b_{24}+b_{25}+b_{35}, \\
			& \hspace{15pt} -b_{14}+b_{15}-b_{24}+b_{25}-b_{34}+b_{35}+b_{45}\}, \\
\veps_5(b) 	&= \text{max} \{b_{11}+b_{12}+b_{13}+b_{14}-b_{22}-b_{23}-b_{24}-b_{25},\\
			&\hspace{15pt} b_{11}+b_{12}+b_{13}+b_{14}-b_{22}-b_{23}-b_{24}-2b_{25}+b_{33}+b_{34}-b_{44}-b_{45}\},\\
\veps_6(b)	&= \begin{cases} l + 	\mathcal{A}_2 	 \ \text{if} \ b \in B^{6,l},\\
\mathcal{A}_2  \ \ \text{if} \ b \in B^{6,\infty}, \end{cases}
\end{align*}
\begin{align*}
\text{where} \  \mathcal{A}_2  =	&\text{max} \{-b_{11}-b_{12}-b_{13}-b_{14}-b_{15}, -b_{11}-b_{12}-b_{13}-b_{14}-2b_{15} \\
				& +b_{22}+b_{23}+b_{24}-b_{33}-b_{34}-b_{35}, -b_{11}-b_{12}-b_{13}-b_{14}-2b_{15} \\
				& +b_{22}+b_{23}+b_{24}-b_{33}-b_{34}-2b_{35}+b_{44}-b_{55}\}.
				\end{align*}
				\begin{align*}
\vphi_0(b) 	&= \begin{cases} l+ \mathcal{A}_3, \ \text{if} \   b \in B^{6,l} , \\
\mathcal{A}_3, \ \text{if} \ b \in B^{6, \infty}, \end{cases} \\
\text{where} \ \mathcal{A}_3  &=\text{max}\{-b_{11}-b_{12}+b_{23}+b_{24}+b_{25}+b_{26}+b_{27}-b_{56}-b_{57}-b_{58}-b_{59}-b_{5,10},\\
				&\hspace{15pt} -b_{11}-b_{12}-b_{13}+b_{24}+b_{25}+b_{26}+b_{27}-b_{46}-b_{47}-b_{48} -b_{49},\\
				&\hspace{15pt} -b_{11}-b_{12}-b_{13}+b_{23}+b_{24}+b_{25}+b_{26}+b_{27}-b_{34}-b_{46}-b_{47} -b_{48}-b_{49},\\
				&\hspace{15pt} -b_{11}-b_{12}-b_{13}+b_{23}+b_{24}+b_{25}+b_{26}+b_{27}-b_{45}-b_{46}-b_{47}-b_{48}-b_{49},\\
				&\hspace{15pt}  -b_{11}-b_{12}+b_{23}+b_{25}+b_{26}+b_{27}-b_{34}-b_{46}-b_{47}b_{48}-b_{49}, \\
				&\hspace{15pt} -b_{11}-b_{12}+b_{23}+b_{25}+b_{26}+b_{27}-b_{45}-b_{46}-b_{47}-b_{48}-b_{49},\\
				&\hspace{15pt}  -b_{11}-b_{12}+b_{23}+b_{24}+b_{25}+b_{26}+b_{27}-b_{35}-b_{45}-b_{46}-b_{47}-b_{48}-b_{49}, \\
				&\hspace{15pt} -b_{11}-b_{12}-b_{13}-b_{14}-b_{15}, -b_{11}-b_{12}-b_{13}-b_{14}+b_{25}+b_{26}+b_{27}-b_{36}-b_{37} \\
				&\hspace{15pt} -b_{38}, -b_{11}-b_{12}-b_{13}-b_{14}+b_{24}+b_{25}+b_{26} +b_{27}-b_{35}-b_{36}-b_{37}-b_{38},\\
				&\hspace{15pt}  -b_{11}-b_{12}-b_{13}+b_{24}+b_{26}+b_{27}-b_{35} -b_{36}-b_{37}-b_{38},-b_{11}-b_{12}-b_{13}\\
				&\hspace{15pt} -b_{14}+b_{23}+b_{24}+b_{25}+b_{26}+b_{27}-b_{34}-b_{35}-b_{36}-b_{37}-b_{38},\\
				&\hspace{15pt}  -b_{11}-b_{12}-b_{13}+b_{23}+b_{24}+b_{26}+b_{27} -b_{34}-b_{35}-b_{36}-b_{37}-b_{38},\\
				&\hspace{15pt} -b_{11}-b_{12}+b_{23}+b_{26}+b_{27}-b_{34}-b_{35} -b_{36}-b_{37}-b_{38} \}.  \\
\vphi_1(b) 	&= b_{11}-b_{22},\\
\vphi_2(b) 	&= \text{max} \{b_{22}-b_{33}, b_{12}+b_{22}-b_{23}-b_{33}\},\\
\vphi_3(b) 	&= \text{max} \{b_{33}-b_{44}, b_{23}+b_{33}-b_{34}-b_{44}, b_{13}+b_{23}-b_{24}+b_{33}-b_{34}-b_{44}\},\\
\vphi_4(b) 	&= \text{max} \{b_{44}-b_{55}, b_{34}+b_{44}-b_{45}-b_{55}, b_{24}+b_{34}-b_{35}+b_{44}-b_{45}-b_{55}, \\
			&\hspace{15pt} b_{14}+b_{24}-b_{25}+b_{34}-b_{35}+b_{44}-b_{45}-b_{55}\},\\
\vphi_5(b) 	&= \text{max} \{b_{45}, b_{25}-b_{33}-b_{34}+b_{44}+2b_{45}\},\\			
\vphi_6(b) 	&= \begin{cases} l+ \mathcal{A}_4, \ \text{if} \ b \in B^{6, i}, \\
\mathcal{A}_4, \ \text{if} \ b \in B^{6, \infty}, \end{cases} 
\end{align*}
\begin{align*}
\text{where} \ \mathcal{A}_4 =	&\text{max} \{-b_{56}-b_{57}-b_{58}-b_{59}-b_{5,10}, b_{35}-b_{44}+b_{55}-b_{56}-b_{57}\\
				&	-b_{58}-b_{59}-b_{5,10}, b_{15}-b_{22}-b_{23}-b_{24}+b_{33}+b_{34}+2b_{35}-b_{44}\\
				 & +b_{55}-b_{56}-b_{57}-b_{58}-b_{59}-b_{5,10}\}. 
\end{align*}
\begin{align*}
\text{wt}_0(b) 	&= -b_{11}-b_{12} +b_{23}+b_{24}+b_{25}+b_{26}+b_{27},\\
\text{wt}_1(b) 	&= b_{11}-b_{12}-b_{22},\\
\text{wt}_2(b) 	&= b_{12} -b_{13}+b_{22}-b_{23}-b_{33},\\
\text{wt}_3(b) 	&= b_{13}-b_{14}+b_{23}-b_{24}+b_{33}-b_{34}-b_{44},\\
\text{wt}_4(b)	&= b_{14}-b_{15} +b_{24}-b_{25}+b_{34}-b_{35}+b_{44}-b_{45} -b_{55},\\
\text{wt}_5(b) 	&= -b_{11}-b_{12} -b_{13}-b_{14}+b_{22}+b_{23}+b_{24}+2b_{25}-b_{33}-b_{34}+b_{44}+2b_{45},\\
\text{wt}_6(b)	&= b_{11}+b_{12} +b_{13}+b_{14}+2b_{15}-b_{22}-b_{23}-b_{24}+b_{33}+b_{34}+2b_{35}-b_{44}+b_{55}\\
			&\hspace{15pt} -b_{56}-b_{57}-b_{58}-b_{59}-b_{5,10}.
\end{align*}

\noindent Choose elements $b^0_{\bf{0}}, b^0_{\bf{1}}, b^0_{\bf{2}}, b^0_{\bf{3}}, b^0_{\bf{4}}, b^0_{\bf{5}}, b^0_{\bf{6}}$ where
\begin{align*} 
(b^0_{\bf{0}})_{ij} &=1 &\text{if} \ (i,j) &= (1,6),(2,7),(3,8),(4,9),(5,10),(6,11),  \\
(b^0_{\bf{1}})_{ij} &=1 &\text{if} \ (i,j) &= (1,1),(2,6),(3,7),(4,8),(5,9), (6,10),  \\
(b^0_{\bf{2}})_{ij} &=1 &\text{if} \ (i,j) &= (1,1),(1,5),(2,2),(2,6),(3,6),(3,8),(4,7),(4,9),(5,8),(5,10),\\
							&&& \qquad(6,9),(6,11),  \\
(b^0_{\bf{3}})_{ij} &=1 &\text{if} \ (i,j) &= (1,1),(1,4),(2,2),(2,5),(3,3),(3,6),(4,6),(4,9),(5,7),(5,10),\\
							&&& \qquad (6,8),(6,11),  \\
(b^0_{\bf{4}})_{ij} &=1 &\text{if} \ (i,j) &= (1,1),(1,3),(2,2),(2,4),(3,3),(3,5),(4,4),(4,6),(5,6),(5,10),\\
							&&& \qquad (6,7),(6,11),  \\
(b^0_{\bf{5}})_{ij} &=1 &\text{if} \ (i,j) &= (1,2),(2,3),(3,4),(4,5),(5,6),(6,11),  \\
(b^0_{\bf{6}})_{ij} &=1 &\text{if} \ (i,j) &= (1,1),(2,2),(3,3),(4,4),(5,5),(6,6), 
\end{align*} and $(b^0_{\bf{k}})_{ij} =0$ \ otherwise,\ for $0 \leq k \leq 6$.

As shown in \cite{KMN2}, the crystal $B^{6,l}$ is a perfect crystal with the set of minimal elements:
\begin{align*}
(B^{6,l})_\text{min}& = \{b \in B^{6,l} \mid \langle {\bf c} , \veps (b)\rangle = l\}\\
& =\left \{ \sum_{k=0}^6 a_k b^0_{\bf k} \mid  a_k \in \mZ_{\geq 0},\ a_0+a_1+2a_2+2a_3+2a_4+a_5+a_6=l \right\}.
\end{align*}

For $\lambda \in P_{cl}$, consider the crystal  $T_{\lambda} = \{t_\lambda\}$ with
\begin{align*}
&\tilde{e_k} (t_\lambda) = \tilde{f_k} (t_\lambda) =0,	&\veps_k (t_\lambda) = \vphi_k (t_\lambda) =-\infty,  \\ 
&\text{wt}(t_\lambda)=\lambda,
\end{align*}
for $k=0,1,2,3,4,5,6$. Then for $\lambda, \mu \in P_{cl}$, $T_{\lambda}\otimes B^{6,l} \otimes T_{\mu}$ is a crystal with the structure given by 
\begin{align*}
\tilde{e_k}(t_{\lambda}\otimes b \otimes t_{\mu})&= t_{\lambda}\otimes \tilde{e_k}b \otimes t_{\mu}, &\tilde{f_k}(t_{\lambda}\otimes b \otimes t_{\mu})&= t_{\lambda}\otimes \tilde{f_k}b \otimes t_{\mu}, \\
\veps_k(t_{\lambda}\otimes b \otimes t_{\mu})&= \veps_k(b) - \langle \check\alpha_k,\lambda \rangle, &\vphi_k(t_{\lambda}\otimes b \otimes t_{\mu})&= \vphi_k(b) + \langle \check\alpha_k,\mu \rangle, \\
\text{wt}(t_{\lambda}\otimes b \otimes t_{\mu})&= \lambda+\mu+\text{wt}(b)
\end{align*}
where $t_{\lambda}\otimes b \otimes t_{\mu} \in T_{\lambda}\otimes B^{6,l} \otimes T_{\mu}$.

The notion of a coherent family of perfect crystals and its limit is defined in \cite{KKM}. In the following theorem we prove that the family of $D_6^{(1)}$ crystals $\Bfamsix$ form a coherent family with limit $B^{6, \infty}$ containing the special vector $b^{\infty} = {\bf{0}}$ (i.e. $(b^{\infty})_{ij} = 0$ for $i \leq j \leq i+5,\ 1 \leq i \leq 6$). 
\begin{theorem} \label{perfectcrystal} The family of perfect crystals $\Bfamsix$ forms a coherent family and the crystal $B^{6,\infty}$ is its limit.
\end{theorem}
\begin{proof} Set $J=\{(l,b)| l \in \mZ_{>0}, b \in (B^{6,l})_\text{min}\}$. By (\cite{KKM}, Definition 4.1), we need to show that
\begin{enumerate}
\item wt$(b^{\infty})=0, \veps(b^{\infty})=\vphi(b^{\infty})=0,$
\item for any $(l,b) \in J$, there exists an embedding of crystals
$$f_{(l,b)} : T_{\veps(b)} \otimes B^{6, l} \otimes T_{-\vphi(b)} \longrightarrow B^{6, \infty}$$
where $f_{(l,b)}(t_{\veps(b)} \otimes b \otimes t_{-\vphi(b)}) = b^{\infty}$,
\item $B^{6, \infty} = \cup_{(l,b) \in J}$ Im $f_{(l,b)}$.
\end{enumerate} 

Since $\veps_k(b^{\infty}) = 0, \vphi_k(b^{\infty}) = 0, 0 \leq k \leq 6$, we have $\veps(b^{\infty}) = 0, \vphi(b^{\infty}) = 0$ and hence wt$(b^{\infty}) = 0$ which proves $(1)$.

Let $l \in \mZ_{>0}$ and $b^0 = (b^0_{ij})$ be an element of $(B^{6,l})_{\text{min}}$. Then there exist $a_k \in \mZ_{\geq 0}, 0 \leq k \leq 6$ such that $a_0+a_1+2a_2+2a_3+2a_4+a_5+a_6=l$ and 
\begin{align*}
b^0_{11}    &= a_1+a_2+a_3+a_4+a_6,  \ b^0_{12} =a_5, \ b^0_{13}  = a_4, \ b^0_{14} =a_3, \ b^0_{15}=a_2, \ b^0_{16}    =a_0,  \\ 
b^0_{22}  &= a_2+a_3+a_4+a_6, \ b^0_{23} =a_5, \ b^0_{24} = a_4, \ b^0_{25} = a_3, \ b^0_{26} = a_1+a_2, \ b^0_{27}     = a_0, \\ 
 b^0_{33} 	&= a_3+a_4+a_6,\ b^0_{34} =a_5, \ b^0_{35}=a_4, \ b^0_{36} =a_2+a_3, \ b^0_{37} = a_1,\ b^0_{38}    =a_0+a_2, \\
 b^0_{44} &= a_4+a_6, \ b^0_{45} = a_5, \ b^0_{46} = a_3+a_4,\ b^0_{47} =a_2, \ b^0_{48} =a_1, \ b^0_{49}    = a_0+a_2+a_3, \\
 b^0_{55} &=a_6, \ b^0_{56} = a_4+a_5, \ b^0_{57} = a_3, \ b^0_{58} = a_2, \ b^0_{59} = a_1, \ b^0_{5,10} =a_0+a_2+a_3+a_4, \\
 b^0_{66} &=a_6, \ b^0_{67} = a_4, \ b^0_{68} = a_3, \ b^0_{69} = a_2, \ b^0_{6,10} = a_1, \ b^0_{6,11} =a_0+a_2+a_3+a_4+a_5, \\
\veps(b^0)& = a_6\Lambda_0+a_5\Lambda_1+a_4\Lambda_2+a_3\Lambda_3+a_2\Lambda_4+a_1\Lambda_5+a_0\Lambda_6,\\
\vphi(b^0)&= a_0\Lambda_0+a_1\Lambda_1+a_2\Lambda_2+a_3\Lambda_3+a_4\Lambda_4+a_5\Lambda_5+a_6\Lambda_6. 
\end{align*}

For any $b= (b_{ij}) \in B^{6,l}$, we define a map
$$f_{(l,b^0)} : T_{\veps(b^0)} \otimes B^{6,l} \otimes T_{-\vphi(b^0)} \longrightarrow B^{6,\infty}$$
by $f_{(l,b^0)}( t_{\veps(b^0)} \otimes b \otimes t_{-\vphi(b^0)} )= b' =(b'_{ij}) $
where $b'_{ij} = b_{ij} - b^0_{ij}$ for all $i \leq j \leq i+5,\ 1 \leq i \leq 6$.
Then it is easy to see that
\begin{align*}
\veps_k(b') &= 
\veps_k(b) - a_{6-k} = \veps_k(b) - \langle \check\alpha_k, \veps(b^0)\rangle \ \ \text{for} \; 0 \leq k  \leq6,\\
\vphi_k(b') &= \vphi_k(b) - a_k = \vphi_k(b) + \langle \check\alpha_k, -\vphi(b^0)\rangle \ \ \text{for} \; 0 \leq k  \leq6.
\end{align*}
Hence we have
\begin{align*}
\veps_k(b') &= \veps_k(b) - \langle \check\alpha_k, \veps(b^0)\rangle = \veps_k(t_{\veps(b^0)}\otimes b\otimes t_{-\vphi(b^0)}), \\
\vphi_k(b') &= \vphi_k(b) + \langle \check\alpha_k, -\vphi(b^0)\rangle=\vphi_k(t_{\veps(b^0)}\otimes b\otimes t_{-\vphi(b^0)}),\\
\text{wt}(b') &= \sum_{k=0}^6(\vphi_k(b') - \veps_k(b'))\L_k = \text{wt}(b) +  \sum_{k=0}^6 \langle \check\alpha_k, -\vphi(b^0)\rangle\L_k + \sum_{k=0}^6 \langle \check\alpha_k, \veps(b^0)\rangle\L_k \\
&= \text{wt}(b) - \vphi(b^0) + \veps(b^0) = \text{wt}(t_{\veps(b^0)}\otimes b\otimes t_{-\vphi(b^0)}).
\end{align*}

For $0 \leq k \leq 6, b \in B^{6,l}$, it can be checked easily that the conditions for the action of $\e_k$ on $b' = b - b^0$ hold if and only if the conditions for the action of $\e_k$ on $b$ hold. Hence from the defined action of $\e_k$, we see that $\e_k(b') = \e_k(b) - b^0, 0\leq k \leq6$. This implies that 
\[
f_{(l,b^0)} (\tilde{e_k}( t_{\veps(b^0)} \otimes b \otimes t_{-\vphi(b^0)})) = f_{(l,b^0)} ( t_{\veps(b^0)} \otimes \e_k(b) \otimes t_{-\vphi(b^0)}) \]
\[= \e_k(b) - b^0 = \e_k(b') = \e_k ( f_{(l, b^0)}( t_{\veps(b^0)} \otimes b \otimes t_{-\vphi(b^0)})).
\]
Similarly, we have  $f_{(l,b^0)} (\tilde{f_k}( t_{\veps(b^0)} \otimes b \otimes t_{-\vphi(b^0)})) =  \f_k ( f_{(l, b^0)}( t_{\veps(b^0)} \otimes b \otimes t_{-\vphi(b^0)}))$. Clearly the map $f_{(l,b^0)}$ is injective with $f_{(l,b^0)} ( t_{\veps(b^0)}  \otimes b^0 \otimes t_{-\vphi(b^0)})=b^{\infty}$. This proves (2).

We observe that $ \sum_{j=i}^{i+5} b'_{ij} =  \sum_{j=i}^{i+5} b_{ij} -  \sum_{j=i}^{i+5} b^0_{ij} = l-l = 0$ for all $1 \leq i \leq 6$. Also,
\begin{align*}
b'_{11} &= b_{11} - b^0_{11} \\
	&= b_{66} + b_{67} + b_{68} + b_{69} + b_{6,10} - a_1 - a_2 - a_3 -a_4 - a_6\\
	&= b'_{66} + b'_{67} + b'_{68} + b'_{69} + b'_{6,10},\\
b'_{11} + b'_{12} &= b_{11} - b^0_{11} + b_{12} - b^0_{12} \\
	&= b_{55} + b_{56} + b_{57} + b_{58} + b_{59} - a_1 - a_2 - a_3 - a_4 - a_5 - a_6 \\
	&= b'_{55} + b'_{56} + b'_{57} + b'_{58} + b'_{59},\\
b'_{11} + b'_{12} + b'_{13} &= b_{11} - b^0_{11} + b_{12} - b^0_{12} + b_{13}  - b^0_{13} \\
	&= b_{44} + b_{45} + b_{46} + b_{47} + b_{48} - a_1 - a_2 - a_3 - 2a_4 - a_5 - a_6 \\
	&= b'_{44} + b'_{45} + b'_{46} + b'_{47} + b'_{48},\\
b'_{11} + b'_{12} + b'_{13} + b'_{14} &= b_{11} - b^0_{11}  + b_{12} - b^0_{12} + b_{13} - b^0_{13} + b_{14} - b^0_{14} \\
	&= b_{33} + b_{34} + b_{35} + b_{36} + b_{37} - a_1 - a_2 - 2a_3 - 2a_4 - a_5 - a_6 \\
	&= b'_{33} + b'_{34} + b'_{35} + b'_{36} + b'_{37}, \\
b'_{11} + b'_{12} + b'_{13} + b'_{14}+b'_{15} &= b_{11} - b^0_{11}  + b_{12} - b^0_{12} + b_{13} - b^0_{13} + b_{14} - b^0_{14} + b_{15} - b^0_{15} \\
	&= b_{22} + b_{23} + b_{24} + b_{25} + b_{26} - a_1 - 2a_2 - 2a_3 - 2a_4 - a_5 - a_6 \\
	&= b'_{22} + b'_{23} + b'_{24} + b'_{25} + b'_{26}.
\end{align*}
Similarly, we can show that $\sum_{j=i}^{6-t} b'_{ij} = \sum_{j=i+t}^{5+t} b'_{i+t,j},$\ for $2 \leq i \leq 5,\ 1 \leq t \leq 5$. Hence we have $B^{6,\infty} \supseteq \cup_{(l,b) \in J}$ Im $f_{(l,b)}$. To prove (3) we also need to show that $B^{6,\infty} \subseteq \cup_{(l,b) \in J}$ Im $f_{(l,b)}$. Let $b' = (b'_{ij}) \in B^{6,\infty}$. By (2), we can assume that $b' \neq b^{\infty}$. Set 
\begin{align*}
a_1 	&=   \text{max} \{- b'_{11} + b'_{22} ,- b'_{11}- b'_{12}+b'_{22}+b'_{23},  - b'_{11}- b'_{12}- b'_{13}+b'_{22}+b'_{23}+b'_{24}, \\
	&\hspace{40pt} - b'_{11}- b'_{12}- b'_{13} - b'_{14}+b'_{22}+b'_{23}+b'_{24}+b'_{25}, 0\},\\
a_2 	&= \text{max} \{- b'_{22}+b'_{33} ,  - b'_{22} - b'_{23}+b'_{33}+b'_{34}, - b'_{22} - b'_{23} - b'_{24}+b'_{33}+b'_{34}+b'_{35}, \\
	&\hspace{40pt} - b'_{15}, -b'_{26} - a_1, 0\}, \\
a_3 	&= \text{max} \{- b'_{33}+b'_{44} , - b'_{33} - b'_{34}+b'_{44}+b'_{45},- b'_{14}, - b'_{25}, -b'_{36} - a_2, 0\}, \\
a_4 	&= \text{max} \{ -b'_{44}+b'_{55}, - b'_{13}, - b'_{24}, - b'_{35}, -b'_{46} - a_3, 0\},\\
a_5 	&= \text{max} \{ -b'_{12}, -b'_{23}, -b'_{34}, -b'_{45}, -b'_{56} - a_4, 0 \}, \\
a_6 	&= \text{max} \{ -b'_{11} - a_1 - a_2 - a_3-a_4, -b'_{22} - a_2 - a_3 - a_4, -b'_{33} - a_3 - a_4, -b'_{44} \\		&\hspace{40pt} - a_4, -b'_{55}, 0 \}, \\
a_0 	&= \text{max} \{ b'_{11} - a_2 - a_3 - a_4-a_5, b'_{11}+b'_{12} - a_2 - a_3 -a_4, b'_{11}+b'_{12}+b'_{13} - a_2  \\
	&\hspace{40pt} -a_3,b'_{11}+b'_{12}+b'_{13}+b'_{14} - a_2, b'_{11}+b'_{12}+b'_{13}+b'_{14}+b'_{15}, 0 \}.
\end{align*}
Let $l=a_0+a_1+2a_2+2a_3+2a_4+a_5+a_6$. Let $b^0 = (b^0_{ij})$ where 
\begin{align*}
b^0_{11} &= a_1+a_2+a_3+a_4+a_6,  \ b^0_{12} =a_5, \ b^0_{13}  = a_4, \ b^0_{14} =a_3, \ b^0_{15}=a_2, \ b^0_{16} =a_0, \\
b^0_{22} &= a_2+a_3+a_4+a_6, \ b^0_{23} =a_5, \ b^0_{24} = a_4, a^0_{25} = a_3, \ b^0_{26} = a_1+a_2, \ b^0_{27} = a_0, \\
b^0_{33} &= a_3+a_4+a_6, \ b^0_{34} =a_5, \ b^0_{35}=a_4, \ b^0_{36} =a_2+a_3, \ b^0_{37} =a_1, \ b^0_{38} = a_0+a_2, \\
b^0_{44} &= a_4+a_6, \ b^0_{45} = a_5, \ b^0_{46} = a_3+a_4,\ b^0_{47} =a_2,  \ b^0_{48} =a_1, \ b^0_{49} =a_0+a_2+a_3, \\
b^0_{55} &= a_6, \ b^0_{56} = a_4+a_5, \ b^0_{57} = a_3,\ b^0_{58} =a_2, \ b^0_{5,9} =a_1, \ b^0_{5,10} =a_0+a_2+a_3+a_4, \\
b^0_{66} &=a_6, \ b^0_{67} = a_4, \ b^0_{68} = a_3, \ b^0_{69} = a_2, \ b^0_{6,10} = a_1, \ b^0_{6,11} = =a_0+a_2+a_3+a_4+a_5.
\end{align*}
Then	$\veps(b^0) 	= a_6\Lambda_0+a_5\Lambda_1+a_4\Lambda_2+a_3\Lambda_3+a_2\Lambda_4+a_1\Lambda_5+a_0\Lambda_6$
and 	$\vphi(b^0)	= a_0\Lambda_0+a_1\Lambda_1+a_2\Lambda_2+a_3\Lambda_3+a_4\Lambda_4+a_5\Lambda_5+a_6\Lambda_6$.
It is easy to see that $b^0 \in (B^{6,l})_{\text{min}}$. 

Set $b = (b_{ij})$ where $b_{ij} = b'_{ij} +b^0_{ij}$. Then $\sum_{j=i}^{i+5} b_{ij} = \sum_{j=i}^{i+5} b'_{ij} + \sum_{j=i}^{i+5} b^0_{ij} = 0 + l = l,\ 1 \leq i \leq 6$ and we observe that 
\begin{align*}
b_{11} &= b'_{11} + b^0_{11} = b'_{11} + a_1+a_2+a_3+a_4+a_6 \geq 0, \\
		& \hspace{4.5cm} \text{since} \; a_6 \geq -b'_{11} - a_1 - a_2 - a_3-a_4,\\ 
b_{12} &= b'_{12} + b^0_{12} = b'_{12} + a_5 \geq 0, \ \text{since} \; a_5 \geq - b'_{12},\\ 
b_{13} &= b'_{13} + b^0_{13} = b'_{13} + a_4 \geq 0, \ \text{since} \; a_4 \geq - b'_{13},\\ 
b_{14} &= b'_{14} + b^0_{14} = b'_{14} + a_3 \geq 0, \ \text{since} \; a_3 \geq - b'_{14},\\ 
b_{15} &= b'_{15} + b^0_{15} = b'_{14} + a_2 \geq 0, \ \text{since} \; a_2 \geq - b'_{15},\\ 
b_{16} &= b'_{16} + b^0_{16} = b'_{15} + a_0 = - b'_{11} - b'_{12} - b'_{13} - b'_{14} - b'_{15} + a_0 \geq 0, \\
	& \hspace{4.5cm} \text{since} \; a_0 \geq b'_{11}+b'_{12}+b'_{13}+b'_{14} +b'_{15}.
\end{align*}
Similarly, we can show that $b_{ij} \in \mZ_{\geq 0}$ for $i \leq j \leq i+5,\ 2 \leq i \leq 6$. We also have
\begin{align*}
b_{44} &= b'_{44} + b^0_{44} = b'_{66} + b'_{67} + a_4  + a_6 = b_{66} + b_{67},\\
b_{44} + b_{45} &= b'_{44} + b^0_{44} + b'_{45} + b^0_{45} = b'_{55} + b'_{56} +  a_4 + a_5 + a_6 = b_{55} + b_{56},\\
b_{55} &= b'_{55} + b^0_{55} = b'_{66} + a_6 = b_{66}.
\end{align*}
Similarly, we see that $\sum_{j=i}^{6-t} b_{ij} = \sum_{j=i+t}^{5+t} b_{i+t,j},\ 1 \leq i \leq 3, 1 \leq t \leq 5$. \ Also,
\begin{align*}
b_{11} &= b'_{11} + b^0_{11} \geq b'_{22} + b^0_{22} = b_{22},\ \text{since} \; b^0_{11} - b^0_{22} = a_1 \geq - b'_{11} + b'_{22},\\ 
b_{22} &= b'_{22} + b^0_{22} \geq b'_{33} + b^0_{33} = b_{33},\ \text{since} \; b^0_{22} - b^0_{33} = a_2 \geq - b'_{22} + b'_{33},\\
b_{33} &= b'_{33} + b^0_{33} \geq b'_{44} + b^0_{44} = b_{44},\ \text{since} \; b^0_{33} - b^0_{44} = a_3 \geq - b'_{33} + b'_{44}, \\
b_{44} &= b'_{44} + b^0_{44} \geq b'_{55} + b^0_{55} = b_{55},\ \text{since} \; b^0_{44} - b^0_{55} = a_4 \geq - b'_{44} + b'_{55} .
\end{align*}
Similarly, $\sum_{j=i}^{t} b_{ij} \geq \sum_{j=i+1}^{t+1} b_{i+1,j}, 1 \leq i < t \leq 5$. Hence $b \in B^{6,l}$.

Then $f_{(l,b^0)}(t_{\veps(b^0)} \otimes b \otimes t_{-\vphi(b^0)}) = b',$ and 
$b' \in \cup_{(l,b) \in J}$ Im $f_{(l,b)}$ which proves (3).
\end{proof}

\section{Ultra-discretization of \bf{$\mathcal{V}(D_6^{(1)})$}} 
It is known that the ultra-discretization of a positive geometric crystal is a Kashiwara crystal \cite{BK, N}. In this section we apply the ultra-discretization functor $\mathcal{UD}$ to the positive geometric crystal $\mathcal{V}=\mathcal{V}(D_6^{(1)})$ for the affine Lie algebra $D_6^{(1)}$ at the spin node $k=6$ in  (\cite{MP},Theorem 5.1). Then we show that as crystal it is isomorphic to the crystal $B^{6, \infty}$ given in the last section which proves the conjecture in \cite{KNO} for this case.
As a set $\mathcal{X}=\mathcal{UD}(\mathcal{V}) = \mZ^{15}$. We denote the variables $x_m^{(l)}$ in $\cV$ by the same notation $x_m^{(l)}$ in 
$\mathcal{UD}(\mathcal{V}) = \mathcal{X}$.

Let $x=(x_6^{(3)},x_4^{(4)}, x_3^{(3)}, x_2^{(2)}, x_5^{(2)}, x_4^{(3)}, x_3^{(2)}, x_6^{(2)}, x_4^{(2)}, x_5^{(1)},x_1^{(1)}, x_2^{(1)}, x_3^{(1)}, x_4^{(1)}, x_6^{(1)}) \\ \in \mathcal{X}$. By applying the ultra-discretization functor $\mathcal{UD}$ to the positive geometric crystal $\cV$ in (\cite{MP},Theorem 5.1), we have for $0 \leq k \leq 6$:
\begin{align*}
\mathcal{UD}(\gamma_k)(x) &=
 \small \begin{cases}
-x_2^{(2)}-x_2^{(1)}, &k=0,\\
2x_1^{(1)}-x_2^{(2)}-x_2^{(1)}, &k=1,\\
-x_1^{(1)}+2x_2^{(2)}+2x_2^{(1)} -x_3^{(3)}-x_3^{(2)}-x_3^{(1)}, &k=2,\\
-x_2^{(2)}-x_2^{(1)}+2x_3^{(3)}+2x_3^{(2)}+2x_3^{(1)}-x_4^{(4)}-x_4^{(3)}&\\
	\hspace{15pt }-x_4^{(2)}-x_4^{(1)}, &k=3,\\
-x_3^{(3)}-x_3^{(2)}-x_3^{(1)}+2x_4^{(4)}+2x_4^{(3)}+2x_4^{(2)}+2x_4^{(1)} &\\
	\hspace{15pt}  -x_5^{(2)}-x_5^{(1)}-x_6^{(3)}-x_6^{(2)}-x_6^{(1)}, &k=4,\\
-x_4^{(4)}-x_4^{(3)}-x_4^{(2)}-x_4^{(1)}+2x_5^{(2)}+2x_5^{(1)}, &k=5, \\
-x_4^{(4)}-x_4^{(3)}-x_4^{(2)}-x_4^{(1)}+2x_6^{(3)}+2x_6^{(2)}+2x_6^{(1)}, &k=6.
\end{cases} \\
\mathcal{UD}(\veps_k)(x) &= 
\small \begin{cases}
\text{max} \{x_6^{(1)}, x_2^{(2)}+x_2^{(1)}-x_3^{(3)}-x_3^{(2)}+x_5^{(1)}, x_2^{(2)}-x_3^{(3)} &\\
 	\hspace{15pt} +x_3^{(1)}-x_4^{(2)}+x_5^{(1)},x_2^{(2)}-x_3^{(3)}+x_4^{(1)}, x_3^{(3)}+x_3^{(2)}&\\
	\hspace{15pt} -x_4^{(3)}-x_4^{(2)}+x_5^{(1)}, x_3^{(2)}-x_4^{(3)}+x_4^{(1)},x_4^{(2)}+x_4^{(1)} &\\
	 \hspace{15pt} -x_6^{(2)}, x_2^{(2)}+x_2^{(1)}-x_6^{(3)} , x_2^{(2)}+x_2^{(1)}+x_6^{(2)}-x_4^{(4)} &\\
	  \hspace{15pt} -x_4^{(3)},x_2^{(2)}+x_2^{(1)}-x_3^{(2)}-x_4^{(4)}+x_4^{(2)},x_2^{(2)}+x_2^{(1)}&\\
	  \hspace{15pt} -x_3^{(3)}-x_3^{(2)}+x_4^{(3)}+x_4^{(2)}-x_5^{(2)},x_2^{(2)}+x_3^{(1)}-x_4^{(4)},&\\
	  \hspace{15pt} x_2^{(2)}-x_3^{(3)}+x_3^{(1)}+x_4^{(3)}-x_5^{(2)}, x_3^{(2)}+x_3^{(1)}-x_5^{(2)}\}, &k=0,\\
-x_1^{(1)}+x_2^{(2)}, &k=1,\\
\text{max} \{-x_2^{(2)}+x_3^{(3)}, x_1^{(1)}-2x_2^{(2)}-x_2^{(1)}+x_3^{(3)}+x_3^{(2)}\}, &k=2,\\
\text{max} \{-x_3^{(3)}+x_4^{(4)}, x_2^{(2)}-2x_3^{(3)}-x_3^{(2)}+x_4^{(4)}+x_4^{(3)}, &\\
	\hspace{15pt}x_2^{(2)}+x_2^{(1)}-2x_3^{(3)}-2x_3^{(2)}-x_3^{(1)}+x_4^{(4)}+x_4^{(3)}+x_4^{(2)}\}, &k=3,\\
	\end{cases}
	\\
\mathcal{UD}(\veps_k)(x) &= : 
	\small \begin{cases}
\text{max} \{-x_4^{(4)}+x_6^{(3)}, x_3^{(3)}-2x_4^{(4)}-x_4^{(3)}+x_5^{(2)}+x_6^{(3)}, &\\
	\hspace{15pt} x_3^{(3)}+x_3^{(2)}-2x_4^{(4)}-2x_4^{(3)}-x_4^{(2)}+x_5^{(2)}+x_6^{(3)}+x_6^{(2)},&\\
	 \hspace{15pt}  x_3^{(3)}+x_3^{(2)}+x_3^{(1)} -2x_4^{(4)}-2x_4^{(3)}-2x_4^{(2)}-x_4^{(1)}+x_5^{(2)}&\\
	 \hspace{15pt} +x_5^{(1)}+x_6^{(3)}+x_6^{(2)} \}, &k=4,\\
\text{max} \{x_4^{(4)}-x_5^{(2)}, x_4^{(4)}+x_4^{(3)}+x_4^{(2)}-2x_5^{(2)}-x_5^{(1)}\}, &k=5, \\
\text{max} \{-x_6^{(3)}, x_4^{(4)}+x_4^{(3)}-2x_6^{(3)}-x_6^{(2)}, x_4^{(4)}+x_4^{(3)}+x_4^{(2)} &\\
	\hspace{15pt} +x_4^{(1)}-2x_6^{(3)}-2x_6^{(2)}-x_6^{(1)}\}, &k=6.
\end{cases}
\end{align*}
We define
\begin{align*}
\breve{c_2} &=\text{max}\{c+x_2^{(2)}+x_2^{(1)}, x_3^{(2)}+x_1^{(1)}\}-\text{max}\{x_2^{(2)}+x_2^{(1)}, x_3^{(2)}+x_1^{(1)}\},\\
\breve{c_{3_1}} &=\text{max}\{c+x_3^{(3)}+2x_3^{(2)}+x_3^{(1)}, x_2^{(2)}+x_3^{(2)}+x_3^{(1)}+x_4^{(3)},x_2^{(2)}+x_2^{(1)}+x_4^{(3)}+x_4^{(2)}\}\\
			& \hspace{15pt} - \text{max}\{x_3^{(3)}+2x_3^{(2)}+x_3^{(1)}, x_2^{(2)}+x_3^{(2)}+x_3^{(1)}+x_4^{(3)},x_2^{(2)}+x_2^{(1)}+x_4^{(3)}+x_4^{(2)}\},\\
\breve{c_{3_2}} &=\text{max}\{c+x_3^{(3)}+2x_3^{(2)}+x_3^{(1)}, c+x_2^{(2)}+x_3^{(2)}+x_3^{(1)}+x_4^{(3)},x_2^{(2)}+x_2^{(1)}+x_4^{(3)}\\
			& \hspace{15pt} +x_4^{(2)}\}-\text{max}\{c+x_3^{(3)}+2x_3^{(2)}+x_3^{(1)}, x_2^{(2)}+x_3^{(2)}+x_3^{(1)}+x_4^{(3)},x_2^{(2)}+x_2^{(1)}\\
			& \hspace{15pt} +x_4^{(3)}+x_4^{(2)}\},\\
\breve{c_{4_1}} &=\text{max}\{c+x_4^{(4)}+2x_4^{(3)}+2x_4^{(2)}+x_4^{(1)}, x_3^{(3)}+x_4^{(3)}+2x_4^{(2)}+x_4^{(1)}+x_5^{(2)}, x_3^{(3)} \\
			& \hspace{15pt} +x_3^{(2)}+x_4^{(2)}+x_4^{(1)}+x_5^{(2)}+x_6^{(2)}, x_3^{(3)}+x_3^{(2)}+x_3^{(1)}+x_5^{(2)}+x_5^{(1)}+x_6^{(2)}\}\\
			& \hspace{15pt} - \text{max}\{x_4^{(4)}+2x_4^{(3)}+2x_4^{(2)}+x_4^{(1)}, x_3^{(3)}+x_4^{(3)}+2x_4^{(2)}+x_4^{(1)}+x_5^{(2)}, x_3^{(3)}+x_3^{(2)} \\
			& \hspace{15pt}+x_4^{(2)}+x_4^{(1)}+x_5^{(2)}+x_6^{(2)}, x_3^{(3)}+x_3^{(2)}+x_3^{(1)}+x_5^{(2)}+x_5^{(1)}+x_6^{(2)}\}\\
\breve{c_{4_2}} &=\text{max}\{c+x_4^{(4)}+2x_4^{(3)}+2x_4^{(2)}+x_4^{(1)}, c+x_3^{(3)}+x_4^{(3)}+2x_4^{(2)}+x_4^{(1)}+x_5^{(2)}, x_3^{(3)}+ \\
			& \hspace{15pt} x_3^{(2)}+x_4^{(2)}+x_4^{(1)}+x_5^{(2)}+x_6^{(2)}, x_3^{(3)}+x_3^{(2)}+x_3^{(1)}+x_5^{(2)}+x_5^{(1)}+x_6^{(2)}\}-\\
			& \hspace{15pt} \text{max}\{c+x_4^{(4)}+2x_4^{(3)}+2x_4^{(2)}+x_4^{(1)}, x_3^{(3)}+x_4^{(3)}+2x_4^{(2)}+x_4^{(1)}+x_5^{(2)}, \\
			& \hspace{15pt} x_3^{(3)}+x_3^{(2)} +x_4^{(2)}+x_4^{(1)}+x_5^{(2)}+x_6^{(2)}, x_3^{(3)}+x_3^{(2)}+x_3^{(1)}+x_5^{(2)}+x_5^{(1)}+x_6^{(2)}\}\\
\breve{c_{4_3}} &=\text{max}\{c+x_4^{(4)}+2x_4^{(3)}+2x_4^{(2)}+x_4^{(1)}, c+x_3^{(3)}+x_4^{(3)}+2x_4^{(2)}+x_4^{(1)}+x_5^{(2)},  \\
			& \hspace{15pt} c+x_3^{(3)}+x_3^{(2)}+x_4^{(2)}+x_4^{(1)}+x_5^{(2)}+x_6^{(2)}, x_3^{(3)}+x_3^{(2)}+x_3^{(1)}+x_5^{(2)}+x_5^{(1)}\\
			& \hspace{15pt} +x_6^{(2)}\}-\text{max}\{c+x_4^{(4)}+2x_4^{(3)}+2x_4^{(2)}+x_4^{(1)}, c+x_3^{(3)}+x_4^{(3)}+2x_4^{(2)}+x_4^{(1)} \\
			& \hspace{15pt}  +x_5^{(2)},x_3^{(3)}+x_3^{(2)}+x_4^{(2)}+x_4^{(1)}+x_5^{(2)}+x_6^{(2)}, x_3^{(3)}+x_3^{(2)}+x_3^{(1)}+x_5^{(2)}+x_5^{(1)}\\
			& \hspace{15pt} +x_6^{(2)}\}\\
\breve{c_5} &=\text{max}\{c+x_5^{(2)}+x_5^{(1)}, x_4^{(3)}+x_4^{(2)}\}-\text{max}\{x_5^{(2)}+x_5^{(1)}, x_4^{(3)}+x_4^{(2)}\},\\
\breve{c_{6_1}} &=\text{max}\{c+x_6^{(3)}+2x_6^{(2)}+x_6^{(1)}, x_4^{(4)}+x_4^{(3)}+x_6^{(2)}+x_6^{(1)},x_4^{(4)}+x_4^{(3)}+x_4^{(2)}+x_4^{(1)}\}\\
			& \hspace{15pt} - \text{max}\{x_6^{(3)}+2x_6^{(2)}+x_6^{(1)}, x_4^{(4)}+x_4^{(3)}+x_6^{(2)}+x_6^{(1)},x_4^{(4)}+x_4^{(3)}+x_4^{(2)}+x_4^{(1)}\},\\
\breve{c_{6_2}} &=\text{max}\{c+x_6^{(3)}+2x_6^{(2)}+x_6^{(1)}, c+x_4^{(4)}+x_4^{(3)}+x_6^{(2)}+x_6^{(1)},x_4^{(4)}+x_4^{(3)}+x_4^{(2)}\\
			& \hspace{15pt} +x_4^{(1)}\} - \text{max}\{c+x_6^{(3)}+2x_6^{(2)}+x_6^{(1)}, x_4^{(4)}+x_4^{(3)}+x_6^{(2)}+x_6^{(1)},x_4^{(4)}+x_4^{(3)}\\
			& \hspace{15pt} +x_4^{(2)}+x_4^{(1)}\},		\\
\breve{K}		& = \text{max} \{x_6^{(1)}, x_2^{(2)}+x_2^{(1)}-x_3^{(3)}-x_3^{(2)}+x_5^{(1)}, x_2^{(2)}-x_3^{(3)}+x_3^{(1)}-x_4^{(2)}+x_5^{(1)},\\
 			&\hspace{15pt} x_2^{(2)}-x_3^{(3)}+x_4^{(1)}, x_3^{(3)}+x_3^{(2)}-x_4^{(3)}-x_4^{(2)}+x_5^{(1)}, x_3^{(2)}-x_4^{(3)}+x_4^{(1)},\\
			&\hspace{15pt} x_4^{(2)}+x_4^{(1)}-x_6^{(2)}, x_2^{(2)}+x_2^{(1)}-x_6^{(3)} , x_2^{(2)}+x_2^{(1)}+x_6^{(2)}-x_4^{(4)}-x_4^{(3)}, \\
			 &\hspace{15pt} x_2^{(2)}+x_2^{(1)}-x_3^{(2)}-x_4^{(4)}+x_4^{(2)},x_2^{(2)}+x_2^{(1)}-x_3^{(3)}-x_3^{(2)}+x_4^{(3)}+x_4^{(2)}-x_5^{(2)},\\
	 		 &\hspace{15pt} x_2^{(2)}+x_3^{(1)}-x_4^{(4)}, x_2^{(2)}-x_3^{(3)}+x_3^{(1)}+x_4^{(3)}-x_5^{(2)}, x_3^{(2)}+x_3^{(1)}-x_5^{(2)}\}, 
\end{align*}
Then we have
\begin{align*}
\mathcal{UD}(e_k^{c})(x) &=
\begin{cases}
(x_6^{(3)'},x_4^{(4)'}, x_3^{(3)'}, x_2^{(2)}-c, x_5^{(2)'}, x_4^{(3)'}, x_3^{(2)'}, x_6^{(2)'}, x_4^{(2)'}, x_5^{(1)'},&\\
	\hspace{15pt} x_1^{(1)}-c, x_2^{(1)}-c, x_3^{(1)'}, x_4^{(1)'}, x_6^{(1)'}), &k=0,\\
(x_6^{(3)},x_4^{(4)}, x_3^{(3)}, x_2^{(2)}, x_5^{(2)}, x_4^{(3)}, x_3^{(2)}, x_6^{(2)}, x_4^{(2)}, x_5^{(1)}, x_1^{(1)}+c,&\\
	\hspace{15pt}  x_2^{(1)}, x_3^{(1)},x_4^{(1)}, x_6^{(1)}), &k=1,\\
(x_6^{(3)},x_4^{(4)}, x_3^{(3)}, x_2^{(2)}+\breve{c_2}, x_5^{(2)}, x_4^{(3)}, x_3^{(2)}, x_6^{(2)}, x_4^{(2)}, x_5^{(1)},&\\
 	\hspace{15pt} x_1^{(1)}, x_2^{(1)}+c-\breve{c_2}, x_3^{(1)}, x_4^{(1)}, x_6^{(1)}), &k=2,\\
(x_6^{(3)},x_4^{(4)}, x_3^{(3)}+\breve{c_{3_1}}, x_2^{(2)}, x_5^{(2)}, x_4^{(3)}, x_3^{(2)}+\breve{c_{3_2}}, x_6^{(2)},x_4^{(2)},   &\\
	\hspace{15pt} x_5^{(1)},x_1^{(1)}, x_2^{(1)}, x_3^{(1)}+c-\breve{c_{3_1}}-\breve{c_{3_2}}, x_4^{(1)}, x_6^{(1)}), &k=3,\\ 
(x_6^{(3)},x_4^{(4)}+\breve{c_{4_1}}, x_3^{(3)}, x_2^{(2)}, x_5^{(2)}, x_4^{(3)}+\breve{c_{4_2}}, x_3^{(2)}, x_6^{(2)}, x_4^{(2)}&\\ 
	\hspace{15pt} +\breve{c_{4_3}},x_5^{(1)},x_1^{(1)}, x_2^{(1)}, x_3^{(1)}, x_4^{(1)}+c-\breve{c_{4_1}}-\breve{c_{4_2}}-\breve{c_{4_3}}, x_6^{(1)}), &k=4,\\ 
(x_6^{(3)},x_4^{(4)}, x_3^{(3)}, x_2^{(2)}, x_5^{(2)}+\breve{c_5}, x_4^{(3)}, x_3^{(2)}, x_6^{(2)}, x_4^{(2)},  x_5^{(1)}&\\
	\hspace{15pt} +c-\breve{c_5},x_1^{(1)}, x_2^{(1)},x_3^{(1)}, x_4^{(1)}, x_6^{(1)}), &k=5,\\ 
(x_6^{(3)}+\breve{c_{6_1}},x_4^{(4)}, x_3^{(3)}, x_2^{(2)}, x_5^{(2)}, x_4^{(3)}, x_3^{(2)}, x_6^{(2)}+\breve{c_{6_2}}, x_4^{(2)}, &\\
	\hspace{15pt} x_5^{(1)},x_1^{(1)}, x_2^{(1)}, x_3^{(1)}, x_4^{(1)}, x_6^{(1)}+c-\breve{c_{6_1}}-\breve{c_{6_2}}), &k=6,\\ 
\end{cases} \\
\text{where} &\\
x_3^{(1)'} &=x_3^{(1)} + \breve{K} - \text{max} \{c+x_6^{(1)}, x_2^{(2)}+x_2^{(1)}-x_3^{(3)}-x_3^{(2)}+x_5^{(1)}, c+x_2^{(2)}\\
 			&\hspace{15pt} -x_3^{(3)}+x_3^{(1)}-x_4^{(2)}+x_5^{(1)},c+x_2^{(2)}-x_3^{(3)}+x_4^{(1)}, c+x_3^{(3)}+x_3^{(2)}-x_4^{(3)}\\
			&\hspace{15pt} -x_4^{(2)}+x_5^{(1)}, c+x_3^{(2)}-x_4^{(3)}+x_4^{(1)}, c+x_4^{(2)}+x_4^{(1)}-x_6^{(2)}, x_2^{(2)}+x_2^{(1)}\\
			&\hspace{15pt} -x_6^{(3)} , x_2^{(2)}+x_2^{(1)}+x_6^{(2)}-x_4^{(4)}-x_4^{(3)}, x_2^{(2)}+x_2^{(1)}-x_3^{(2)}-x_4^{(4)}+x_4^{(2)},\\
			&\hspace{15pt} x_2^{(2)}+x_2^{(1)}-x_3^{(3)}-x_3^{(2)}+x_4^{(3)}+x_4^{(2)}-x_5^{(2)},c+x_2^{(2)}+x_3^{(1)}-x_4^{(4)}, c\\
			&\hspace{15pt} +x_2^{(2)} -x_3^{(3)}+x_3^{(1)}+x_4^{(3)}-x_5^{(2)}, c+x_3^{(2)}+x_3^{(1)}-x_5^{(2)}\}, \\
x_3^{(2)'} &=-c + x_3^{(2)}  + \text{max} \{c+x_6^{(1)}, x_2^{(2)}+x_2^{(1)}-x_3^{(3)}-x_3^{(2)}+x_5^{(1)}, c+x_2^{(2)}\\
			&\hspace{15pt} -x_3^{(3)}+x_3^{(1)}-x_4^{(2)}+x_5^{(1)},c+x_2^{(2)}-x_3^{(3)}+x_4^{(1)}, c+x_3^{(3)}+x_3^{(2)}-x_4^{(3)}\\
			&\hspace{15pt} -x_4^{(2)}+x_5^{(1)}, c+x_3^{(2)}-x_4^{(3)}+x_4^{(1)},c+x_4^{(2)}+x_4^{(1)}-x_6^{(2)}, x_2^{(2)}+x_2^{(1)}\\
			&\hspace{15pt} -x_6^{(3)} , x_2^{(2)}+x_2^{(1)}+x_6^{(2)}-x_4^{(4)}-x_4^{(3)}, x_2^{(2)}+x_2^{(1)}-x_3^{(2)}-x_4^{(4)}+x_4^{(2)},\\
			&\hspace{15pt} x_2^{(2)}+x_2^{(1)}-x_3^{(3)}-x_3^{(2)}+x_4^{(3)}+x_4^{(2)}-x_5^{(2)}, c+x_2^{(2)}+x_3^{(1)}-x_4^{(4)}, \\
			&\hspace{15pt} c+x_2^{(2)}-x_3^{(3)}+x_3^{(1)}+x_4^{(3)}-x_5^{(2)}, c+x_3^{(2)}+x_3^{(1)}-x_5^{(2)}\} - \text{max} \{c\\
			&\hspace{15pt} +x_6^{(1)}, c+x_2^{(2)}+x_2^{(1)}-x_3^{(3)}-x_3^{(2)}+x_5^{(1)}, x_2^{(2)}-x_3^{(3)}+x_3^{(1)}-x_4^{(2)}\\
			&\hspace{15pt} +x_5^{(1)},x_2^{(2)}-x_3^{(3)}+x_4^{(1)}, c+x_3^{(3)}+x_3^{(2)}-x_4^{(3)}-x_4^{(2)}+x_5^{(1)}, c+x_3^{(2)}\\
			&\hspace{15pt} -x_4^{(3)}+x_4^{(1)},c+x_4^{(2)}+x_4^{(1)}-x_6^{(2)}, c+x_2^{(2)}+x_2^{(1)}-x_6^{(3)} , c+x_2^{(2)}\\
			&\hspace{15pt} +x_2^{(1)}+x_6^{(2)}-x_4^{(4)}-x_4^{(3)}, c+x_2^{(2)}+x_2^{(1)}-x_3^{(2)}-x_4^{(4)}+x_4^{(2)},c+x_2^{(2)}\\
			&\hspace{15pt} +x_2^{(1)} -x_3^{(3)}-x_3^{(2)}+x_4^{(3)}+x_4^{(2)}-x_5^{(2)}, x_2^{(2)}+x_3^{(1)}-x_4^{(4)}, x_2^{(2)}-x_3^{(3)}\\
			&\hspace{15pt} +x_3^{(1)}+x_4^{(3)}-x_5^{(2)}, c+x_3^{(2)}+x_3^{(1)}-x_5^{(2)}\}, \\
x_3^{(3)'} &=-c + x_3^{(3)}  + \text{max} \{c+x_6^{(1)}, c+x_2^{(2)}+x_2^{(1)}-x_3^{(3)}-x_3^{(2)}+x_5^{(1)}, x_2^{(2)}\\
			&\hspace{15pt} -x_3^{(3)}+x_3^{(1)}-x_4^{(2)}+x_5^{(1)},x_2^{(2)}-x_3^{(3)}+x_4^{(1)}, c+x_3^{(3)}+x_3^{(2)}-x_4^{(3)}\\
			&\hspace{15pt} -x_4^{(2)}+x_5^{(1)}, c+x_3^{(2)}-x_4^{(3)}+x_4^{(1)}, c+x_4^{(2)}+x_4^{(1)}-x_6^{(2)}, c+x_2^{(2)}\\
			&\hspace{15pt} +x_2^{(1)} -x_6^{(3)} , c+x_2^{(2)}+x_2^{(1)}+x_6^{(2)}-x_4^{(4)}-x_4^{(3)},  c+x_2^{(2)}+x_2^{(1)}\\
			&\hspace{15pt} -x_3^{(2)}-x_4^{(4)}+x_4^{(2)}, c+x_2^{(2)}+x_2^{(1)}-x_3^{(3)}-x_3^{(2)}+x_4^{(3)}+x_4^{(2)}-x_5^{(2)},\\
			&\hspace{15pt}  x_2^{(2)}+x_3^{(1)}-x_4^{(4)}, x_2^{(2)} -x_3^{(3)}+x_3^{(1)}+x_4^{(3)}-x_5^{(2)}, c+x_3^{(2)}+x_3^{(1)}\\
			&\hspace{15pt} -x_5^{(2)}\} -  \breve{K},\\ 	
x_4^{(1)'} &=x_4^{(1)} + \breve{K} - \text{max} \{c+ x_6^{(1)}, x_2^{(2)}+x_2^{(1)}-x_3^{(3)}-x_3^{(2)}+x_5^{(1)}, x_2^{(2)}-x_3^{(3)}\\
			&\hspace{15pt} +x_3^{(1)}-x_4^{(2)}+x_5^{(1)},  c+x_2^{(2)}-x_3^{(3)}+x_4^{(1)}, x_3^{(3)}+x_3^{(2)}-x_4^{(3)}-x_4^{(2)}\\
			&\hspace{15pt} +x_5^{(1)}, c+x_3^{(2)}-x_4^{(3)}+x_4^{(1)}, c+x_4^{(2)}+x_4^{(1)}-x_6^{(2)}, x_2^{(2)}+x_2^{(1)}-x_6^{(3)} ,\\
			&\hspace{15pt}  x_2^{(2)}+x_2^{(1)}+x_6^{(2)}-x_4^{(4)}-x_4^{(3)}, x_2^{(2)}+x_2^{(1)}-x_3^{(2)}-x_4^{(4)}+x_4^{(2)},x_2^{(2)}\\
			&\hspace{15pt} +x_2^{(1)}-x_3^{(3)}-x_3^{(2)}+x_4^{(3)}+x_4^{(2)}-x_5^{(2)}, x_2^{(2)}+x_3^{(1)}-x_4^{(4)}, x_2^{(2)}-x_3^{(3)}\\
			&\hspace{15pt} +x_3^{(1)}+x_4^{(3)}-x_5^{(2)}, x_3^{(2)}+x_3^{(1)}-x_5^{(2)}\}, \\	
x_4^{(2)'} &=x_4^{(2)} + \breve{K} + \text{max} \{c+ x_6^{(1)}, x_2^{(2)}+x_2^{(1)}-x_3^{(3)}-x_3^{(2)}+x_5^{(1)}, x_2^{(2)}-x_3^{(3)}\\
			&\hspace{15pt} +x_3^{(1)}-x_4^{(2)}+x_5^{(1)},c+x_2^{(2)}-x_3^{(3)}+x_4^{(1)}, x_3^{(3)}+x_3^{(2)}-x_4^{(3)}-x_4^{(2)}\\
			&\hspace{15pt} +x_5^{(1)}, c+x_3^{(2)}-x_4^{(3)}+x_4^{(1)}, c+x_4^{(2)}+x_4^{(1)}-x_6^{(2)}, x_2^{(2)}+x_2^{(1)}-x_6^{(3)} ,\\
			&\hspace{15pt}  x_2^{(2)}+x_2^{(1)}+x_6^{(2)}-x_4^{(4)}-x_4^{(3)},  x_2^{(2)}+x_2^{(1)}-x_3^{(2)}-x_4^{(4)}+x_4^{(2)},x_2^{(2)}\\
			&\hspace{15pt} +x_2^{(1)}-x_3^{(3)}-x_3^{(2)}+x_4^{(3)}+x_4^{(2)}-x_5^{(2)}, x_2^{(2)}+x_3^{(1)}-x_4^{(4)}, x_2^{(2)}-x_3^{(3)}\\
			&\hspace{15pt} +x_3^{(1)}+x_4^{(3)}-x_5^{(2)}, x_3^{(2)}+x_3^{(1)}-x_5^{(2)}\} -  \text{max} \{c+ x_6^{(1)} + \text{max} \{ c\\
			&\hspace{15pt} +x_6^{(1)},c+x_2^{(2)}+x_2^{(1)}-x_3^{(3)}-x_3^{(2)}+x_5^{(1)},c+x_2^{(2)}-x_3^{(3)}+x_3^{(1)}\\
			&\hspace{15pt} -x_4^{(2)} +x_5^{(1)},c+x_2^{(2)}-x_3^{(3)}+x_4^{(1)}, c+x_3^{(3)}+x_3^{(2)}-x_4^{(3)}-x_4^{(2)}\\
			&\hspace{15pt} +x_5^{(1)}, c+x_3^{(2)} -x_4^{(3)}+x_4^{(1)}, c+x_4^{(2)}+x_4^{(1)}-x_6^{(2)}, x_2^{(2)}+x_2^{(1)}-x_6^{(3)} , \\
			&\hspace{15pt}  x_2^{(2)} +x_2^{(1)}+x_6^{(2)}-x_4^{(4)}-x_4^{(3)}, c+x_2^{(2)}+x_2^{(1)}-x_3^{(2)}-x_4^{(4)}+x_4^{(2)},c\\
			&\hspace{15pt} +x_2^{(2)}+x_2^{(1)}- x_3^{(3)}-x_3^{(2)}+x_4^{(3)}+x_4^{(2)}-x_5^{(2)}, c+x_2^{(2)}+x_3^{(1)}-x_4^{(4)}, \\
			&\hspace{15pt} c+x_2^{(2)}-x_3^{(3)}+x_3^{(1)}+x_4^{(3)}- x_5^{(2)}, c+x_3^{(2)}+x_3^{(1)}-x_5^{(2)}\}, c+x_2^{(2)}\\
			&\hspace{15pt} +x_2^{(1)} -x_3^{(3)}-x_3^{(2)}+x_5^{(1)} + \breve{K}, c+x_2^{(2)}-x_3^{(3)}+x_3^{(1)}-x_4^{(2)}+x_5^{(1)} \\
			&\hspace{15pt} + \breve{K}, c+x_2^{(2)}-x_3^{(3)}+x_4^{(1)} + \breve{K}, c+x_3^{(3)}+x_3^{(2)}-x_4^{(3)}-x_4^{(2)}+ x_5^{(1)} \\
			&\hspace{15pt}+ \breve{K}, c+x_3^{(2)}-x_4^{(3)}+x_4^{(1)} + \text{max} \{x_6^{(1)}, x_2^{(2)}+x_2^{(1)}-x_3^{(3)}-x_3^{(2)}\\
			&\hspace{15pt} +x_5^{(1)}, x_2^{(2)}-x_3^{(3)}+x_3^{(1)}-x_4^{(2)}+x_5^{(1)}, x_2^{(2)}-x_3^{(3)}+x_4^{(1)}, x_3^{(3)}+x_3^{(2)}\\
			&\hspace{15pt} -x_4^{(3)}-x_4^{(2)}+x_5^{(1)},  x_3^{(2)}-x_4^{(3)}+x_4^{(1)}, x_4^{(2)}+x_4^{(1)}-x_6^{(2)}, x_2^{(2)}+x_2^{(1)}\\
			&\hspace{15pt} -x_6^{(3)} , x_2^{(2)}+x_2^{(1)}+x_6^{(2)}-x_4^{(4)}- x_4^{(3)}, c+x_2^{(2)}+x_2^{(1)}-x_3^{(2)}-x_4^{(4)}\\
			&\hspace{15pt} +x_4^{(2)},x_2^{(2)}+x_2^{(1)}-x_3^{(3)}-x_3^{(2)}+x_4^{(3)}+x_4^{(2)}-x_5^{(2)}, x_2^{(2)}+x_3^{(1)}-x_4^{(4)},\\
			&\hspace{15pt} x_2^{(2)}-x_3^{(3)}+x_3^{(1)}+x_4^{(3)}-x_5^{(2)}, x_3^{(2)}+x_3^{(1)}-x_5^{(2)}\}, c+x_4^{(2)}+x_4^{(1)}\\
			&\hspace{15pt} -x_6^{(2)} + \text{max} \{ c+x_6^{(1)}, c+x_2^{(2)}+x_2^{(1)}-x_3^{(3)}-x_3^{(2)}+x_5^{(1)}, c+x_2^{(2)}\\
			&\hspace{15pt} -x_3^{(3)}+x_3^{(1)}-x_4^{(2)}+x_5^{(1)}, c+x_2^{(2)}-x_3^{(3)}+x_4^{(1)}, c+x_3^{(3)}+x_3^{(2)}-x_4^{(3)}\\
			&\hspace{15pt} -x_4^{(2)}+x_5^{(1)}, c+x_3^{(2)}-x_4^{(3)}+x_4^{(1)}, c+x_4^{(2)}+x_4^{(1)}-x_6^{(2)}, x_2^{(2)}+x_2^{(1)}\\
			&\hspace{15pt} -x_6^{(3)} , x_2^{(2)}+x_2^{(1)}+x_6^{(2)}-x_4^{(4)}-x_4^{(3)},  c+x_2^{(2)}+x_2^{(1)}-x_3^{(2)}-x_4^{(4)}\\
			&\hspace{15pt} +x_4^{(2)},c+x_2^{(2)}+x_2^{(1)}-x_3^{(3)}-x_3^{(2)}+x_4^{(3)}+x_4^{(2)}-x_5^{(2)}, c+x_2^{(2)}+x_3^{(1)}\\
			&\hspace{15pt} -x_4^{(4)}, c+x_2^{(2)}-x_3^{(3)}+x_3^{(1)}+x_4^{(3)}-x_5^{(2)}, c+x_3^{(2)}+x_3^{(1)}-x_5^{(2)}\},  x_2^{(2)}\\
			&\hspace{15pt} +x_2^{(1)}-x_6^{(3)} +\text{max} \{c+x_6^{(1)}, x_2^{(2)}+x_2^{(1)}-x_3^{(3)}-x_3^{(2)}+x_5^{(1)}, x_2^{(2)}\\
			&\hspace{15pt} -x_3^{(3)}+x_3^{(1)}-x_4^{(2)}+x_5^{(1)}, x_2^{(2)}-x_3^{(3)}+x_4^{(1)}, x_3^{(3)}+x_3^{(2)}-x_4^{(3)}-x_4^{(2)}\\
			&\hspace{15pt} +x_5^{(1)}, x_3^{(2)}-x_4^{(3)}+x_4^{(1)}, c+x_4^{(2)}+x_4^{(1)}-x_6^{(2)}, x_2^{(2)}+x_2^{(1)}-x_6^{(3)} , \\
			&\hspace{15pt} x_2^{(2)}+x_2^{(1)}+x_6^{(2)}-x_4^{(4)}-x_4^{(3)}, x_2^{(2)}+x_2^{(1)}-x_3^{(2)}-x_4^{(4)}+x_4^{(2)},x_2^{(2)}\\
			&\hspace{15pt} +x_2^{(1)}-x_3^{(3)}-x_3^{(2)}+x_4^{(3)}+x_4^{(2)}-x_5^{(2)}, x_2^{(2)}+x_3^{(1)}-x_4^{(4)}, x_2^{(2)}-x_3^{(3)}\\
			&\hspace{15pt} +x_3^{(1)}+x_4^{(3)}-x_5^{(2)}, x_3^{(2)}+x_3^{(1)}-x_5^{(2)}\}, x_2^{(2)}+x_2^{(1)}+x_6^{(2)} -x_4^{(4)}\\
			&\hspace{15pt} -x_4^{(3)}+\text{max} \{c+x_6^{(1)}, x_2^{(2)}+x_2^{(1)}-x_3^{(3)}-x_3^{(2)}+x_5^{(1)}, x_2^{(2)}-x_3^{(3)}\\
			&\hspace{15pt} +x_3^{(1)}-x_4^{(2)}+x_5^{(1)},x_2^{(2)}-x_3^{(3)}+x_4^{(1)}, x_3^{(3)}+x_3^{(2)}-x_4^{(3)}-x_4^{(2)}+x_5^{(1)}, \\
			&\hspace{15pt} x_3^{(2)}-x_4^{(3)}+x_4^{(1)}, x_4^{(2)}+x_4^{(1)}-x_6^{(2)}, x_2^{(2)}+x_2^{(1)}-x_6^{(3)} , x_2^{(2)}+x_2^{(1)}\\
			&\hspace{15pt} +x_6^{(2)}-x_4^{(4)}-x_4^{(3)},  x_2^{(2)}+x_2^{(1)}-x_3^{(2)}-x_4^{(4)}+x_4^{(2)},x_2^{(2)}+x_2^{(1)}-x_3^{(3)}\\
			&\hspace{15pt} -x_3^{(2)}+x_4^{(3)}+x_4^{(2)}-x_5^{(2)}, x_2^{(2)}+x_3^{(1)}-x_4^{(4)}, x_2^{(2)}-x_3^{(3)}+x_3^{(1)}+x_4^{(3)}\\
			&\hspace{15pt} -x_5^{(2)}, x_3^{(2)}+x_3^{(1)}-x_5^{(2)}\}, c+x_2^{(2)}+x_2^{(1)}-x_3^{(2)}- x_4^{(4)}+x_4^{(2)} + \breve{K}, \\
			&\hspace{15pt} c+x_2^{(2)}+x_2^{(1)}-x_3^{(3)}-x_3^{(2)}+x_4^{(3)}+x_4^{(2)}-x_5^{(2)} + \breve{K}, c+x_2^{(2)}+x_3^{(1)}\\
			&\hspace{15pt} -x_4^{(4)} + \breve{K}, c+x_2^{(2)}-x_3^{(3)}+x_3^{(1)}+x_4^{(3)}-x_5^{(2)} + \breve{K}, c+x_3^{(2)}+x_3^{(1)}\\
			&\hspace{15pt} -x_5^{(2)} + \breve{K}\} ,\\
x_4^{(3)'} &=-c +x_4^{(3)} + \text{max} \{c+ x_6^{(1)} + \text{max} \{ c+x_6^{(1)}, c+x_2^{(2)}+x_2^{(1)}-x_3^{(3)}-x_3^{(2)}\\
			&\hspace{15pt} +x_5^{(1)}, c+x_2^{(2)}-x_3^{(3)}+x_3^{(1)}-x_4^{(2)}+x_5^{(1)},c+x_2^{(2)}-x_3^{(3)}+x_4^{(1)}, c\\
			&\hspace{15pt} +x_3^{(3)}+x_3^{(2)}-x_4^{(3)}-x_4^{(2)}+x_5^{(1)}, c+x_3^{(2)}-x_4^{(3)}+x_4^{(1)}, c+x_4^{(2)}\\
			&\hspace{15pt} +x_4^{(1)}-x_6^{(2)},x_2^{(2)}+x_2^{(1)}-x_6^{(3)} ,x_2^{(2)}+x_2^{(1)}+x_6^{(2)}-x_4^{(4)}-x_4^{(3)}, c\\
			&\hspace{15pt} +x_2^{(2)}+x_2^{(1)}-x_3^{(2)}-x_4^{(4)}+x_4^{(2)},c+x_2^{(2)}+x_2^{(1)}-x_3^{(3)}-x_3^{(2)}+x_4^{(3)}\\
			&\hspace{15pt} +x_4^{(2)}-x_5^{(2)}, c+x_2^{(2)}+x_3^{(1)}-x_4^{(4)}, c+x_2^{(2)}-x_3^{(3)}+x_3^{(1)}+x_4^{(3)}\\
			&\hspace{15pt} -x_5^{(2)}, c+x_3^{(2)}+x_3^{(1)}-x_5^{(2)}\}, c+x_2^{(2)}+x_2^{(1)}-x_3^{(3)}-x_3^{(2)}+x_5^{(1)} \\
			&\hspace{15pt} + \breve{K}, c+x_2^{(2)}-x_3^{(3)}+ x_3^{(1)}-x_4^{(2)}+x_5^{(1)} + \breve{K}, c+x_2^{(2)}-x_3^{(3)}+x_4^{(1)} \\
			&\hspace{15pt} + \breve{K}, c+x_3^{(3)}+x_3^{(2)}-x_4^{(3)}-x_4^{(2)}+x_5^{(1)} + \breve{K}, c+x_3^{(2)}-x_4^{(3)} +x_4^{(1)} \\
			&\hspace{15pt} + \text{max} \{x_6^{(1)}, x_2^{(2)}+x_2^{(1)}-x_3^{(3)}-x_3^{(2)}+x_5^{(1)}, x_2^{(2)}-x_3^{(3)}+x_3^{(1)} -x_4^{(2)}\\
			&\hspace{15pt} +x_5^{(1)}, x_2^{(2)}-x_3^{(3)}+x_4^{(1)}, x_3^{(3)}+x_3^{(2)}-x_4^{(3)}-x_4^{(2)}+x_5^{(1)}, x_3^{(2)}-x_4^{(3)}\\
			&\hspace{15pt} +x_4^{(1)}, x_4^{(2)}+x_4^{(1)}-x_6^{(2)}, x_2^{(2)}+x_2^{(1)}-x_6^{(3)} , x_2^{(2)}+x_2^{(1)}+x_6^{(2)}-x_4^{(4)}\\
			&\hspace{15pt} - x_4^{(3)}, c+x_2^{(2)}+x_2^{(1)}-x_3^{(2)}-x_4^{(4)}+x_4^{(2)},x_2^{(2)}+x_2^{(1)}-x_3^{(3)}-x_3^{(2)}\\
			&\hspace{15pt} +x_4^{(3)}+x_4^{(2)}-x_5^{(2)}, x_2^{(2)}+x_3^{(1)}-x_4^{(4)}, x_2^{(2)}-x_3^{(3)}+x_3^{(1)}+x_4^{(3)}-x_5^{(2)}, \\
			&\hspace{15pt} x_3^{(2)}+x_3^{(1)}-x_5^{(2)}\}, c+x_4^{(2)}+x_4^{(1)}-x_6^{(2)}+ \text{max} \{ c+x_6^{(1)}, c+x_2^{(2)}\\
			&\hspace{15pt} +x_2^{(1)}-x_3^{(3)}-x_3^{(2)}+x_5^{(1)}, c+x_2^{(2)}-x_3^{(3)}+x_3^{(1)}-x_4^{(2)}+x_5^{(1)}, c\\
			&\hspace{15pt} +x_2^{(2)}-x_3^{(3)}+x_4^{(1)}, c+x_3^{(3)}+x_3^{(2)}-x_4^{(3)}-x_4^{(2)}+x_5^{(1)}, c+x_3^{(2)}-x_4^{(3)}\\
			&\hspace{15pt} +x_4^{(1)}, c+x_4^{(2)}+x_4^{(1)}-x_6^{(2)}, x_2^{(2)}+x_2^{(1)}-x_6^{(3)} , x_2^{(2)}+x_2^{(1)}+x_6^{(2)}\\
			&\hspace{15pt} -x_4^{(4)}-x_4^{(3)},  c+x_2^{(2)}+x_2^{(1)}-x_3^{(2)}-x_4^{(4)}+x_4^{(2)},c+x_2^{(2)}+x_2^{(1)}-x_3^{(3)}\\
			&\hspace{15pt} -x_3^{(2)}+x_4^{(3)}+x_4^{(2)}-x_5^{(2)}, c+x_2^{(2)}+x_3^{(1)}-x_4^{(4)}, c+x_2^{(2)}-x_3^{(3)}\\
			&\hspace{15pt} +x_3^{(1)} +x_4^{(3)}-x_5^{(2)}, c +x_3^{(2)}+x_3^{(1)}-x_5^{(2)}\},  x_2^{(2)}+x_2^{(1)}- x_6^{(3)} +\text{max} \{c\\
			&\hspace{15pt} +x_6^{(1)},x_2^{(2)}+x_2^{(1)}-x_3^{(3)}-x_3^{(2)}+x_5^{(1)}, x_2^{(2)}-x_3^{(3)}+x_3^{(1)}-x_4^{(2)} +x_5^{(1)},\\
			&\hspace{15pt} x_2^{(2)}-x_3^{(3)}+x_4^{(1)}, x_3^{(3)}+x_3^{(2)}-x_4^{(3)}-x_4^{(2)}+x_5^{(1)}, x_3^{(2)}-x_4^{(3)}+x_4^{(1)},\\
			&\hspace{15pt} c+x_4^{(2)}+x_4^{(1)}-x_6^{(2)}, x_2^{(2)}+x_2^{(1)}-x_6^{(3)} , x_2^{(2)}+x_2^{(1)}+x_6^{(2)}-x_4^{(4)} \\
			&\hspace{15pt} -x_4^{(3)}, x_2^{(2)}+x_2^{(1)}-x_3^{(2)}-x_4^{(4)}+x_4^{(2)},x_2^{(2)}+x_2^{(1)}-x_3^{(3)}-x_3^{(2)}+x_4^{(3)}\\
			&\hspace{15pt} +x_4^{(2)}-x_5^{(2)}, x_2^{(2)}+x_3^{(1)}-x_4^{(4)}, x_2^{(2)}-x_3^{(3)}+x_3^{(1)}+x_4^{(3)}-x_5^{(2)}, x_3^{(2)}\\
			&\hspace{15pt} +x_3^{(1)}-x_5^{(2)}\}, x_2^{(2)}+x_2^{(1)}+x_6^{(2)} -x_4^{(4)}-x_4^{(3)} +\text{max} \{c+x_6^{(1)}, x_2^{(2)}\\
			&\hspace{15pt} +x_2^{(1)}-x_3^{(3)}-x_3^{(2)}+x_5^{(1)}, x_2^{(2)}-x_3^{(3)}+x_3^{(1)}-x_4^{(2)}+x_5^{(1)}, x_2^{(2)}-x_3^{(3)}\\
			&\hspace{15pt} +x_4^{(1)}, x_3^{(3)}+x_3^{(2)}-x_4^{(3)}-x_4^{(2)}+x_5^{(1)}, x_3^{(2)}-x_4^{(3)}+x_4^{(1)}, x_4^{(2)}+x_4^{(1)}\\
			&\hspace{15pt} -x_6^{(2)}, x_2^{(2)}+x_2^{(1)}-x_6^{(3)} , x_2^{(2)}+x_2^{(1)}+x_6^{(2)}-x_4^{(4)} -x_4^{(3)}, \ x_2^{(2)}+x_2^{(1)}\\
			&\hspace{15pt} -x_3^{(2)}-x_4^{(4)}+x_4^{(2)},x_2^{(2)}+x_2^{(1)}-x_3^{(3)}-x_3^{(2)}+x_4^{(3)} +x_4^{(2)}-x_5^{(2)}, x_2^{(2)}\\
			&\hspace{15pt} +x_3^{(1)}-x_4^{(4)}, x_2^{(2)}-x_3^{(3)}+x_3^{(1)}+x_4^{(3)}-x_5^{(2)}, x_3^{(2)} +x_3^{(1)}-x_5^{(2)}\}, c\\
			&\hspace{15pt} +x_2^{(2)}+x_2^{(1)}-x_3^{(2)}- x_4^{(4)}+x_4^{(2)} + \breve{K}, c+x_2^{(2)}+x_2^{(1)} -x_3^{(3)}-x_3^{(2)}\\
			&\hspace{15pt} +x_4^{(3)}+x_4^{(2)}-x_5^{(2)} + \breve{K}, c+x_2^{(2)}+x_3^{(1)}- x_4^{(4)} + \breve{K}, c +x_2^{(2)}-x_3^{(3)}\\
			&\hspace{15pt} +x_3^{(1)}+x_4^{(3)}-x_5^{(2)} + \breve{K}, c+x_3^{(2)}+x_3^{(1)}-x_5^{(2)} + \breve{K}\}  - \breve{K} -  \text{max} \{c\\
			&\hspace{15pt} +x_6^{(1)}, c+x_2^{(2)}+x_2^{(1)}-x_3^{(3)}-x_3^{(2)}+x_5^{(1)}, c+x_2^{(2)}-x_3^{(3)}+x_3^{(1)}-x_4^{(2)}\\
			&\hspace{15pt} +x_5^{(1)},c+x_2^{(2)}-x_3^{(3)}+x_4^{(1)}, c+x_3^{(3)}+x_3^{(2)}-x_4^{(3)}-x_4^{(2)}+x_5^{(1)}, c\\
			&\hspace{15pt} +x_3^{(2)}-x_4^{(3)}+x_4^{(1)},c+x_4^{(2)}+x_4^{(1)}-x_6^{(2)}, x_2^{(2)}+x_2^{(1)}-x_6^{(3)} , x_2^{(2)}\\
			&\hspace{15pt} +x_2^{(1)}+x_6^{(2)}-x_4^{(4)}-x_4^{(3)}, x_2^{(2)}+x_2^{(1)}-x_3^{(2)}-x_4^{(4)}+x_4^{(2)},c+x_2^{(2)}\\
			&\hspace{15pt} +x_2^{(1)}-x_3^{(3)}-x_3^{(2)} +x_4^{(3)}+x_4^{(2)}-x_5^{(2)}, x_2^{(2)}+x_3^{(1)}-x_4^{(4)}, c+x_2^{(2)}\\
			&\hspace{15pt} -x_3^{(3)}+x_3^{(1)}+x_4^{(3)}-x_5^{(2)}, c+x_3^{(2)}+x_3^{(1)}-x_5^{(2)}\} \\
x_4^{(4)'} &=-c + x_4^{(4)}  + \text{max} \{c+x_6^{(1)}, c+x_2^{(2)}+x_2^{(1)}-x_3^{(3)}-x_3^{(2)}+x_5^{(1)}, c+x_2^{(2)}\\
			&\hspace{15pt}	-x_3^{(3)}+x_3^{(1)}-x_4^{(2)}+x_5^{(1)}, c+x_2^{(2)}-x_3^{(3)}+x_4^{(1)}, c+x_3^{(3)}+x_3^{(2)}\\
			&\hspace{15pt} -x_4^{(3)}-x_4^{(2)}+x_5^{(1)}, c+x_3^{(2)}-x_4^{(3)}+x_4^{(1)}, c+x_4^{(2)}+x_4^{(1)}-x_6^{(2)}, x_2^{(2)}\\
			&\hspace{15pt} +x_2^{(1)}-x_6^{(3)} , x_2^{(2)}+x_2^{(1)}+x_6^{(2)}-x_4^{(4)}-x_4^{(3)}, x_2^{(2)}+x_2^{(1)}-x_3^{(2)}-x_4^{(4)}\\
			&\hspace{15pt}+x_4^{(2)}, c+x_2^{(2)}+x_2^{(1)}-x_3^{(3)}-x_3^{(2)}+x_4^{(3)}+x_4^{(2)}-x_5^{(2)},x_2^{(2)}+x_3^{(1)}\\
			&\hspace{15pt}-x_4^{(4)}, c+x_2^{(2)}-x_3^{(3)}+x_3^{(1)}+x_4^{(3)}-x_5^{(2)}, c+x_3^{(2)}+x_3^{(1)}-x_5^{(2)}\} -  \breve{K},\\ 	 
x_5^{(1)'} &=x_5^{(1)} + \breve{K} -  \text{max} \{c+x_6^{(1)}, c+x_2^{(2)}+x_2^{(1)}-x_3^{(3)}-x_3^{(2)}+x_5^{(1)}, c+x_2^{(2)}\\
			&\hspace{15pt} -x_3^{(3)}+x_3^{(1)}-x_4^{(2)}+x_5^{(1)},c+x_2^{(2)}-x_3^{(3)}+x_4^{(1)}, c+x_3^{(3)}+x_3^{(2)}-x_4^{(3)}\\
			&\hspace{15pt} -x_4^{(2)}+x_5^{(1)}, c+x_3^{(2)}-x_4^{(3)}+x_4^{(1)},c+x_4^{(2)}+x_4^{(1)}-x_6^{(2)}, x_2^{(2)}+x_2^{(1)}\\
			&\hspace{15pt} -x_6^{(3)} , x_2^{(2)}+x_2^{(1)}+x_6^{(2)}-x_4^{(4)}-x_4^{(3)}, x_2^{(2)}+x_2^{(1)}-x_3^{(2)}-x_4^{(4)}+x_4^{(2)},\\
			&\hspace{15pt} x_2^{(2)}+x_2^{(1)}-x_3^{(3)}-x_3^{(2)}+x_4^{(3)}+x_4^{(2)}-x_5^{(2)},x_2^{(2)}+x_3^{(1)}-x_4^{(4)}, x_2^{(2)}\\
			&\hspace{15pt} -x_3^{(3)}+x_3^{(1)}+x_4^{(3)}-x_5^{(2)}, x_3^{(2)}+x_3^{(1)}-x_5^{(2)}\}, \\		 
x_5^{(2)'} &=-c + x_5^{(2)}  + \text{max} \{c+x_6^{(1)}, c+x_2^{(2)}+x_2^{(1)}-x_3^{(3)}-x_3^{(2)}+x_5^{(1)}, c+x_2^{(2)}\\
			&\hspace{15pt} -x_3^{(3)}+x_3^{(1)}-x_4^{(2)}+x_5^{(1)}, c+x_2^{(2)}-x_3^{(3)}+x_4^{(1)}, c+x_3^{(3)}+x_3^{(2)}-x_4^{(3)}\\
			&\hspace{15pt} -x_4^{(2)}+x_5^{(1)}, c+x_3^{(2)}-x_4^{(3)}+x_4^{(1)},c+x_4^{(2)}+x_4^{(1)}-x_6^{(2)}, x_2^{(2)}+x_2^{(1)}\\
			&\hspace{15pt} -x_6^{(3)} , x_2^{(2)}+x_2^{(1)}+x_6^{(2)}-x_4^{(4)}-x_4^{(3)}, x_2^{(2)}+x_2^{(1)}-x_3^{(2)}-x_4^{(4)}+x_4^{(2)},\\
			&\hspace{15pt} x_2^{(2)}+x_2^{(1)}-x_3^{(3)}-x_3^{(2)}+x_4^{(3)}+x_4^{(2)}-x_5^{(2)}, x_2^{(2)}+x_3^{(1)}-x_4^{(4)}, x_2^{(2)}\\
			&\hspace{15pt} -x_3^{(3)}+x_3^{(1)}+x_4^{(3)}-x_5^{(2)}, x_3^{(2)}+x_3^{(1)}-x_5^{(2)}\} -  \breve{K},\\ 
x_6^{(1)'} &=x_6^{(1)} + \breve{K} - \text{max} \{c+ x_6^{(1)}, x_2^{(2)}+x_2^{(1)}-x_3^{(3)}-x_3^{(2)}+x_5^{(1)}, x_2^{(2)}-x_3^{(3)}\\
			&\hspace{15pt} +x_3^{(1)}-x_4^{(2)}+x_5^{(1)}, x_2^{(2)}-x_3^{(3)}+x_4^{(1)}, x_3^{(3)}+x_3^{(2)}-x_4^{(3)}-x_4^{(2)}+x_5^{(1)}, \\
			&\hspace{15pt} x_3^{(2)}-x_4^{(3)}+x_4^{(1)}, x_4^{(2)}+x_4^{(1)}-x_6^{(2)}, x_2^{(2)}+x_2^{(1)}-x_6^{(3)} , x_2^{(2)}+x_2^{(1)}\\
			&\hspace{15pt} +x_6^{(2)}-x_4^{(4)}-x_4^{(3)}, x_2^{(2)}+x_2^{(1)}-x_3^{(2)}-x_4^{(4)}+x_4^{(2)},x_2^{(2)}+x_2^{(1)}-x_3^{(3)}\\
			&\hspace{15pt} -x_3^{(2)}+x_4^{(3)}+x_4^{(2)}-x_5^{(2)},x_2^{(2)}+x_3^{(1)}-x_4^{(4)}, x_2^{(2)}-x_3^{(3)}+x_3^{(1)}+x_4^{(3)}\\
			&\hspace{15pt} -x_5^{(2)}, x_3^{(2)}+x_3^{(1)}-x_5^{(2)}\}, \\
x_6^{(2)'} &=x_6^{(2)} + \text{max} \{c+ x_6^{(1)}, x_2^{(2)}+x_2^{(1)}-x_3^{(3)}-x_3^{(2)}+x_5^{(1)}, x_2^{(2)}-x_3^{(3)}+x_3^{(1)}\\
			&\hspace{15pt} -x_4^{(2)}+x_5^{(1)}, x_2^{(2)}-x_3^{(3)}+x_4^{(1)}, x_3^{(3)}+x_3^{(2)}-x_4^{(3)}-x_4^{(2)}+x_5^{(1)}, x_3^{(2)}\\
			&\hspace{15pt} -x_4^{(3)}+x_4^{(1)}, x_4^{(2)}+x_4^{(1)}-x_6^{(2)}, x_2^{(2)}+x_2^{(1)}-x_6^{(3)} , x_2^{(2)}+x_2^{(1)}+x_6^{(2)}\\
			&\hspace{15pt} -x_4^{(4)}-x_4^{(3)}, x_2^{(2)}+x_2^{(1)}-x_3^{(2)}-x_4^{(4)}+x_4^{(2)},x_2^{(2)}+x_2^{(1)}-x_3^{(3)}-x_3^{(2)}\\
			&\hspace{15pt} +x_4^{(3)}+x_4^{(2)}-x_5^{(2)}, x_2^{(2)}+x_3^{(1)}-x_4^{(4)}, x_2^{(2)}-x_3^{(3)}+x_3^{(1)}+x_4^{(3)}-x_5^{(2)},\\
			&\hspace{15pt} x_3^{(2)}+x_3^{(1)}-x_5^{(2)}\} -  \text{max} \{ c+x_6^{(1)}, c+x_2^{(2)}+x_2^{(1)}-x_3^{(3)}-x_3^{(2)}+x_5^{(1)}, \\
			&\hspace{15pt} c+x_2^{(2)}-x_3^{(3)}+x_3^{(1)}-x_4^{(2)}+x_5^{(1)},c+x_2^{(2)}-x_3^{(3)}+x_4^{(1)}, c+x_3^{(3)}\\
			&\hspace{15pt} +x_3^{(2)} -x_4^{(3)}-x_4^{(2)}+x_5^{(1)}, c+x_3^{(2)}-x_4^{(3)}+x_4^{(1)}, c+x_4^{(2)}+x_4^{(1)}-x_6^{(2)}, \\
			&\hspace{15pt} x_2^{(2)}+x_2^{(1)}-x_6^{(3)} , c+x_2^{(2)}+x_2^{(1)}+x_6^{(2)}-x_4^{(4)}-x_4^{(3)}, c+x_2^{(2)}+x_2^{(1)}\\
			&\hspace{15pt} -x_3^{(2)}-x_4^{(4)}+x_4^{(2)},c+x_2^{(2)}+x_2^{(1)}-x_3^{(3)}-x_3^{(2)}+x_4^{(3)}+x_4^{(2)}-x_5^{(2)},\\
	 		 &\hspace{15pt} c+x_2^{(2)}+x_3^{(1)}-x_4^{(4)}, c+x_2^{(2)}-x_3^{(3)}+x_3^{(1)}+x_4^{(3)}-x_5^{(2)}, c+x_3^{(2)}\\
			 &\hspace{15pt} +x_3^{(1)}-x_5^{(2)}\}, \\
x_6^{(3)'} &=-c + x_6^{(3)}  + \text{max} \{ c+x_6^{(1)}, c+x_2^{(2)}+x_2^{(1)}-x_3^{(3)}-x_3^{(2)}+x_5^{(1)}, c+x_2^{(2)}\\
			&\hspace{15pt} -x_3^{(3)}+x_3^{(1)}-x_4^{(2)}+x_5^{(1)}, c+x_2^{(2)}-x_3^{(3)}+x_4^{(1)}, c+x_3^{(3)}+x_3^{(2)}-x_4^{(3)}\\
			&\hspace{15pt} -x_4^{(2)}+x_5^{(1)}, c+x_3^{(2)}-x_4^{(3)}+x_4^{(1)}, c+x_4^{(2)}+x_4^{(1)}-x_6^{(2)}, x_2^{(2)}+x_2^{(1)}\\
			&\hspace{15pt} -x_6^{(3)} , c+x_2^{(2)}+x_2^{(1)}+x_6^{(2)}-x_4^{(4)}-x_4^{(3)}, c+x_2^{(2)}+x_2^{(1)}-x_3^{(2)}\\
			&\hspace{15pt} -x_4^{(4)} +x_4^{(2)},c+x_2^{(2)}+x_2^{(1)}-x_3^{(3)}-x_3^{(2)}+x_4^{(3)}+x_4^{(2)}-x_5^{(2)},c+x_2^{(2)}\\
			&\hspace{15pt} +x_3^{(1)} -x_4^{(4)}, c+x_2^{(2)}-x_3^{(3)}+x_3^{(1)}+x_4^{(3)}-x_5^{(2)}, c+x_3^{(2)}+x_3^{(1)}\\
			&\hspace{15pt} -x_5^{(2)}\} -  \breve{K}.
\end{align*}

As shown in \cite{BK, N}, $\mathcal{X}$ with maps $\e_k, \f_k :\mathcal{X} \longrightarrow \mathcal{X}\cup \{0\}, \; \veps_k, \vphi_k : \mathcal{X} \longrightarrow \mZ, \; 0\leq k \leq 6$ and $\text{wt}: \mathcal{X} \longrightarrow P_{cl}$ is a Kashiwara crystal where for $x \in \mathcal{X}$
\begin{align*}
\e_k(x) &= \mathcal{UD}(e_k^c)(x)\arrowvert_{c=1}, \; \f_k(x) = \mathcal{UD}(e_k^c)(x)\arrowvert_{c=-1}, \\
\text{wt}(x) &= \sum_{k=0}^6\text{wt}_k(x)\L_k, \text{where} \; \text{wt}_k(x) = \mathcal{UD}(\gamma_k)(x), \\
\veps_k(x) &= \mathcal{UD}(\veps_k)(x), \; \vphi_k(x) = \text{wt}_k(x) + \veps_k(x).
\end{align*}

In particular, the explicit actions of $\f_k, 1\leq k \leq 6$ on $\mathcal{X}$ is given as follows.
\begin{align*}
\tilde{f_1}(x) & =(x_6^{(3)},x_4^{(4)}, x_3^{(3)}, x_2^{(2)}, x_5^{(2)}, x_4^{(3)}, x_3^{(2)}, x_6^{(2)}, x_4^{(2)}, x_5^{(1)},x_1^{(1)}-1, x_2^{(1)}, x_3^{(1)}, x_4^{(1)}, \\
			& \hspace{15pt} x_6^{(1)}),\\
\tilde{f_2}(x)  & =\begin{cases} (x_6^{(3)},x_4^{(4)}, x_3^{(3)}, x_2^{(2)}-1, x_5^{(2)}, x_4^{(3)}, x_3^{(2)}, x_6^{(2)}, x_4^{(2)}, x_5^{(1)},x_1^{(1)}, x_2^{(1)}, x_3^{(1)}, x_4^{(1)},  \\
						\hspace{15pt} x_6^{(1)}) \ \text{if} \  x_2^{(2)}+x_2^{(1)} > x_1^{(1)}+x_3^{(2)},\\ 
						(x_6^{(3)},x_4^{(4)}, x_3^{(3)}, x_2^{(2)}, x_5^{(2)}, x_4^{(3)}, x_3^{(2)}, x_6^{(2)}, x_4^{(2)}, x_5^{(1)},x_1^{(1)}, x_2^{(1)}-1, x_3^{(1)}, x_4^{(1)},  \\
						\hspace{15pt} x_6^{(1)}) \ \text{if} \  x_2^{(2)}+x_2^{(1)} \leq x_1^{(1)}+x_3^{(2)}, \end{cases}\\
\tilde{f_3}(x)  & =\begin{cases} (x_6^{(3)},x_4^{(4)}, x_3^{(3)}-1, x_2^{(2)}, x_5^{(2)}, x_4^{(3)}, x_3^{(2)}, x_6^{(2)}, x_4^{(2)}, x_5^{(1)},x_1^{(1)}, x_2^{(1)}, x_3^{(1)}, x_4^{(1)}, \\ 
						\hspace{15pt} x_6^{(1)}) \ \text{if} \ x_3^{(3)}+x_3^{(2)} > x_2^{(2)}+x_4^{(3)},\\
						\hspace{1.7cm} x_3^{(3)}+2x_3^{(2)}+x_3^{(1)} > x_2^{(2)}+x_2^{(1)}+x_4^{(3)}+x_4^{(2)}, \end{cases}\\
\tilde{f_3}(x)  & =\begin{cases}	(x_6^{(3)},x_4^{(4)}, x_3^{(3)}, x_2^{(2)}, x_5^{(2)}, x_4^{(3)}, x_3^{(2)}-1, x_6^{(2)}, x_4^{(2)}, x_5^{(1)},x_1^{(1)}, x_2^{(1)}, x_3^{(1)}, x_4^{(1)}, \\
						\hspace{15pt} x_6^{(1)})  \ \text{if} \ x_3^{(3)}+x_3^{(2)} \leq x_2^{(2)}+x_4^{(3)}, \ x_3^{(2)}+x_3^{(1)} > x_2^{(1)}+x_4^{(2)},\\ 
						(x_6^{(3)},x_4^{(4)}, x_3^{(3)}, x_2^{(2)}, x_5^{(2)}, x_4^{(3)}, x_3^{(2)}, x_6^{(2)}, x_4^{(2)}, x_5^{(1)},x_1^{(1)}, x_2^{(1)}, x_3^{(1)}-1, x_4^{(1)},  \\
						\hspace{15pt} x_6^{(1)})  \ \text{if} \ x_3^{(2)}+x_3^{(1)} \leq x_2^{(1)}+x_4^{(2)},  \\
						 \hspace{1.7cm} x_3^{(3)}+2x_3^{(2)}+x_3^{(1)} \leq x_2^{(2)}+x_2^{(1)}+x_4^{(3)}+x_4^{(2)}, \end{cases}\\
\tilde{f_4}(x)  & =\begin{cases} (x_6^{(3)},x_4^{(4)}-1, x_3^{(3)}, x_2^{(2)}, x_5^{(2)}, x_4^{(3)}, x_3^{(2)}, x_6^{(2)}, x_4^{(2)}, x_5^{(1)},x_1^{(1)}, x_2^{(1)}, x_3^{(1)}, x_4^{(1)},  \\
						\hspace{15pt} x_6^{(1)}) \ \text{if} \  x_4^{(4)}+x_4^{(3)}> x_3^{(3)}+x_5^{(2)}, \\
							\hspace{1.7cm} x_4^{(4)}+2x_4^{(3)}+x_4^{(2)} > x_3^{(3)}+x_3^{(2)}+x_5^{(2)}+x_6^{(2)}, \\
							 \hspace{25pt} x_4^{(4)}+2x_4^{(3)}+2x_4^{(2)}+x_4^{(1)} > x_3^{(3)}	+x_3^{(2)}	+x_3^{(1)}	+x_5^{(2)}+x_5^{(1)}+x_6^{(2)}, \\ 
						(x_6^{(3)},x_4^{(4)}, x_3^{(3)}, x_2^{(2)}, x_5^{(2)}, x_4^{(3)}-1, x_3^{(2)}, x_6^{(2)}, x_4^{(2)}, x_5^{(1)},x_1^{(1)}, x_2^{(1)}, x_3^{(1)}, x_4^{(1)},  \\
						\hspace{15pt} x_6^{(1)}) \ \text{if} \  x_4^{(4)}+x_4^{(3)} \leq x_3^{(3)}+x_5^{(2)}, \ x_4^{(3)}+x_4^{(2)}> x_3^{(2)}+x_6^{(2)}, \\ 
							\hspace{1.7cm} x_4^{(3)}+2x_4^{(2)}+x_4^{(1)} > x_3^{(2)}+x_3^{(1)}+x_5^{(1)}+x_6^{(2)}, \\
						(x_6^{(3)},x_4^{(4)}, x_3^{(3)}, x_2^{(2)}, x_5^{(2)}, x_4^{(3)}, x_3^{(2)}, x_6^{(2)}, x_4^{(2)}-1, x_5^{(1)},x_1^{(1)}, x_2^{(1)}, x_3^{(1)}, x_4^{(1)},  \\
						\hspace{15pt}x_6^{(1)}) \ \text{if} \  x_4^{(4)}+2x_4^{(3)}+x_4^{(2)} \leq x_3^{(3)}+x_3^{(2)}+x_5^{(2)}+x_6^{(2)},  \\ 
							\hspace{1.7cm}  x_4^{(3)}+x_4^{(2)} \leq x_3^{(2)}+x_6^{(2)},x_4^{(2)}+x_4^{(1)}> x_3^{(1)}+x_5^{(1)},   \\
						(x_6^{(3)},x_4^{(4)}, x_3^{(3)}, x_2^{(2)}, x_5^{(2)}, x_4^{(3)}, x_3^{(2)}, x_6^{(2)}, x_4^{(2)}, x_5^{(1)},x_1^{(1)}, x_2^{(1)}, x_3^{(1)}, x_4^{(1)}-1,  \\
						\hspace{15pt}x_6^{(1)}) \ \text{if} \ x_4^{(2)}+x_4^{(1)} \leq x_3^{(1)}+x_5^{(1)}, \\							
							\hspace{1.7cm} x_4^{(3)}+2x_4^{(2)}+x_4^{(1)} \leq x_3^{(2)}+x_3^{(1)}+x_5^{(1)}+x_6^{(2)}, \\ 
							\hspace{25pt} x_4^{(4)}+2x_4^{(3)}+2x_4^{(2)}+x_4^{(1)} \leq x_3^{(3)}+x_3^{(2)}+x_3^{(1)}	+x_5^{(2)}+x_5^{(1)}+x_6^{(2)},    \end{cases}\\
\tilde{f_5}(x) & =\begin{cases} (x_6^{(3)},x_4^{(4)}, x_3^{(3)}, x_2^{(2)}, x_5^{(2)}-1, x_4^{(3)}, x_3^{(2)}, x_6^{(2)}, x_4^{(2)}, x_5^{(1)},x_1^{(1)}, x_2^{(1)}, x_3^{(1)}, x_4^{(1)},   \\ 
						\hspace{15pt}x_6^{(1)}) \ \text{if} \  x_5^{(2)}+x_5^{(1)} > x_4^{(3)}+x_4^{(2)},\\ 
						(x_6^{(3)},x_4^{(4)}, x_3^{(3)}, x_2^{(2)}, x_5^{(2)}, x_4^{(3)}, x_3^{(2)}, x_6^{(2)}, x_4^{(2)}, x_5^{(1)}-1,x_1^{(1)}, x_2^{(1)}, x_3^{(1)}, x_4^{(1)},  \\
						\hspace{15pt}x_6^{(1)}) \ \text{if} \  x_5^{(2)}+x_5^{(1)} \leq x_4^{(3)}+x_4^{(2)}, \end{cases} \\
\tilde{f_6}(x)  & =\begin{cases} (x_6^{(3)}-1,x_4^{(4)}, x_3^{(3)}, x_2^{(2)}, x_5^{(2)}, x_4^{(3)}, x_3^{(2)}, x_6^{(2)}, x_4^{(2)}, x_5^{(1)},x_1^{(1)}, x_2^{(1)}, x_3^{(1)}, x_4^{(1)},  \\ 
						\hspace{15pt} x_6^{(1)}) \ \text{if} \ x_6^{(3)}+x_6^{(2)} > x_4^{(4)}+x_4^{(3)},\\
							\hspace{1.7cm} x_6^{(3)}+2x_6^{(2)}+x_6^{(1)} > x_4^{(4)}+x_4^{(3)}+x_4^{(2)}+x_4^{(1)}, \\ 
						(x_6^{(3)},x_4^{(4)}, x_3^{(3)}, x_2^{(2)}, x_5^{(2)}, x_4^{(3)}, x_3^{(2)}, x_6^{(2)}-1, x_4^{(2)}, x_5^{(1)},x_1^{(1)}, x_2^{(1)}, x_3^{(1)}, x_4^{(1)},  \\
						\hspace{15pt} x_6^{(1)}) \ \text{if} \ x_6^{(3)}+x_6^{(2)} \leq x_4^{(4)}+x_4^{(3)}, \ x_6^{(2)}+x_6^{(1)} > x_4^{(2)}+x_4^{(1)},\\ 
						(x_6^{(3)},x_4^{(4)}, x_3^{(3)}, x_2^{(2)}, x_5^{(2)}, x_4^{(3)}, x_3^{(2)}, x_6^{(2)}, x_4^{(2)}, x_5^{(1)},x_1^{(1)}, x_2^{(1)}, x_3^{(1)}, x_4^{(1)},  \\
						\hspace{15pt} x_6^{(1)}-1) \  \text{if} \ x_6^{(2)}+x_6^{(1)} \leq x_4^{(2)}+x_4^{(1)},  \\
							\hspace{2.3cm} x_6^{(3)}+2x_6^{(2)}+x_6^{(1)} \leq x_4^{(4)}+x_4^{(3)}+x_4^{(2)}+x_4^{(1)}. \end{cases}
\end{align*}

To determine the explicit action of $\tilde{f_0}(x)$ we define conditions $(\breve{F1})-(\breve{F14})$ as follows.
\begin{align*}
(\breve{F1}) 	& \hspace{5pt}x_2^{(2)}+x_2^{(1)}-x_6^{(3)} \geq x_6^{(1)}, \\
		& \hspace{5pt}x_2^{(2)}+x_2^{(1)}-x_6^{(3)} \geq x_2^{(2)}+x_2^{(1)}-x_3^{(3)}-x_3^{(2)}+x_5^{(1)}, \\
		& \hspace{5pt}x_2^{(2)}+x_2^{(1)}-x_6^{(3)} \geq x_2^{(2)}-x_3^{(3)}+x_3^{(1)}-x_4^{(2)}+x_5^{(1)} , \\
		& \hspace{5pt}x_2^{(2)}+x_2^{(1)}-x_6^{(3)} \geq x_2^{(2)}-x_3^{(3)}+x_4^{(1)}, \\
		& \hspace{5pt}x_2^{(2)}+x_2^{(1)}-x_6^{(3)} \geq x_3^{(2)}+x_3^{(1)}-x_4^{(3)}-x_4^{(2)}+x_5^{(1)}, \\
		& \hspace{5pt}x_2^{(2)}+x_2^{(1)}-x_6^{(3)} \geq x_3^{(2)}-x_4^{(3)}+x_4^{(1)}, \\
		& \hspace{5pt}x_2^{(2)}+x_2^{(1)}-x_6^{(3)} \geq x_4^{(2)}+x_4^{(1)}-x_6^{(2)}, \\
		& \hspace{5pt}x_2^{(2)}+x_2^{(1)}-x_6^{(3)} \geq x_2^{(2)}+x_2^{(1)}-x_4^{(4)}-x_4^{(3)}+x_6^{(2)}, \\
		& \hspace{5pt}x_2^{(2)}+x_2^{(1)}-x_6^{(3)} \geq x_2^{(2)}+x_2^{(1)}-x_3^{(2)}-x_4^{(4)}+x_4^{(2)}, \\
		& \hspace{5pt}x_2^{(2)}+x_2^{(1)}-x_6^{(3)} \geq x_2^{(2)}+x_2^{(1)}-x_3^{(3)}-x_3^{(2)}+x_4^{(3)}+x_4^{(2)}-x_5^{(2)}, \\
		& \hspace{5pt}x_2^{(2)}+x_2^{(1)}-x_6^{(3)} \geq x_2^{(2)}+x_3^{(1)}-x_4^{(4)}, \\
		& \hspace{5pt}x_2^{(2)}+x_2^{(1)}-x_6^{(3)} \geq x_2^{(2)}-x_3^{(3)}+x_3^{(1)}+x_4^{(3)}-x_5^{(2)}, \\
		& \hspace{5pt}x_2^{(2)}+x_2^{(1)}-x_6^{(3)} \geq x_3^{(2)}+x_3^{(3)}-x_5^{(2)}, \\
(\breve{F2})	& \hspace{5pt}x_2^{(2)}+x_2^{(1)}-x_4^{(4)}-x_4^{(3)}+x_6^{(2)} \geq x_6^{(1)}, \\
		& \hspace{5pt}x_2^{(2)}+x_2^{(1)}-x_4^{(4)}-x_4^{(3)}+x_6^{(2)} \geq x_2^{(2)}+x_2^{(1)}-x_3^{(3)}-x_3^{(2)}+x_5^{(1)}, \\
		& \hspace{5pt}x_2^{(2)}+x_2^{(1)}-x_4^{(4)}-x_4^{(3)}+x_6^{(2)} \geq x_2^{(2)}-x_3^{(3)}+x_3^{(1)}-x_4^{(2)}+x_5^{(1)} , \\
		& \hspace{5pt}x_2^{(2)}+x_2^{(1)}-x_4^{(4)}-x_4^{(3)}+x_6^{(2)} \geq x_2^{(2)}-x_3^{(3)}+x_4^{(1)}, \\
		& \hspace{5pt}x_2^{(2)}+x_2^{(1)}-x_4^{(4)}-x_4^{(3)}+x_6^{(2)} \geq x_3^{(2)}+x_3^{(1)}-x_4^{(3)}-x_4^{(2)}+x_5^{(1)}, \\
		& \hspace{5pt}x_2^{(2)}+x_2^{(1)}-x_4^{(4)}-x_4^{(3)}+x_6^{(2)} \geq x_3^{(2)}-x_4^{(3)}+x_4^{(1)}, \\
		& \hspace{5pt}x_2^{(2)}+x_2^{(1)}-x_4^{(4)}-x_4^{(3)}+x_6^{(2)} \geq x_4^{(2)}+x_4^{(1)}-x_6^{(2)}, \\
		& \hspace{5pt}x_2^{(2)}+x_2^{(1)}-x_4^{(4)}-x_4^{(3)}+x_6^{(2)} > x_2^{(2)}+x_2^{(1)}-x_6^{(3)}, \\
		& \hspace{5pt}x_2^{(2)}+x_2^{(1)}-x_4^{(4)}-x_4^{(3)}+x_6^{(2)} \geq x_2^{(2)}+x_2^{(1)}-x_3^{(2)}-x_4^{(4)}+x_4^{(2)}, \\
		& \hspace{5pt}x_2^{(2)}+x_2^{(1)}-x_4^{(4)}-x_4^{(3)}+x_6^{(2)} \geq x_2^{(2)}+x_2^{(1)}-x_3^{(3)}-x_3^{(2)}+x_4^{(3)}+x_4^{(2)}-x_5^{(2)}, \\
		& \hspace{5pt}x_2^{(2)}+x_2^{(1)}-x_4^{(4)}-x_4^{(3)}+x_6^{(2)} \geq x_2^{(2)}+x_3^{(1)}-x_4^{(4)}, \\
		& \hspace{5pt}x_2^{(2)}+x_2^{(1)}-x_4^{(4)}-x_4^{(3)}+x_6^{(2)} \geq x_2^{(2)}-x_3^{(3)}+x_3^{(1)}+x_4^{(3)}-x_5^{(2)}, \\
		& \hspace{5pt}x_2^{(2)}+x_2^{(1)}-x_4^{(4)}-x_4^{(3)}+x_6^{(2)} \geq x_3^{(2)}+x_3^{(3)}-x_5^{(2)}, \\
(\breve{F3}) 	& \hspace{5pt}x_2^{(2)}+x_2^{(1)}-x_3^{(2)}-x_4^{(4)}+x_4^{(2)} \geq x_6^{(1)}, \\
		& \hspace{5pt}x_2^{(2)}+x_2^{(1)}-x_3^{(2)}-x_4^{(4)}+x_4^{(2)} \geq x_2^{(2)}+x_2^{(1)}-x_3^{(3)}-x_3^{(2)}+x_5^{(1)}, \\
		& \hspace{5pt}x_2^{(2)}+x_2^{(1)}-x_3^{(2)}-x_4^{(4)}+x_4^{(2)} \geq x_2^{(2)}-x_3^{(3)}+x_3^{(1)}-x_4^{(2)}+x_5^{(1)} , \\
		& \hspace{5pt}x_2^{(2)}+x_2^{(1)}-x_3^{(2)}-x_4^{(4)}+x_4^{(2)} \geq x_2^{(2)}-x_3^{(3)}+x_4^{(1)}, \\
		& \hspace{5pt}x_2^{(2)}+x_2^{(1)}-x_3^{(2)}-x_4^{(4)}+x_4^{(2)} \geq x_3^{(2)}+x_3^{(1)}-x_4^{(3)}-x_4^{(2)}+x_5^{(1)}, \\
		& \hspace{5pt}x_2^{(2)}+x_2^{(1)}-x_3^{(2)}-x_4^{(4)}+x_4^{(2)} \geq x_3^{(2)}-x_4^{(3)}+x_4^{(1)}, \\
		& \hspace{5pt}x_2^{(2)}+x_2^{(1)}-x_3^{(2)}-x_4^{(4)}+x_4^{(2)} \geq x_4^{(2)}+x_4^{(1)}-x_6^{(2)}, \\
		& \hspace{5pt}x_2^{(2)}+x_2^{(1)}-x_3^{(2)}-x_4^{(4)}+x_4^{(2)} > x_2^{(2)}+x_2^{(1)}-x_6^{(3)}, \\
		& \hspace{5pt}x_2^{(2)}+x_2^{(1)}-x_3^{(2)}-x_4^{(4)}+x_4^{(2)} > x_2^{(2)}+x_2^{(1)}-x_4^{(4)}-x_4^{(3)}+x_6^{(2)}, \\
		& \hspace{5pt}x_2^{(2)}+x_2^{(1)}-x_3^{(2)}-x_4^{(4)}+x_4^{(2)} \geq x_2^{(2)}+x_2^{(1)}-x_3^{(3)}-x_3^{(2)}+x_4^{(3)}+x_4^{(2)}-x_5^{(2)}, \\
		& \hspace{5pt}x_2^{(2)}+x_2^{(1)}-x_3^{(2)}-x_4^{(4)}+x_4^{(2)} \geq x_2^{(2)}+x_3^{(1)}-x_4^{(4)}, \\
		& \hspace{5pt}x_2^{(2)}+x_2^{(1)}-x_3^{(2)}-x_4^{(4)}+x_4^{(2)} \geq x_2^{(2)}-x_3^{(3)}+x_3^{(1)}+x_4^{(3)}-x_5^{(2)}, \\
		& \hspace{5pt}x_2^{(2)}+x_2^{(1)}-x_3^{(2)}-x_4^{(4)}+x_4^{(2)} \geq x_3^{(2)}+x_3^{(3)}-x_5^{(2)}, \\
(\breve{F4}) 	& \hspace{5pt}x_2^{(2)}+x_2^{(1)}-x_3^{(3)}-x_3^{(2)}+x_4^{(3)}+x_4^{(2)}-x_5^{(2)} \geq x_6^{(1)}, \\
		& \hspace{5pt}x_2^{(2)}+x_2^{(1)}-x_3^{(3)}-x_3^{(2)}+x_4^{(3)}+x_4^{(2)}-x_5^{(2)} \geq x_2^{(2)}+x_2^{(1)}-x_3^{(3)}-x_3^{(2)}+x_5^{(1)}, \\
		& \hspace{5pt}x_2^{(2)}+x_2^{(1)}-x_3^{(3)}-x_3^{(2)}+x_4^{(3)}+x_4^{(2)}-x_5^{(2)} \geq x_2^{(2)}-x_3^{(3)}+x_3^{(1)}-x_4^{(2)}+x_5^{(1)} , \\
		& \hspace{5pt}x_2^{(2)}+x_2^{(1)}-x_3^{(3)}-x_3^{(2)}+x_4^{(3)}+x_4^{(2)}-x_5^{(2)} \geq x_2^{(2)}-x_3^{(3)}+x_4^{(1)}, \\
		& \hspace{5pt}x_2^{(2)}+x_2^{(1)}-x_3^{(3)}-x_3^{(2)}+x_4^{(3)}+x_4^{(2)}-x_5^{(2)} \geq x_3^{(2)}+x_3^{(1)}-x_4^{(3)}-x_4^{(2)}+x_5^{(1)}, \\
		& \hspace{5pt}x_2^{(2)}+x_2^{(1)}-x_3^{(3)}-x_3^{(2)}+x_4^{(3)}+x_4^{(2)}-x_5^{(2)} \geq x_3^{(2)}-x_4^{(3)}+x_4^{(1)}, \\
		& \hspace{5pt}x_2^{(2)}+x_2^{(1)}-x_3^{(3)}-x_3^{(2)}+x_4^{(3)}+x_4^{(2)}-x_5^{(2)} \geq x_4^{(2)}+x_4^{(1)}-x_6^{(2)}, \\
		& \hspace{5pt}x_2^{(2)}+x_2^{(1)}-x_3^{(3)}-x_3^{(2)}+x_4^{(3)}+x_4^{(2)}-x_5^{(2)} > x_2^{(2)}+x_2^{(1)}-x_6^{(3)}, \\
		& \hspace{5pt}x_2^{(2)}+x_2^{(1)}-x_3^{(3)}-x_3^{(2)}+x_4^{(3)}+x_4^{(2)}-x_5^{(2)} > x_2^{(2)}+x_2^{(1)}-x_4^{(4)}-x_4^{(3)}+x_6^{(2)}, \\
		& \hspace{5pt}x_2^{(2)}+x_2^{(1)}-x_3^{(3)}-x_3^{(2)}+x_4^{(3)}+x_4^{(2)}-x_5^{(2)} > x_2^{(2)}+x_2^{(1)}-x_3^{(2)}-x_4^{(4)}+x_4^{(2)}, \\
		& \hspace{5pt}x_2^{(2)}+x_2^{(1)}-x_3^{(3)}-x_3^{(2)}+x_4^{(3)}+x_4^{(2)}-x_5^{(2)} \geq x_2^{(2)}+x_3^{(1)}-x_4^{(4)}, \\
		& \hspace{5pt}x_2^{(2)}+x_2^{(1)}-x_3^{(3)}-x_3^{(2)}+x_4^{(3)}+x_4^{(2)}-x_5^{(2)} \geq x_2^{(2)}-x_3^{(3)}+x_3^{(1)}+x_4^{(3)}-x_5^{(2)}, \\
		& \hspace{5pt}x_2^{(2)}+x_2^{(1)}-x_3^{(3)}-x_3^{(2)}+x_4^{(3)}+x_4^{(2)}-x_5^{(2)} \geq x_3^{(2)}+x_3^{(3)}-x_5^{(2)}, \\
(\breve{F5}) 	& \hspace{5pt}x_2^{(2)}+x_3^{(1)}-x_4^{(4)} \geq x_6^{(1)}, \\
		& \hspace{5pt}x_2^{(2)}+x_3^{(1)}-x_4^{(4)} \geq x_2^{(2)}+x_2^{(1)}-x_3^{(3)}-x_3^{(2)}+x_5^{(1)}, \\
		& \hspace{5pt}x_2^{(2)}+x_3^{(1)}-x_4^{(4)} \geq x_2^{(2)}-x_3^{(3)}+x_3^{(1)}-x_4^{(2)}+x_5^{(1)} , \\
		& \hspace{5pt}x_2^{(2)}+x_3^{(1)}-x_4^{(4)} \geq x_2^{(2)}-x_3^{(3)}+x_4^{(1)}, \\
		& \hspace{5pt}x_2^{(2)}+x_3^{(1)}-x_4^{(4)} \geq x_3^{(2)}+x_3^{(1)}-x_4^{(3)}-x_4^{(2)}+x_5^{(1)}, \\
		& \hspace{5pt}x_2^{(2)}+x_3^{(1)}-x_4^{(4)} \geq x_3^{(2)}-x_4^{(3)}+x_4^{(1)}, \\
		& \hspace{5pt}x_2^{(2)}+x_3^{(1)}-x_4^{(4)} \geq x_4^{(2)}+x_4^{(1)}-x_6^{(2)}, \\
		& \hspace{5pt}x_2^{(2)}+x_3^{(1)}-x_4^{(4)} > x_2^{(2)}+x_2^{(1)}-x_6^{(3)}, \\
		& \hspace{5pt}x_2^{(2)}+x_3^{(1)}-x_4^{(4)} > x_2^{(2)}+x_2^{(1)}-x_4^{(4)}-x_4^{(3)}+x_6^{(2)}, \\
		& \hspace{5pt}x_2^{(2)}+x_3^{(1)}-x_4^{(4)} > x_2^{(2)}+x_2^{(1)}-x_3^{(2)}-x_4^{(4)}+x_4^{(2)}, \\
		& \hspace{5pt}x_2^{(2)}+x_3^{(1)}-x_4^{(4)} \geq x_2^{(2)}+x_2^{(1)}-x_3^{(3)}-x_3^{(2)}+x_4^{(3)}+x_4^{(2)}-x_5^{(2)}, \\
		& \hspace{5pt}x_2^{(2)}+x_3^{(1)}-x_4^{(4)} \geq x_2^{(2)}-x_3^{(3)}+x_3^{(1)}+x_4^{(3)}-x_5^{(2)}, \\
		& \hspace{5pt}x_2^{(2)}+x_3^{(1)}-x_4^{(4)} \geq x_3^{(2)}+x_3^{(3)}-x_5^{(2)}, \\
(\breve{F6}) 	& \hspace{5pt}x_2^{(2)}+x_2^{(1)}-x_3^{(3)}-x_3^{(2)}+x_5^{(1)} \geq x_6^{(1)}, \\
		& \hspace{5pt}x_2^{(2)}+x_2^{(1)}-x_3^{(3)}-x_3^{(2)}+x_5^{(1)} \geq x_2^{(2)}-x_3^{(3)}+x_3^{(1)}-x_4^{(2)}+x_5^{(1)} , \\
		& \hspace{5pt}x_2^{(2)}+x_2^{(1)}-x_3^{(3)}-x_3^{(2)}+x_5^{(1)} \geq x_2^{(2)}-x_3^{(3)}+x_4^{(1)}, \\
		& \hspace{5pt}x_2^{(2)}+x_2^{(1)}-x_3^{(3)}-x_3^{(2)}+x_5^{(1)} \geq x_3^{(2)}+x_3^{(1)}-x_4^{(3)}-x_4^{(2)}+x_5^{(1)}, \\
		& \hspace{5pt}x_2^{(2)}+x_2^{(1)}-x_3^{(3)}-x_3^{(2)}+x_5^{(1)} \geq x_3^{(2)}-x_4^{(3)}+x_4^{(1)}, \\
		& \hspace{5pt}x_2^{(2)}+x_2^{(1)}-x_3^{(3)}-x_3^{(2)}+x_5^{(1)} \geq x_4^{(2)}+x_4^{(1)}-x_6^{(2)}, \\
		& \hspace{5pt}x_2^{(2)}+x_2^{(1)}-x_3^{(3)}-x_3^{(2)}+x_5^{(1)} > x_2^{(2)}+x_2^{(1)}-x_6^{(3)}, \\
		& \hspace{5pt}x_2^{(2)}+x_2^{(1)}-x_3^{(3)}-x_3^{(2)}+x_5^{(1)} > x_2^{(2)}+x_2^{(1)}-x_4^{(4)}-x_4^{(3)}+x_6^{(2)}, \\
		& \hspace{5pt}x_2^{(2)}+x_2^{(1)}-x_3^{(3)}-x_3^{(2)}+x_5^{(1)} > x_2^{(2)}+x_2^{(1)}-x_3^{(2)}-x_4^{(4)}+x_4^{(2)}, \\
		& \hspace{5pt}x_2^{(2)}+x_2^{(1)}-x_3^{(3)}-x_3^{(2)}+x_5^{(1)} > x_2^{(2)}+x_2^{(1)}-x_3^{(3)}-x_3^{(2)}+x_4^{(3)}+x_4^{(2)}-x_5^{(2)}, \\
		& \hspace{5pt}x_2^{(2)}+x_2^{(1)}-x_3^{(3)}-x_3^{(2)}+x_5^{(1)} \geq x_2^{(2)}+x_3^{(1)}-x_4^{(4)}, \\
		& \hspace{5pt}x_2^{(2)}+x_2^{(1)}-x_3^{(3)}-x_3^{(2)}+x_5^{(1)} \geq x_2^{(2)}-x_3^{(3)}+x_3^{(1)}+x_4^{(3)}-x_5^{(2)}, \\
		& \hspace{5pt}x_2^{(2)}+x_2^{(1)}-x_3^{(3)}-x_3^{(2)}+x_5^{(1)} \geq x_3^{(2)}+x_3^{(3)}-x_5^{(2)}, \\	
(\breve{F7})	& \hspace{5pt}x_2^{(2)}-x_3^{(3)}+x_3^{(1)}+x_4^{(3)}-x_5^{(2)}  \geq x_6^{(1)}, \\
		& \hspace{5pt}x_2^{(2)}-x_3^{(3)}+x_3^{(1)}+x_4^{(3)}-x_5^{(2)} \geq x_2^{(2)}+x_2^{(1)}-x_3^{(3)}-x_3^{(2)}+x_5^{(1)}, \\
		& \hspace{5pt}x_2^{(2)}-x_3^{(3)}+x_3^{(1)}+x_4^{(3)}-x_5^{(2)} \geq x_2^{(2)}-x_3^{(3)}+x_3^{(1)}-x_4^{(2)}+x_5^{(1)} , \\
		& \hspace{5pt}x_2^{(2)}-x_3^{(3)}+x_3^{(1)}+x_4^{(3)}-x_5^{(2)} \geq x_2^{(2)}-x_3^{(3)}+x_4^{(1)}, \\
		& \hspace{5pt}x_2^{(2)}-x_3^{(3)}+x_3^{(1)}+x_4^{(3)}-x_5^{(2)} \geq x_3^{(2)}+x_3^{(1)}-x_4^{(3)}-x_4^{(2)}+x_5^{(1)}, \\
		& \hspace{5pt}x_2^{(2)}-x_3^{(3)}+x_3^{(1)}+x_4^{(3)}-x_5^{(2)} \geq x_3^{(2)}-x_4^{(3)}+x_4^{(1)}, \\
		& \hspace{5pt}x_2^{(2)}-x_3^{(3)}+x_3^{(1)}+x_4^{(3)}-x_5^{(2)} \geq x_4^{(2)}+x_4^{(1)}-x_6^{(2)}, \\
		& \hspace{5pt}x_2^{(2)}-x_3^{(3)}+x_3^{(1)}+x_4^{(3)}-x_5^{(2)} > x_2^{(2)}+x_2^{(1)}-x_6^{(3)}, \\
		& \hspace{5pt}x_2^{(2)}-x_3^{(3)}+x_3^{(1)}+x_4^{(3)}-x_5^{(2)} > x_2^{(2)}+x_2^{(1)}-x_4^{(4)}-x_4^{(3)}+x_6^{(2)}, \\
		& \hspace{5pt}x_2^{(2)}-x_3^{(3)}+x_3^{(1)}+x_4^{(3)}-x_5^{(2)} > x_2^{(2)}+x_2^{(1)}-x_3^{(2)}-x_4^{(4)}+x_4^{(2)}, \\
		& \hspace{5pt}x_2^{(2)}-x_3^{(3)}+x_3^{(1)}+x_4^{(3)}-x_5^{(2)} > x_2^{(2)}+x_2^{(1)}-x_3^{(3)}-x_3^{(2)}+x_4^{(3)}+x_4^{(2)}-x_5^{(2)}, \\
		& \hspace{5pt}x_2^{(2)}-x_3^{(3)}+x_3^{(1)}+x_4^{(3)}-x_5^{(2)} > x_2^{(2)}+x_3^{(1)}-x_4^{(4)}, \\
		& \hspace{5pt}x_2^{(2)}-x_3^{(3)}+x_3^{(1)}+x_4^{(3)}-x_5^{(2)} \geq x_3^{(2)}+x_3^{(3)}-x_5^{(2)}, \\		
(\breve{F8})	& \hspace{5pt}x_3^{(2)}+x_3^{(3)}-x_5^{(2)} \geq x_6^{(1)}, \\
		& \hspace{5pt}x_3^{(2)}+x_3^{(3)}-x_5^{(2)} \geq x_2^{(2)}+x_2^{(1)}-x_3^{(3)}-x_3^{(2)}+x_5^{(1)}, \\
		& \hspace{5pt}x_3^{(2)}+x_3^{(3)}-x_5^{(2)} \geq x_2^{(2)}-x_3^{(3)}+x_3^{(1)}-x_4^{(2)}+x_5^{(1)} , \\
		& \hspace{5pt}x_3^{(2)}+x_3^{(3)}-x_5^{(2)} \geq x_2^{(2)}-x_3^{(3)}+x_4^{(1)}, \\
		& \hspace{5pt}x_3^{(2)}+x_3^{(3)}-x_5^{(2)} \geq x_3^{(2)}+x_3^{(1)}-x_4^{(3)}-x_4^{(2)}+x_5^{(1)}, \\
		& \hspace{5pt}x_3^{(2)}+x_3^{(3)}-x_5^{(2)} \geq x_3^{(2)}-x_4^{(3)}+x_4^{(1)}, \\
		& \hspace{5pt}x_3^{(2)}+x_3^{(3)}-x_5^{(2)} \geq x_4^{(2)}+x_4^{(1)}-x_6^{(2)}, \\
		& \hspace{5pt}x_3^{(2)}+x_3^{(3)}-x_5^{(2)} > x_2^{(2)}+x_2^{(1)}-x_6^{(3)}, \\
		& \hspace{5pt}x_3^{(2)}+x_3^{(3)}-x_5^{(2)} > x_2^{(2)}+x_2^{(1)}-x_4^{(4)}-x_4^{(3)}+x_6^{(2)}, \\
		& \hspace{5pt}x_3^{(2)}+x_3^{(3)}-x_5^{(2)} > x_2^{(2)}+x_2^{(1)}-x_3^{(2)}-x_4^{(4)}+x_4^{(2)}, \\
		& \hspace{5pt}x_3^{(2)}+x_3^{(3)}-x_5^{(2)} > x_2^{(2)}+x_2^{(1)}-x_3^{(3)}-x_3^{(2)}+x_4^{(3)}+x_4^{(2)}-x_5^{(2)}, \\
		& \hspace{5pt}x_3^{(2)}+x_3^{(3)}-x_5^{(2)} > x_2^{(2)}+x_3^{(1)}-x_4^{(4)}, \\
		& \hspace{5pt}x_3^{(2)}+x_3^{(3)}-x_5^{(2)} > x_2^{(2)}-x_3^{(3)}+x_3^{(1)}+x_4^{(3)}-x_5^{(2)}, \\
(\breve{F9}) 	& \hspace{5pt}x_2^{(2)}-x_3^{(3)}+x_3^{(1)}-x_4^{(2)}+x_5^{(1)}  \geq x_6^{(1)}, \\
		& \hspace{5pt}x_2^{(2)}-x_3^{(3)}+x_3^{(1)}-x_4^{(2)}+x_5^{(1)} > x_2^{(2)}+x_2^{(1)}-x_3^{(3)}-x_3^{(2)}+x_5^{(1)}, \\
		& \hspace{5pt}x_2^{(2)}-x_3^{(3)}+x_3^{(1)}-x_4^{(2)}+x_5^{(1)} \geq x_2^{(2)}-x_3^{(3)}+x_4^{(1)}, \\
		& \hspace{5pt}x_2^{(2)}-x_3^{(3)}+x_3^{(1)}-x_4^{(2)}+x_5^{(1)} \geq x_3^{(2)}+x_3^{(1)}-x_4^{(3)}-x_4^{(2)}+x_5^{(1)}, \\
		& \hspace{5pt}x_2^{(2)}-x_3^{(3)}+x_3^{(1)}-x_4^{(2)}+x_5^{(1)} \geq x_3^{(2)}-x_4^{(3)}+x_4^{(1)}, \\
		& \hspace{5pt}x_2^{(2)}-x_3^{(3)}+x_3^{(1)}-x_4^{(2)}+x_5^{(1)} \geq x_4^{(2)}+x_4^{(1)}-x_6^{(2)}, \\
		& \hspace{5pt}x_2^{(2)}-x_3^{(3)}+x_3^{(1)}-x_4^{(2)}+x_5^{(1)} > x_2^{(2)}+x_2^{(1)}-x_6^{(3)}, \\
		& \hspace{5pt}x_2^{(2)}-x_3^{(3)}+x_3^{(1)}-x_4^{(2)}+x_5^{(1)} > x_2^{(2)}+x_2^{(1)}-x_4^{(4)}-x_4^{(3)}+x_6^{(2)}, \\
		& \hspace{5pt}x_2^{(2)}-x_3^{(3)}+x_3^{(1)}-x_4^{(2)}+x_5^{(1)} > x_2^{(2)}+x_2^{(1)}-x_3^{(2)}-x_4^{(4)}+x_4^{(2)}, \\
		& \hspace{5pt}x_2^{(2)}-x_3^{(3)}+x_3^{(1)}-x_4^{(2)}+x_5^{(1)} > x_2^{(2)}+x_2^{(1)}-x_3^{(3)}-x_3^{(2)}+x_4^{(3)}+x_4^{(2)}-x_5^{(2)}, \\
		& \hspace{5pt}x_2^{(2)}-x_3^{(3)}+x_3^{(1)}-x_4^{(2)}+x_5^{(1)} > x_2^{(2)}+x_3^{(1)}-x_4^{(4)}, \\
		& \hspace{5pt}x_2^{(2)}-x_3^{(3)}+x_3^{(1)}-x_4^{(2)}+x_5^{(1)} > x_2^{(2)}-x_3^{(3)}+x_3^{(1)}+x_4^{(3)}-x_5^{(2)}, \\
		& \hspace{5pt}x_2^{(2)}-x_3^{(3)}+x_3^{(1)}-x_4^{(2)}+x_5^{(1)} \geq x_3^{(2)}+x_3^{(3)}-x_5^{(2)}, \\	
(\breve{F10}) 	& \hspace{5pt}x_2^{(2)}-x_3^{(3)}+x_4^{(1)}  \geq x_6^{(1)}, \\
		& \hspace{5pt}x_2^{(2)}-x_3^{(3)}+x_4^{(1)} > x_2^{(2)}+x_2^{(1)}-x_3^{(3)}-x_3^{(2)}+x_5^{(1)}, \\
		& \hspace{5pt}x_2^{(2)}-x_3^{(3)}+x_4^{(1)} > x_2^{(2)}-x_3^{(3)}+x_3^{(1)}-x_4^{(2)}+x_5^{(1)}, \\
		& \hspace{5pt}x_2^{(2)}-x_3^{(3)}+x_4^{(1)} \geq x_3^{(2)}+x_3^{(1)}-x_4^{(3)}-x_4^{(2)}+x_5^{(1)}, \\
		& \hspace{5pt}x_2^{(2)}-x_3^{(3)}+x_4^{(1)} \geq x_3^{(2)}-x_4^{(3)}+x_4^{(1)}, \\
		& \hspace{5pt}x_2^{(2)}-x_3^{(3)}+x_4^{(1)} \geq x_4^{(2)}+x_4^{(1)}-x_6^{(2)}, \\
		& \hspace{5pt}x_2^{(2)}-x_3^{(3)}+x_4^{(1)} > x_2^{(2)}+x_2^{(1)}-x_6^{(3)}, \\
		& \hspace{5pt}x_2^{(2)}-x_3^{(3)}+x_4^{(1)} > x_2^{(2)}+x_2^{(1)}-x_4^{(4)}-x_4^{(3)}+x_6^{(2)}, \\
		& \hspace{5pt}x_2^{(2)}-x_3^{(3)}+x_4^{(1)} > x_2^{(2)}+x_2^{(1)}-x_3^{(2)}-x_4^{(4)}+x_4^{(2)}, \\
		& \hspace{5pt}x_2^{(2)}-x_3^{(3)}+x_4^{(1)} > x_2^{(2)}+x_2^{(1)}-x_3^{(3)}-x_3^{(2)}+x_4^{(3)}+x_4^{(2)}-x_5^{(2)}, \\
		& \hspace{5pt}x_2^{(2)}-x_3^{(3)}+x_4^{(1)} > x_2^{(2)}+x_3^{(1)}-x_4^{(4)}, \\
		& \hspace{5pt}x_2^{(2)}-x_3^{(3)}+x_4^{(1)} > x_2^{(2)}-x_3^{(3)}+x_3^{(1)}+x_4^{(3)}-x_5^{(2)}, \\
		& \hspace{5pt}x_2^{(2)}-x_3^{(3)}+x_4^{(1)} \geq x_3^{(2)}+x_3^{(3)}-x_5^{(2)}, \\
(\breve{F11}) 	& \hspace{5pt}x_3^{(2)}+x_3^{(1)}-x_4^{(3)}-x_4^{(2)}+x_5^{(1)} \geq x_6^{(1)}, \\
		& \hspace{5pt}x_3^{(2)}+x_3^{(1)}-x_4^{(3)}-x_4^{(2)}+x_5^{(1)} > x_2^{(2)}+x_2^{(1)}-x_3^{(3)}-x_3^{(2)}+x_5^{(1)}, \\
		& \hspace{5pt}x_3^{(2)}+x_3^{(1)}-x_4^{(3)}-x_4^{(2)}+x_5^{(1)} > x_2^{(2)}-x_3^{(3)}+x_3^{(1)}-x_4^{(2)}+x_5^{(1)}, \\
		& \hspace{5pt}x_3^{(2)}+x_3^{(1)}-x_4^{(3)}-x_4^{(2)}+x_5^{(1)} \geq x_2^{(2)}-x_3^{(3)}+x_4^{(1)}, \\
		& \hspace{5pt}x_3^{(2)}+x_3^{(1)}-x_4^{(3)}-x_4^{(2)}+x_5^{(1)} \geq x_3^{(2)}-x_4^{(3)}+x_4^{(1)}, \\
		& \hspace{5pt}x_3^{(2)}+x_3^{(1)}-x_4^{(3)}-x_4^{(2)}+x_5^{(1)} \geq x_4^{(2)}+x_4^{(1)}-x_6^{(2)}, \\
		& \hspace{5pt}x_3^{(2)}+x_3^{(1)}-x_4^{(3)}-x_4^{(2)}+x_5^{(1)} > x_2^{(2)}+x_2^{(1)}-x_6^{(3)}, \\
		& \hspace{5pt}x_3^{(2)}+x_3^{(1)}-x_4^{(3)}-x_4^{(2)}+x_5^{(1)} > x_2^{(2)}+x_2^{(1)}-x_4^{(4)}-x_4^{(3)}+x_6^{(2)}, \\
		& \hspace{5pt}x_3^{(2)}+x_3^{(1)}-x_4^{(3)}-x_4^{(2)}+x_5^{(1)} > x_2^{(2)}+x_2^{(1)}-x_3^{(2)}-x_4^{(4)}+x_4^{(2)}, \\
		& \hspace{5pt}x_3^{(2)}+x_3^{(1)}-x_4^{(3)}-x_4^{(2)}+x_5^{(1)} > x_2^{(2)}+x_2^{(1)}-x_3^{(3)}-x_3^{(2)}+x_4^{(3)}+x_4^{(2)}-x_5^{(2)}, \\
		& \hspace{5pt}x_3^{(2)}+x_3^{(1)}-x_4^{(3)}-x_4^{(2)}+x_5^{(1)} > x_2^{(2)}+x_3^{(1)}-x_4^{(4)}, \\
		& \hspace{5pt}x_3^{(2)}+x_3^{(1)}-x_4^{(3)}-x_4^{(2)}+x_5^{(1)} > x_2^{(2)}-x_3^{(3)}+x_3^{(1)}+x_4^{(3)}-x_5^{(2)}, \\
		& \hspace{5pt}x_3^{(2)}+x_3^{(1)}-x_4^{(3)}-x_4^{(2)}+x_5^{(1)} > x_3^{(2)}+x_3^{(3)}-x_5^{(2)}, \\		
(\breve{F12}) 	& \hspace{5pt}x_3^{(2)}-x_4^{(3)}+x_4^{(1)} \geq x_6^{(1)}, \\
		& \hspace{5pt}x_3^{(2)}-x_4^{(3)}+x_4^{(1)} > x_2^{(2)}+x_2^{(1)}-x_3^{(3)}-x_3^{(2)}+x_5^{(1)}, \\
		& \hspace{5pt}x_3^{(2)}-x_4^{(3)}+x_4^{(1)} > x_2^{(2)}-x_3^{(3)}+x_3^{(1)}-x_4^{(2)}+x_5^{(1)}, \\
		& \hspace{5pt}x_3^{(2)}-x_4^{(3)}+x_4^{(1)} > x_2^{(2)}-x_3^{(3)}+x_4^{(1)}, \\
		& \hspace{5pt}x_3^{(2)}-x_4^{(3)}+x_4^{(1)} > x_3^{(2)}+x_3^{(1)}-x_4^{(3)}-x_4^{(2)}+x_5^{(1)}, \\
		& \hspace{5pt}x_3^{(2)}-x_4^{(3)}+x_4^{(1)} \geq x_4^{(2)}+x_4^{(1)}-x_6^{(2)}, \\
		& \hspace{5pt}x_3^{(2)}-x_4^{(3)}+x_4^{(1)} > x_2^{(2)}+x_2^{(1)}-x_6^{(3)}, \\
		& \hspace{5pt}x_3^{(2)}-x_4^{(3)}+x_4^{(1)} > x_2^{(2)}+x_2^{(1)}-x_4^{(4)}-x_4^{(3)}+x_6^{(2)}, \\
		& \hspace{5pt}x_3^{(2)}-x_4^{(3)}+x_4^{(1)} > x_2^{(2)}+x_2^{(1)}-x_3^{(2)}-x_4^{(4)}+x_4^{(2)}, \\
		& \hspace{5pt}x_3^{(2)}-x_4^{(3)}+x_4^{(1)} > x_2^{(2)}+x_2^{(1)}-x_3^{(3)}-x_3^{(2)}+x_4^{(3)}+x_4^{(2)}-x_5^{(2)}, \\
		& \hspace{5pt}x_3^{(2)}-x_4^{(3)}+x_4^{(1)} > x_2^{(2)}+x_3^{(1)}-x_4^{(4)}, \\
		& \hspace{5pt}x_3^{(2)}-x_4^{(3)}+x_4^{(1)} > x_2^{(2)}-x_3^{(3)}+x_3^{(1)}+x_4^{(3)}-x_5^{(2)}, \\
		& \hspace{5pt}x_3^{(2)}-x_4^{(3)}+x_4^{(1)} > x_3^{(2)}+x_3^{(3)}-x_5^{(2)}, \\		
(\breve{F13}) 	& \hspace{5pt}x_4^{(2)}+x_4^{(1)}-x_6^{(2)} \geq x_6^{(1)}, \\
		& \hspace{5pt}x_4^{(2)}+x_4^{(1)}-x_6^{(2)} > x_2^{(2)}+x_2^{(1)}-x_3^{(3)}-x_3^{(2)}+x_5^{(1)}, \\
		& \hspace{5pt}x_4^{(2)}+x_4^{(1)}-x_6^{(2)} > x_2^{(2)}-x_3^{(3)}+x_3^{(1)}-x_4^{(2)}+x_5^{(1)}, \\
		& \hspace{5pt}x_4^{(2)}+x_4^{(1)}-x_6^{(2)} > x_2^{(2)}-x_3^{(3)}+x_4^{(1)}, \\
		& \hspace{5pt}x_4^{(2)}+x_4^{(1)}-x_6^{(2)} > x_3^{(2)}+x_3^{(1)}-x_4^{(3)}-x_4^{(2)}+x_5^{(1)}, \\
		& \hspace{5pt}x_4^{(2)}+x_4^{(1)}-x_6^{(2)} > x_3^{(2)}-x_4^{(3)}+x_4^{(1)}, \\
		& \hspace{5pt}x_4^{(2)}+x_4^{(1)}-x_6^{(2)} > x_2^{(2)}+x_2^{(1)}-x_6^{(3)}, \\
		& \hspace{5pt}x_4^{(2)}+x_4^{(1)}-x_6^{(2)} > x_2^{(2)}+x_2^{(1)}-x_4^{(4)}-x_4^{(3)}+x_6^{(2)}, \\
		& \hspace{5pt}x_4^{(2)}+x_4^{(1)}-x_6^{(2)} > x_2^{(2)}+x_2^{(1)}-x_3^{(2)}-x_4^{(4)}+x_4^{(2)}, \\
		& \hspace{5pt}x_4^{(2)}+x_4^{(1)}-x_6^{(2)} > x_2^{(2)}+x_2^{(1)}-x_3^{(3)}-x_3^{(2)}+x_4^{(3)}+x_4^{(2)}-x_5^{(2)}, \\
		& \hspace{5pt}x_4^{(2)}+x_4^{(1)}-x_6^{(2)} > x_2^{(2)}+x_3^{(1)}-x_4^{(4)}, \\
		& \hspace{5pt}x_4^{(2)}+x_4^{(1)}-x_6^{(2)} > x_2^{(2)}-x_3^{(3)}+x_3^{(1)}+x_4^{(3)}-x_5^{(2)}, \\
		& \hspace{5pt}x_4^{(2)}+x_4^{(1)}-x_6^{(2)} > x_3^{(2)}+x_3^{(3)}-x_5^{(2)}, \\
(\breve{F14})	& \hspace{5pt}x_6^{(1)} > x_2^{(2)}+x_2^{(1)}-x_3^{(3)}-x_3^{(2)}+x_5^{(1)}, \\
		& \hspace{5pt}x_6^{(1)} > x_2^{(2)}-x_3^{(3)}+x_3^{(1)}-x_4^{(2)}+x_5^{(1)} , \\
		& \hspace{5pt}x_6^{(1)} > x_2^{(2)}-x_3^{(3)}+x_4^{(1)}, \\
		& \hspace{5pt}x_6^{(1)} > x_3^{(2)}+x_3^{(1)}-x_4^{(3)}-x_4^{(2)}+x_5^{(1)}, \\
		& \hspace{5pt}x_6^{(1)} > x_3^{(2)}-x_4^{(3)}+x_4^{(1)}, \\
		& \hspace{5pt}x_6^{(1)} > x_4^{(2)}+x_4^{(1)}-x_6^{(2)}, \\
		& \hspace{5pt}x_6^{(1)} > x_2^{(2)}+x_2^{(1)}-x_6^{(3)}, \\
		& \hspace{5pt}x_6^{(1)} > x_2^{(2)}+x_2^{(1)}-x_4^{(4)}-x_4^{(3)}+x_6^{(2)}, \\
		& \hspace{5pt}x_6^{(1)} > x_2^{(2)}+x_2^{(1)}-x_3^{(2)}-x_4^{(4)}+x_4^{(2)}, \\
		& \hspace{5pt}x_6^{(1)} > x_2^{(2)}+x_2^{(1)}-x_3^{(3)}-x_3^{(2)}+x_4^{(3)}+x_4^{(2)}-x_5^{(2)}, \\
		& \hspace{5pt}x_6^{(1)} > x_2^{(2)}+x_3^{(1)}-x_4^{(4)}, \\
		& \hspace{5pt}x_6^{(1)} > x_2^{(2)}-x_3^{(3)}+x_3^{(1)}+x_4^{(3)}-x_5^{(2)}, \\
		& \hspace{5pt}x_6^{(1)} > x_3^{(2)}+x_3^{(3)}-x_5^{(2)}, 
\end{align*}
Then for $x \in \mathcal{X}$ we have $\f_0(x) = \mathcal{UD}(e_0^c)(x)\arrowvert_{c=-1}$ given by
\begin{align*}
\tilde{f_0}(x) = \begin{cases} 
&(x_6^{(3)}+1,x_4^{(4)}+1, x_3^{(3)}+1, x_2^{(2)}+1, x_5^{(2)}+1, x_4^{(3)}+1, x_3^{(2)}+1, x_6^{(2)}, x_4^{(2)},  \\ 
	& \hspace{15pt} x_5^{(1)},x_1^{(1)}+1, x_2^{(1)}+1, x_3^{(1)}, x_4^{(1)}, x_6^{(1)}) \ \text{if} \ (\breve{F1}),  \vspace{1pt}\\ 
&(x_6^{(3)},x_4^{(4)}+1, x_3^{(3)}+1, x_2^{(2)}+1, x_5^{(2)}+1, x_4^{(3)}+1, x_3^{(2)}+1, x_6^{(2)}+1, x_4^{(2)},  \\ 
	& \hspace{15pt}x_5^{(1)},x_1^{(1)}+1, x_2^{(1)}+1, x_3^{(1)}, x_4^{(1)}, x_6^{(1)}) \ \text{if} \ (\breve{F2}),  \vspace{1pt}\\ 
&(x_6^{(3)},x_4^{(4)}+1, x_3^{(3)}+1, x_2^{(2)}+1, x_5^{(2)}+1, x_4^{(3)}, x_3^{(2)}+1, x_6^{(2)}+1, x_4^{(2)}+1, \\ 
	& \hspace{15pt}x_5^{(1)},x_1^{(1)}+1, x_2^{(1)}+1, x_3^{(1)}, x_4^{(1)}, x_6^{(1)})  \ \text{if} \ (\breve{F3}),  \vspace{1pt}\\ 
&(x_6^{(3)},x_4^{(4)}, x_3^{(3)}+1, x_2^{(2)}+1, x_5^{(2)}+1, x_4^{(3)}+1, x_3^{(2)}+1, x_6^{(2)}+1, x_4^{(2)}+1,  \\ 
	& \hspace{15pt}x_5^{(1)},x_1^{(1)}+1, x_2^{(1)}+1, x_3^{(1)}, x_4^{(1)}, x_6^{(1)}) \ \text{if} \ (\breve{F4}),  \vspace{1pt}\\ 
&(x_6^{(3)},x_4^{(4)}+1, x_3^{(3)}+1, x_2^{(2)}+1, x_5^{(2)}+1, x_4^{(3)}, x_3^{(2)}, x_6^{(2)}+1, x_4^{(2)}+1,  \\ 
	& \hspace{15pt}x_5^{(1)},x_1^{(1)}+1, x_2^{(1)}+1, x_3^{(1)}+1, x_4^{(1)}, x_6^{(1)}) \ \text{if} \ (\breve{F5}),  \vspace{1pt}
 \end{cases}
\end{align*}
\begin{align*}
\tilde{f_0}(x) = \begin{cases} 
&(x_6^{(3)},x_4^{(4)}, x_3^{(3)}+1, x_2^{(2)}+1, x_5^{(2)}, x_4^{(3)}+1, x_3^{(2)}+1, x_6^{(2)}+1, x_4^{(2)}+1,  \\ 
	& \hspace{15pt}x_5^{(1)}+1,x_1^{(1)}+1, x_2^{(1)}+1, x_3^{(1)}, x_4^{(1)}, x_6^{(1)}) \ \text{if} \ (\breve{F6}),  \vspace{1pt}\\ 
&(x_6^{(3)},x_4^{(4)}, x_3^{(3)}+1, x_2^{(2)}+1, x_5^{(2)}+1, x_4^{(3)}+1, x_3^{(2)}, x_6^{(2)}+1, x_4^{(2)}+1, \\ 
	& \hspace{15pt}x_5^{(1)},x_1^{(1)}+1, x_2^{(1)}+1, x_3^{(1)}+1, x_4^{(1)}, x_6^{(1)})  \ \text{if} \ (\breve{F7}),  \vspace{1pt}\\ 
&(x_6^{(3)},x_4^{(4)}, x_3^{(3)}, x_2^{(2)}+1, x_5^{(2)}+1, x_4^{(3)}+1, x_3^{(2)}+1, x_6^{(2)}+1, x_4^{(2)}+1, \\ 
	& \hspace{15pt}x_5^{(1)},x_1^{(1)}+1, x_2^{(1)}+1, x_3^{(1)}+1, x_4^{(1)}, x_6^{(1)})  \ \text{if} \ (\breve{F8}),  \vspace{1pt}\\ 
&(x_6^{(3)},x_4^{(4)}, x_3^{(3)}+1, x_2^{(2)}+1, x_5^{(2)}, x_4^{(3)}+1, x_3^{(2)}, x_6^{(2)}+1, x_4^{(2)}+1,  \\ 
	& \hspace{15pt}x_5^{(1)}+1,x_1^{(1)}+1, x_2^{(1)}+1, x_3^{(1)}+1, x_4^{(1)}, x_6^{(1)}) \ \text{if} \ (\breve{F9}),  \vspace{1pt}\\ 
&(x_6^{(3)},x_4^{(4)}, x_3^{(3)}+1, x_2^{(2)}+1, x_5^{(2)}, x_4^{(3)}+1, x_3^{(2)}, x_6^{(2)}+1, x_4^{(2)},  \\ 
	& \hspace{15pt}x_5^{(1)}+1,x_1^{(1)}+1, x_2^{(1)}+1, x_3^{(1)}+1, x_4^{(1)}+1, x_6^{(1)}) \ \text{if} \ (\breve{F10}),  \vspace{1pt}\\    
&(x_6^{(3)},x_4^{(4)}, x_3^{(3)}, x_2^{(2)}+1, x_5^{(2)}, x_4^{(3)}+1, x_3^{(2)}+1, x_6^{(2)}+1, x_4^{(2)}+1, \\ 
	& \hspace{15pt}x_5^{(1)}+1,x_1^{(1)}+1, x_2^{(1)}+1, x_3^{(1)}+1, x_4^{(1)}, x_6^{(1)})  \ \text{if} \ (\breve{F11}),  \vspace{1pt}\\ 
&(x_6^{(3)},x_4^{(4)}, x_3^{(3)}, x_2^{(2)}+1, x_5^{(2)}, x_4^{(3)}+1, x_3^{(2)}+1, x_6^{(2)}+1, x_4^{(2)},  \\ 
	& \hspace{15pt}x_5^{(1)}+1,x_1^{(1)}+1, x_2^{(1)}+1, x_3^{(1)}+1, x_4^{(1)}+1, x_6^{(1)}) \ \text{if} \ (\breve{F12}),  \vspace{1pt}\\ 
&(x_6^{(3)},x_4^{(4)}, x_3^{(3)}, x_2^{(2)}+1, x_5^{(2)}, x_4^{(3)}, x_3^{(2)}+1, x_6^{(2)}+1, x_4^{(2)}+1,  \\ 
	& \hspace{15pt}x_5^{(1)}+1,x_1^{(1)}+1, x_2^{(1)}+1, x_3^{(1)}+1, x_4^{(1)}+1, x_6^{(1)}) \ \text{if} \ (\breve{F13}),  \vspace{1pt}\\ 
&(x_6^{(3)},x_4^{(4)}, x_3^{(3)}, x_2^{(2)}+1, x_5^{(2)}, x_4^{(3)}, x_3^{(2)}+1, x_6^{(2)}, x_4^{(2)}+1,  \\ 
	& \hspace{15pt}x_5^{(1)}+1,x_1^{(1)}+1, x_2^{(1)}+1, x_3^{(1)}+1, x_4^{(1)}+1, x_6^{(1)}+1) \ \text{if} \ (\breve{F14}).
 \end{cases}
\end{align*}

\begin{theorem} The map
\begin{displaymath}
\begin{array}{lccc}
\Omega : & B^{6,\infty} & \rightarrow & \mathcal{X},\\
&b=(b_{ij})_{i \leq j \leq i+5,\ 1 \leq i \leq 6} &\mapsto & x=(x_6^{(3)},x_4^{(4)}, x_3^{(3)}, x_2^{(2)}, x_5^{(2)}, x_4^{(3)}, x_3^{(2)}, x_6^{(2)}, x_4^{(2)}, \\
& & & \qquad x_5^{(1)},x_1^{(1)}, x_2^{(1)}, x_3^{(1)}, x_4^{(1)}, x_6^{(1)})
\end{array}
\end{displaymath}
defined by 
\begin{align*}
x_m^{(l)} &= \begin{cases}
\sum_{j=m-l+1}^m b_{m-l+1, j}, \ \ \text{for} \; \ \ m= 1, 2, 3, 4\\
\sum_{j=m-2l+1}^m b_{m-2l+1, j}, \ \ \text{for} \; \ \ m= 5 \\
\sum_{j=m-2l+1}^{m-1} b_{m-2l+1, j}, \ \ \text{for} \; \ \ m= 6.
\end{cases}
\end{align*}
is an isomorphism of crystals.
\end{theorem}

\begin{proof} First we observe that the map $\Omega^{-1} : \mathcal{X} \rightarrow B^{6, \infty}$ is given by $\Omega^{-1}(x)=b$ 
\begin{align*}
&\text{where}\ b_{11} = x_1^{(1)},  \ b_{12} =x_2^{(2)}-x_1^{(1)}, \ b_{13}  = x_3^{(3)}-x_2^{(2)}, \ b_{14} =x_4^{(4)}-x_3^{(3)}, \\ &b_{15}=x_6^{(3)}-x_4^{(4)}, \ b_{16} =-x_6^{(3)}, \ b_{22} = x_2^{(1)}, \ b_{23} =x_3^{(2)}-x_2^{(1)}, \ b_{24} = x_4^{(3)}-x_3^{(2)}, \\
& b_{25} = x_5^{(2)}-x_4^{(3)}, \ b_{26} = x_6^{(3)}-x_5^{(2)}, \ b_{27} = -x_6^{(3)}, \ b_{33} = x_3^{(1)}, \ b_{34} =x_4^{(2)}-x_3^{(1)}, \\
& b_{35}=x_6^{(2)}-x_4^{(2)}, \ b_{36} =x_5^{(2)}-x_6^{(2)}, \ b_{37} = x_4^{(4)}-x_5^{(2)},\ b_{38} = -x_4^{(4)}, \ b_{44} = x_4^{(1)}, \\
& b_{45} = x_5^{(1)}-x_4^{(1)}, \ b_{46} = x_6^{(2)}-x_5^{(1)},\ b_{47} =x_4^{(3)}-x_6^{(2)}, \ b_{48} =x_3^{(3)}-x_4^{(3)},  \ b_{49} =-x_3^{(3)}, \\
& b_{55} = x_6^{(1)}, \ b_{56} = x_5^{(1)}-x_6^{(1)}, \ b_{57} = x_4^{(2)}-x_5^{(1)}, \ b_{58} = x_3^{(2)}-x_4^{(2)}, \ b_{59} = x_2^{(2)}-x_3^{(2)}, \\ 
   &b_{5,10} = -x_2^{(2)},\ b_{66} = x_6^{(1)}, \ b_{67} = x_4^{(1)}-x_6^{(1)}, \ b_{68} = x_3^{(1)}-x_4^{(1)}, \ b_{69} = x_2^{(1)}-x_3^{(1)}, \\
    &b_{6,10} = x_1^{(1)}-x_2^{(1)}, \ b_{6,11} = -x_1^{(1)}.
\end{align*}
Hence the map $\Omega$ is bijective.  To prove that $\Omega$ is an isomorphism of crystals we need to show 
that for $b \in B^{6,\infty}$ and $0 \leq k \leq 6$ we have:
\begin{align*}
\Omega(\tilde{f_k} (b)) 	&= \tilde{f_k} (\Omega(b)),\\
\Omega(\tilde{e_k} (b)) 	&= \tilde{e_k} (\Omega(b)),\\
\text{wt}_k (\Omega(b))	&= \text{wt}_k(b),\\
\veps_k (\Omega(b))		&= \veps_k(b).
\end{align*}
Hence $\vphi_k(\Omega(b)) = \text{wt}_k (\Omega(b)) + \veps_k (\Omega(b)) = \text{wt}_k(b) + \veps_k(b) = \vphi_k(b)$.	We observe that the conditions for the action of $\f_k$ on $\Omega(b)$ in $\mathcal{X}$ hold if and only if the corresponding conditions for the action of $\f_k$  on $b$ in $B^{6, \infty}$ hold for all $0\leq k \leq 6$. Suppose $\Omega(b) = x$ and $x_2^{(2)} + x_2^{(1)} > x_1^{(1)} + x_3^{(2)}$, then
$b_{11} + b_{12} + b_{22} > b_{11} + b_{22} + b_{23}$ and $\f_2(x) = (x_6^{(3)},x_4^{(4)}, x_3^{(3)}, x_2^{(2)}-1, x_5^{(2)}, x_4^{(3)}, x_3^{(2)}, x_6^{(2)}, x_4^{(2)}, x_5^{(1)},x_1^{(1)}, x_2^{(1)}, x_3^{(1)}, x_4^{(1)}, x_6^{(1)}) = \\\Omega(\f_2(b))$. Similarly, we can show $\Omega(\tilde{f_k} (b)) = \tilde{f_k} (\Omega(b))$  and $\Omega(\tilde{e_k} (b)) = \tilde{e_k} (\Omega(b))$ for $k=0,1,3,4,5,6$. We also have
$\text{wt}_0(\Omega(b)) = \text{wt}_0(x) = -x_2^{(2)} - x_2^{(1)} = - b_{11} - b_{12} - b_{22} = - b_{11} - b_{12} + b_{23} + b_{24} + b_{25} + b_{26} + b_{27} = \text{wt}_0(b)$ for all $b \in B^{6,\infty}$. Similarly, $\text{wt}_k (\Omega(b)) = \text{wt}_k(b)$ for $1 \leq k \leq 6$. Also, $\veps_6 (\Omega(b)) = \veps_6(x) = \text{max} \{-x_6^{(3)}, x_4^{(4)}+x_4^{(3)}-2x_6^{(3)}-x_6^{(2)}, x_4^{(4)}+x_4^{(3)}+x_4^{(2)}+x_4^{(1)}-2x_6^{(3)}-2x_6^{(2)}-x_6^{(1)}\} = \text{max}  \{-b_{11}-b_{12}-b_{13}-b_{14}-b_{15}, -b_{11}-b_{12}-b_{13}-b_{14}-2b_{15}+ b_{22}+b_{23}+b_{24}-b_{33}-b_{34}-b_{35}, -b_{11}-b_{12}-b_{13}-b_{14}-2b_{15}+b_{22}+b_{23}+b_{24}-b_{33}-b_{34}-2b_{35}+b_{44}-b_{55}\}= \veps_6(b)$. Similarly, 
$\veps_k (\Omega(b)) = \veps_k(b)$ for $0 \leq k \leq 5$ which completes the proof.
\end{proof}

\bibliographystyle{amsalpha}

\end{document}